\documentclass[12pt]{article}
\usepackage{amsmath,amssymb,amsthm,hyperref,graphicx,natbib,geometry,authblk}
\usepackage{tikz}

\newcommand{\Exp}[2][]{{\rm E}_{#1}[ #2 ]  }

\newcommand{\bl}[1]{{\mathbf #1}}
\newcommand{\bs}[1]{\boldsymbol #1}

\newtheorem{theorem}{Theorem}
\newtheorem{corollary}{Corollary}
\newtheorem{definition}{Definition}
\newtheorem{lemma}{Lemma}

\def\bbR{\mathbb{R}}
\def\bbU{\mathbb{U}}

\def\diag{\mathrm{diag}}
\def\x{{\bs x}}
\def\X{{\bs X}}

\def\psel{p_{\text{sel}}}
\def\pcon{p_{\text{con}}}
\def\Fcon{F_{\!\text{con}}}

\definecolor{Rcol1}{rgb}{0,0,0}
\definecolor{Rcol2}{rgb}{.87,.32,.42}
\definecolor{Rcol3}{rgb}{.38,.81,.31}
\definecolor{Rcol4}{rgb}{.13,.59,.90}
\definecolor{Rcol5}{rgb}{.16,.88,.89}
\definecolor{Rcol6}{rgb}{.80,.04,.73}
\definecolor{Rcol7}{rgb}{.96,.78,.06}

\newcommand{\colorline}[1]{\raisebox{2pt}{\tikz{\draw[-,#1,solid,line width = 0.75pt](0,0) -- (5mm,0);}}}

\begin{document}

\title{Selective and marginal selective inference for exceptional groups}
\author{Peter Hoff and Surya Tokdar}
\affil{Department of Statistical Science, Duke University} 
\date{\today}
\maketitle

\begin{abstract} 
Statistical analyses of multipopulation studies
often  use the data to select 
a particular population as the target of inference. 
For example, a confidence interval may be constructed for a population only in the 
event that its sample mean is larger than that of the other populations. 
We show that for the normal means model, confidence interval procedures that maintain strict coverage control conditional on such a selection event will 
have infinite expected width.
For applications where  such selective 
coverage control is of interest, this result motivates the development of 
procedures with finite expected width
 and  approximate selective coverage control
over a range of plausible parameter values. 
To this end, we develop selection-adjusted empirical Bayes confidence procedures
that use information from the data to approximate an oracle confidence procedure
that has exact selective coverage control and finite expected width. In 
numerical comparisons of the oracle and empirical Bayes procedures to procedures that only 
guarantee selective coverage control marginally over selection events, 
we find that improved selective coverage control comes at the cost of increased expected interval width. 


\end{abstract}

\section{Introduction} 
\label{sec:intro}
A common practice in multipopulation data analysis is to report estimates and inferences for one or more data-selected populations, treatments or groups. 
For example, 
an estimate and a confidence interval for the mean of a group 
might be reported only in the event that its sample mean is larger than that of the 
other groups. 
However, it is well known that standard estimates 
that do not adjust for the selection process 
will be biased for the mean of the selected group. More profoundly,
 the actual coverage rates of unadjusted confidence intervals 
 will not match their nominal rates \citep{dawid1994selection}. 

Referring to the selection bias as the 
``winner's curse,'' \citet{efron2011tweedie}
studied empirical Bayes techniques for bias correction. 
For the simple multiple normal means model,
\citet{andrews2024inference} 
and 
\citet{zrnic_fithian_2024} 
have recently considered the problem of confidence interval construction for the top group or ``winner'', that is, the 
group with the largest observed sample mean. 
 \citet{andrews2024inference} 
 developed and evaluated several procedures, one 
of which guarantees exact, constant 
frequentist coverage that holds conditionally on (and hence marginally over) the selection event. However, this procedure was noticed to have unreasonably large expected width. As an alternative, these authors developed a procedure with a smaller expected width, but at the cost of having approximate coverage that only holds marginally over a selection process. 

The distinction between marginal and conditional properties 
also arises in general non-selective multipopulation data analysis.
For example, shrinkage estimators of group means 
 obtained from a random-effects model 
are sometimes referred to as the best linear unbiased predictors 
(BLUPs).
While the average bias of such estimators across groups is generally close to zero, the bias 
for any particular group is not, and so as estimators of 
the individual values of the group means, the BLUPs are biased 
\citep[Section 4.8]{snijders2011multilevel}.
Similarly, 
while the coverage rates of so-called $1-\alpha$ 
prediction intervals for individual group means  may be approximately 
$1-\alpha$ on-average across groups, the coverage for any particular group 
will depend on the group's mean, and could be much lower than $1-\alpha$ if 
the mean is far from those of the  other groups
(\citet[Section 4.8]{snijders2011multilevel},\citet{yu2018adaptive}). 
Generally speaking, statistical properties that hold marginally - on-average over groups or parameter values - may not hold conditionally for specific groups, or for specific parameter values.

If coverage control is only required on-average 
across groups, selection events or parameter values, then intervals with only marginal control
are to be preferred to those 
with selective control, as the former are generally narrower than the latter. 
However, there are scenarios where selective coverage may be of interest. Consider a study designed to identify an 
underperforming school based on test scores of a sample of students from each school in a 
county. 
While the county superintendent may only be concerned with the marginal 
coverage rate of the interval on-average over which school is identified, 
the staff of a given school would likely be more concerned with the coverage rate 
in the specific case that 
their school is selected, that is, the selective coverage rate. 
A slightly different scenario is 
where a confidence interval for the effect of a new
or previously unremarkable treatment is constructed only in the event that it 
outperforms an established collection of treatments. 
In this case, the only treatment for which an interval will be constructed 
is the ``underdog'' treatment, 
and so there are no 
other selection events to average over and hence no relevant notion of marginal coverage over different selection 
events. 

In this article, we study the 
coverage and precision of confidence interval procedures 
for a normal population mean, 
conditional on this population yielding  a larger sample response than
those of 
several other normal populations.  
In the ``underdog'' scenario described above, confidence 
procedures may be constructed and evaluated using 
the following two nested types of probability, which
we define precisely in the next section:
\begin{description}
\item{Conditional:} conditional on the selection event and the data from unselected groups;
\item{Selective:} conditional on the selection event. 
\end{description} 
In the ``winners'' scenario where multiple selection events are possible, 
procedures may additionally be evaluated in terms of a third type of probability:
\begin{description}
\item{Marginal:} marginal over different selection events. 
\end{description} 
Coverage control at one level of the hierarchy implies control at higher levels: A procedure with $1-\alpha$ conditional coverage has $1-\alpha$ selective coverage, 
and  a procedure with $1-\alpha$ selective coverage has $1-\alpha$ marginal coverage. 
\citet{andrews2024inference} provided a confidence procedure 
with  constant conditional coverage which, in a simulation study, 
produced very wide intervals that
 the authors conjectured had infinite expected width. 
This result suggests looking for procedures with more coarse-grained 
coverage control, such as a procedure with constant selective coverage.
In Section 2 of this article, we show that, unfortunately, 
any procedure that maintains exact $1-\alpha$ selective coverage
must also have exact $1-\alpha$ conditional coverage, suggesting that 
any procedure that maintains constant selective coverage 
will have infinite expected width. For the normal means model, 
we prove that this is indeed the case.

Foreshadowing these negative results,  \citet{andrews2024inference} 
developed a second confidence procedure that has 
finite expected width and 
marginal, but not selective, coverage rate control. 
%{\sc ELABORATE: Recently, \citet{zrnic_fithian_2024} also developed 
%other procedures with marginal coverage rate control}. 
However, for applications where selective coverage is of interest, 
it may be preferable to use a procedure that has approximate
 selective coverage control over a range of plausible parameter values, if not 
over the entire parameter space. 
To this end, in Section 3 we introduce an 
oracle selective confidence interval that, given (unavailable) knowledge of 
the means of the non-selected groups, maintains exact $1-\alpha$ coverage 
and has finite expected width. We then illustrate in a simple two-group
case that, 
given accurate prior information about the non-selected group, 
a selection-adjusted Bayes interval may be constructed that mimics the 
performance of the oracle procedure.
%, \red{which is otherwise difficult to 
%reproduce without such accurate prior information}.   

Absent prior information, it seems possible that
in the case of multiple 
non-selected groups, knowledge of the means of the non-selected groups 
may be estimated from the data, then used to construct 
a selection-adjusted empirical Bayes procedure that 
approximates the oracle procedure. 
While our results from Section 2
rule out the possibility of 
global coverage control 
 without infinite expected interval width, 
in Section 4 we illustrate numerically that selection-adjusted empirical Bayes procedures 
can locally approximate the oracle procedure to some degree, by maintaining comparable 
expected widths and a useful 
degree of selective coverage over a range of parameter values. 
However, in comparison to procedures with only marginal coverage guarantees, we find that 
improved selective coverage control comes at the cost of increased interval width. 
A discussion follows in Section 5. 
Replication 
code for the numerical results in this article is available at the 
first author's 
website.

\section{Implications of conditional coverage control} 
\subsection{A hierarchy of coverage rates} 
\label{sec:coverate-types}
A simple but widely applicable model for studying and developing 
multipopulation inference procedures is the multiple normal means model, where 
scalar observations $Z_1,\ldots, Z_{p+1}$ are independently sampled from 
$p+1$ potentially different normal populations, so that 
$Z_j \sim N(\mu_j,\psi^2_j)$ independently for $j=1,\ldots,p+1$, 
with $\mu_1,\ldots, \mu_{p+1}$ being
unknown and $\psi^2_1,\ldots,\psi^2_{p+1}$ 
(approximately) known. This scenario might arise if the elements of 
 $\bs Z = (Z_1,\ldots, Z_{p+1})$ 
are sample averages from $p+1$ populations with means equal to the 
corresponding elements of 
$\bs\mu = ( \mu_1,\ldots,\mu_{p+1})$, and a common population variance $\psi^2$ 
which could be 
precisely estimated by pooling data across the groups. 
Letting $n_j$ be the sample size for population $j$, 
the variance of $Z_j$ would be $\psi^2_j = \psi^2/n_j$. 

We are interested in the 
selective coverage rates and expected widths of confidence interval procedures 
for the population mean 
of the group  having the largest observation.  
The selective properties we study arise from two slightly different 
inferential scenarios.  
In the first, which we refer to as ``inference on underdogs'',
$(Z_1,\ldots,Z_p)$ are the observations from an established set of $p$ groups and 
$Z_{p+1}$ is the observation 
from an unknown or previously unremarkable ``underdog'' group. 
Upon the remarkable 
event that $Z_{p+1}> \max\{ Z_1,\ldots, Z_p\}$, we construct a confidence 
interval for $\mu_{p+1}$. 
The second scenario 
 is that of making  ``inference on winners'' \citep{andrews2024inference}. 
In this case, $Z_1,\ldots, Z_{p+1}$ are independently sampled and 
a confidence interval is constructed for the population mean of the group 
having the largest observation. 

Let $S = \arg \max_j \{ Z_j: j=1,\ldots p+1\}$  be the index of the group with the
largest observation, so that 
$S=p+1$ in the ``underdog'' scenario and $S$ is a random 
variable in the ``winners'' scenario. 
In both scenarios the goal is to make inference on the mean $\mu_S$ of the selected group  based on the
data $\bs Z_{-S}$ from the unselected groups and  $Z_S$ from the selected group. 
For notational simplicity in what follows, we write $Y=Z_S$, $\theta=\mu_S$
 and $\sigma = \psi_S$
 as 
the outcome, mean and standard deviation of the selected group, and 
write $\bs X=(X_1,\ldots,X_p)= \bs Z_{-S}$, 
 $\bs \eta = 
 (\eta_1,\ldots, \eta_p)= \bs \mu_{-S}$ and  $\bs \tau = (\tau_1,\ldots, \tau_p ) = \bs\psi_{-S}$ as the outcomes, means and standard deviations of the unselected groups.
A confidence procedure $C$ is any (appropriately measurable) set-valued function $C: \mathbb R^{p+1} \rightarrow 2^\mathbb R$. 
The coverage rate of $C$ is defined as
the probability of the event 
$\mu_{S} \in C(\bs Z_{-S},Z_S)$, or equivalently, 
$\theta \in C(\bs X,Y)$, where the probability could be 
one of three types:
\begin{description}
\item{Marginal:}  $\Pr( \mu_S \in C(\bs Z_{-S},Z_S) | \bs\mu )$, the  marginal coverage rate; 
\item{Selective:}  $\Pr( \theta \in C(\bs X,Y) | S=s, \bs\eta,\theta )$, the selective coverage rate;
\item{Conditional:} $\Pr(\theta \in C(\bs x,Y ) | \bs X=\bs x, S=s, \bs\eta,\theta )$,
  the conditional coverage rate, 
\end{description}
where $s\in \{1,\ldots, p+1\}$ and $\bs x\in \mathbb R^p$. 
The marginal coverage rate is only relevant for the ``winners'' scenario, and 
is obtained by averaging 
over different selection events, and hence different correspondences between 
the elements of $\bs Z$ and $(\bs X,Y)$ and between $\bs\mu $ and $(\bs\eta,\theta)$. 
In the ``underdog'' scenario, or conditional upon any particular selection event 
$S=s$ in the ``winners'' scenario, 
 these correspondences are fixed, 
and so when considering selective coverage 
we drop the $s$ in the notation 
and write
selective and conditional coverage as $\Pr( \theta\in C(\bs X,Y) | \bs X\prec Y, \bs\eta,\theta)$ 
and $\Pr( \theta\in C(\bs x,Y) | \bs x\prec Y, \bs\eta,\theta)$  respectively, where 
$\bs x \prec y$ means that $y$ is larger than every element of $\bs x$. 
These coverage  probabilities are computed from the ``selective'' distribution of 
$(\bs X, Y)$ given $\bs X\prec Y$  and the conditional distribution 
of $Y$ given $\bs X=\bs x,\bs x \prec Y$ respectively, having 
densities $\psel$ and $\pcon$  given by
\begin{align}
\psel(\bs x,y | \bs\eta,\theta) 
 & =\frac{f(y| \theta,\sigma) \prod_{j} f(x_j|\eta_j,\tau_j)}
 {\int f(y|\theta,\sigma) \prod_{j} F(y|\mu_j,\tau_j)\, dy}\times 1_{[\bs x\prec y]}  
\label{eqn:dselective} 
\\ 
 \pcon( y | \bs x,\theta)
 &=\frac{  f(y| \theta,\sigma)}{ 1-F(x|\theta,\sigma) }\times 1_{[x<y]} 
\label{eqn:dconditional}  
\end{align}
where $x= \max\{ x_1,\ldots, x_p\}$, and  $f(\cdot|\mu,\psi)$ and $F(\cdot|\mu,\psi)$ are the density and cumulative distribution function (CDF) of the normal distribution  with mean $\mu$ and standard deviation $\psi$.

\subsection{Equivalence of constant conditional coverage procedures} 
\label{sec:equivalence}
For many applications it would seem desirable to use a confidence procedure $C$ with constant $1-\alpha$ selective coverage, so that 
$\Pr( \theta \in C(\bs X,Y) | \bs X\prec Y, \bs\eta,\theta ) =1-\alpha$ for 
all $(\bs\eta ,\theta) \in \mathbb R^{p+1}$.
In the ``winners'' scenario  where one group is selected from among several based on the sampled outcomes,
such a procedure seems ``fair'' in that is has the same coverage 
rate  regardless of which population mean among $(\mu_1,\ldots, \mu_{p+1})$ corresponds to $\theta$. 
In the ``underdog'' scenario  where the groups associated with $\bs X$ and $Y$ 
are fixed in advance and 
data are analyzed only upon
the occurrence of $\bs X\prec Y$,
such an interval covers  $\theta$ with probability $1-\alpha$, 
regardless of  the value of  $\theta$ or the values of the population means $\bs\eta$ of the 
other groups.

We now review the construction in
 \citet{andrews2024inference} of such a confidence 
procedure.
Any procedure with constant $1-\alpha$ selective coverage 
can be written as the inversion 
of the acceptance regions $\{ A (\theta_0) : \theta_0\in \mathbb R\}$ 
of a collection 
of level-$\alpha$ tests of $H: \theta = \theta_0$, for 
each $\theta_0\in \mathbb R$. 
For the selective coverage of $C$ to be $1-\alpha$ under the selection event  we must have 
\[
\Pr ( (\bs X,Y) \in A(\theta_0) | \bs X \prec Y,\bs\eta,\theta_0) = 
1-\alpha 
\]
for all $\bs\eta \in \mathbb R^{p}$ and $\theta_0\in \mathbb R$. 

Finding such a set $A(\theta_0)$ is challenging as it must have the 
same probability for all values of $\bs\eta$.
To circumvent this issue, 
 \citet{andrews2024inference} construct a 
confidence interval (or equivalently, a collection of level-$\alpha$ 
tests) using the conditional distribution of $Y$ given  
$x<Y $ where $x$ is the observed value of $X=\max\{ X_1,\ldots, X_p\}$. 
 \citet{andrews2024inference} study a quantile-based confidence 
interval for $\theta$ that is the inversion of 
acceptance regions $A$ of the form 
$A(\theta_0) = ( l(\theta_0),u(\theta_0))$ 
where $l(\theta_0)$ and $u(\theta_0)$ are the lower and upper $\alpha/2$ 
quantiles of the distribution with density $\pcon$ given by (\ref{eqn:dconditional}). 
The lower and upper endpoints $\theta_l<\theta_u$ of the resulting interval 
can be obtained as solutions to the  equations
$1-\Fcon(y | \theta_l)  =\alpha/2$ and 
$\Fcon(y | \theta_u)  = \alpha/2$, where $\Fcon$ is the CDF of the distribution with density $\pcon$. 
This interval has conditional coverage exactly 
equal to $1-\alpha$  for all values 
of $\theta$ and $x$. Since selective coverage 
is simply the expectation of conditional coverage over possible 
values of $X$, this interval also has selective coverage 
exactly equal to $1-\alpha$ for all values of $(\bs\eta,\theta)$. 

 \citet{andrews2024inference} noticed that numerically
this interval is very wide when 
$y$ is very close to $x$. Based on simulation studies, they 
speculated that the selective expected width of this interval 
is infinite (indeed it is, as we discuss in the next subsection). 
This undesirable performance leads one to wonder whether or not 
there exist narrower confidence interval procedures with constant 
$1-\alpha$ selective coverage. Intuitively, maintaining $1-\alpha$ 
conditional coverage for all possible values of $\bs X$ 
is a strong 
restriction, and so perhaps stronger than necessary to maintain 
$1-\alpha$ selective coverage for all $(\bs\eta,\theta)$ values. 
It turns out that this is not so - the next result shows that 
\emph{any} confidence interval procedure with constant 
$1-\alpha$ selective coverage for all $(\bs\eta,\theta)$   must also have constant 
$1-\alpha$ conditional coverage for all values $\bs x$ of $\bs X$. 

\begin{theorem} 
\label{thm:css}
\sloppy Let $C:\mathbb R^{p+1} \rightarrow 2^\mathbb R$ be a set-valued function such that 
$\Pr( \theta \in C(\bs X,Y) | \bs X \prec Y,\bs \eta, \bs\theta ) = 1-\alpha$ for  
all $(\bs\eta,\theta)\in \mathbb R^{p+1}$. 
Then $C$ satisfies 
\[ 
\Pr( \theta \in C(\bs x,Y) \mid \bs x\prec Y,\theta)  = 1-\alpha  
\]
for all $\theta\in \mathbb R$ and almost all 
    $ \bs x \in \mathbb R^{p}$. 
\end{theorem}
\begin{proof}  
    The assumption on $C$ implies 
$\Pr( \theta \not\in C(\bs X,Y) |  \bs X\prec Y , \bs\eta,\theta ) = \alpha$, or 
equivalently, 
$ \Pr(\theta \not\in C(\bs X,Y), \bs X\prec Y \mid \bs \eta,\theta) - \alpha \Pr(\bs X\prec Y \mid \bs\eta,\theta) =0$. 
Let $G(\cdot | \bs\eta)$ be the probability measure of a Gaussian 
random vector with  mean $\bs\eta$ and variance $\diag(\tau^2_1,\ldots, \tau^2_p)$.  
Conditioning on $\bs X=\bs x$, the coverage condition becomes 
    \begin{align}
        0 & = \Pr(\theta \not\in C(\bs X,Y), \bs X\prec Y \mid \bs \eta,\theta) - \alpha \Pr(\bs X\prec Y \mid \bs\eta,\theta) \nonumber\\
        & = \int \left[\Pr(\theta \not\in C(\bs x,Y), \bs x \prec Y \mid \theta) - \alpha \Pr( \bs x \prec Y  \mid \theta) \right] \,  dG(\bs x  \mid \bs\eta) \label{eq:cond}
    \end{align}
    for every $(\bs\eta,\theta)$. Fix an arbitrary $\theta$. Since $\bs X$ is a complete sufficient statistic for the statistical model $\bs X \sim G(\cdot \mid \bs\eta)$, $ \bs\eta \in \bbR^{p}$, it follows that the integrand in \eqref{eq:cond} must be zero for almost every $\bs x$ under $G(\cdot|\bs\eta)$ for every $\bs\eta$. 
\end{proof}

This result implies that any confidence region with constant $1-\alpha$ 
selective 
coverage also has constant $1-\alpha$ conditional coverage, 
and can therefore be represented for each selection event
as the inversion of a collection of size-$\alpha$ tests of values of $\theta$ based on observation of 
$Y = y$ where $Y$ follows a truncated $N(\theta,\sigma^2)$ distribution, constrained to 
be above $x=\max \{ x_1,\ldots, x_p\}$. Note that the tests may also depend on the observed values $\bs x$  of all of the  elements of $\bs X$. 

\subsection{Expected widths and conditional coverage control}  
\label{sec:theory}

As shown above, any interval with constant selective coverage also has constant 
conditional coverage, which, based on the observation of 
\citet{andrews2024inference}, suggests undesirably high 
expected interval widths, where the expectation is ``selective'', that is, 
conditional on the selection event but on-average with respect to $\bs X$ and $Y$:
\begin{definition}
For any Lebesgue measurable $A \subset \bbR$, its width $|A|$ is 
 its Lebesgue measure $\int_A dx$. 
%Let $\lambda$ denote the Lebesgue measure on the real line. 
The selective expected width of a confidence procedure $C$ is  
the expected value of $|C|$ conditional on the  event $\bs X\prec Y$. 
\end{definition} 
First  we show that any procedure with conditional coverage control  will have infinite selective 
expected width. It follows that marginal expected width, obtained by averaging over different selection events, will 
be infinite as well. 
Our result applies to the case where the conditional coverage is constant at the nominal level, as well as 
 the case where the conditional coverage is not constant but never falls  below the nominal level.
\begin{theorem}
\label{thm:inf1}
Let $C$ be a confidence procedure with conditional coverage control, so that $\Pr(\theta \in C(\bs x, Y)| \bs x\prec Y ,\theta) \ge 1 - \alpha$ for all $\theta\in \mathbb R$ and almost surely in $\bs x$ for every $\bs \eta\in \mathbb R^p$. Then $C$ has an infinite selective expected width.
\end{theorem}

\begin{proof}
Define $X$ to be the largest element of $\bs X$.  Let $F^*(\cdot|\bs\eta,\theta)$ denote the conditional distribution of $X$ given $\bs X\prec Y$. 
The conditional distribution of $Y$ given $[\bs X \prec Y, \bs X = \bs x]$ is simply $N(\theta,\sigma^2)$ truncated to the interval $(x,\infty)$ with $x$ being the largest element of $\bs x$. 
 Denote the corresponding probability measure as $P_{\theta}(\cdot|x)$. Without loss of generality let $\sigma = 1$.

For any fixed $\bs x$ with maximum element $x$, we may view the map $y \mapsto C(\bs x,y)$ as the inversion of a collection of acceptance regions of level-$\alpha$ conditional tests of $H:\theta = \theta_0$ under the model $Y \sim P_{\theta}(\cdot|x)$, $\theta \in \bbR$. Specifically, let $A(\theta_0,x) = \{y > x: \theta_0 \in C(\bs x,  y)\}$ for each $\theta_0 \in \bbR$. The conditional coverage control assumption on $C$ implies
\[
P_{\theta}(A(\theta,x)|x) = \Pr(\theta \in C(\bs x,Y) \mid \bs x  \prec Y, \theta) \ge 1 - \alpha
\]
for all $\theta \in \bbR$. By the Ghosh-Pratt identity \citep{ghosh_1961,pratt_1961}, the conditional expected width of $C$ may be related to the average type II error rate of the corresponding tests:
\[
E[|C|(\bs X,Y) |  \bs X\prec Y, \bs X = \bs x, \bs \eta,\theta] = \int_{\bbR} P_{\theta}(A(\theta_0,x)|x)\, d\theta_0.
\]
We will show that the last integral is infinite whenever $x \ge \theta + \Delta$ for some fixed positive number $\Delta$. This would immediately give $E[|C|(\bs X,Y)|\bs X\prec Y,\bs\eta,\theta] = \infty$ because $X$ admits a positive density on all of $\bbR$ for every $(\bs\eta,\theta)$.

Fix $\theta$ and let $A^*(\theta_0,x)$ be the acceptance region of the most powerful level-$\alpha$ test of 
\[
H: Y \sim P_{\theta_0}(\cdot|x)~\text{vs}~K: Y \sim P_{\theta}(\cdot|x).
\]
Then $P_{\theta}(A(\theta_0,x)|x) \ge P_{\theta}(A^*(\theta_0,x)|x)$ since the most powerful test must have a smaller type II error rate than the level-$\alpha$ test of $H$ with acceptance region $A(\theta_0,x)$. Thus, it suffices to show $\int_{\bbR} P_{\theta}(A^*(\theta_0,x)|x)d\theta_0 = \infty$ whenever $x \ge \theta + \Delta$.

Let $x + z(x)$ be the $(1-\alpha)$ quantile of a standard normal distribution truncated to $(x,\infty)$, i.e., $\Phi(-x-z(x)) = \alpha \Phi(-x)$, where $\Phi$ denotes the standard normal distribution function. It is easy to see that if $x > \theta > \theta_0$, then $A^*(\theta_0,x) = (x,x + z(x-\theta_0)]$. 
Clearly $z(x) \downarrow 0$ as $x \to \infty$.
  However, $z(x)$ cannot vanish too rapidly. 
In fact, $z(x) \ge \frac{1-\alpha}{4x}$ for all $x > 1$, because 
\begin{align*} P_0([x, x+ \tfrac{1-\alpha}{4x}] \mid x) & 
 = \tfrac{ \Phi(x + (1-\alpha)/(4x)) - \Phi(x)}{1-\Phi(x)} \\
 & \leq \tfrac{1-\alpha}{4x} \times \tfrac{\phi(x)}{1 - \Phi(x)} \\
& \leq \tfrac{1 - \alpha}{4} \times (1 + \tfrac1{x^2}) \\ & 
  \le \tfrac{1-\alpha}{2}   < 1-\alpha = P_0([x, x+z(x)] \mid x),  
\end{align*}
where the first inequality results from the concavity of $\Phi$ on $x>0$,
the second from a Mill's ratio inequality (Appendix \ref{app:A}), and the final equality comes from the definition of $z(x)$. 
Fix $\Delta > 1$ such that $z(x) \le 1$ for all $x \ge \Delta$. For any $x \ge \theta + \Delta$, if $\theta_0 < \theta$ then $z(x-\theta_0)<1$ and 
\begin{align*}
P_0(A^*(\theta_0,x) |  x) & = \frac{\Phi( x+z(x-\theta_0)) - \Phi(x) } {1-\Phi(x)} \\
&  \ge \frac{\phi(x+1)}{1 - \Phi(x)} \times z(x-\theta_0) \ge  \frac{(1-\alpha)\phi(x+1)}{4(1 - \Phi(x))} \times \frac{1}{x-\theta_0}, 
\end{align*}
where the first inequality follows from the mean value theorem and because 
$\phi(x)$ is decreasing on $x>0$. 
We then have 
\begin{align*}
 \int_{-\infty}^{\theta} P_{\theta}(A^*(\theta_0,x) | x ) d\theta_0
 &  \geq     \frac{(1-\alpha)\phi(x+1)}{4(1 - \Phi(x))} \times \int_{-\infty}^{\theta} \frac{1}{x-\theta_0}\, d\theta_0 =\infty 
\end{align*} 
for every $x \ge \theta + \Delta$, thus completing the proof.
\end{proof}

Combining Theorems \ref{thm:css} and \ref{thm:inf1} we 
immediately conclude the following:
\begin{corollary}
Let $C$ be a confidence procedure with constant selective coverage,  $\Pr(\theta \in C(\bs X,Y)|\bs X\prec Y,\bs\eta,\theta) = 1 - \alpha$ for
every $(\bs\eta ,\theta)\in \bbR^{p+1}$. Then $C$ has infinite selective expected width
for every $(\bs\eta,\theta) \in \bbR^{p+1}$.
\label{cor:inf}
\end{corollary}

This suggests that 
precise intervals that maintain $1-\alpha$ constant selective coverage 
are out of reach. However, the result in Theorem \ref{thm:css} 
relies on the fact that $\bs X$ is a complete sufficient statistic 
for the normal model, and so  
$\Exp{ f( \bs X)  | \bs\eta } = c$ for all $\bs\eta \in \mathbb R^{p}$ implies $f(\bs x) = c$ almost surely. 
But complete sufficiency does not  mean that 
$\Exp{ f( \bs X)  | \bs\eta } \geq c$  for all $\bs\eta$ 
implies $f(\bs x) \geq c$ for all $\bs x$. Therefore, Theorem \ref{thm:css} 
does not rule out the possibility that there exist 
procedures $C$ with selective coverage $\Pr( \theta \in C(\bs X,Y) | \bs X\prec Y,\bs \eta,\theta) \geq 1-\alpha$ for all $(\bs \eta,\theta)$, but where 
$\Pr( \theta \in C(\bs x,Y) | \bs x \prec Y,\theta) < 1-\alpha$ for 
a non-negligible set of 
 values of $\bs x$. 

Such a procedure, by not maintaining $1-\alpha$ coverage conditionally, 
could perhaps be narrower than the quantile-based procedure of 
\citet{andrews2024inference}, and provide a 
viable and precise selective confidence interval that maintains 
selective coverage at or above $1-\alpha$ for all  $(\bs\eta,\theta)\in \mathbb R^{p+1}$. 
Our next result suggests that such a procedure is not available, at least not 
among procedures that are location equivariant.

\begin{definition}
A confidence procedure $C(\bs x,y)$ is location equivariant if 
$C( \bs x + \bl 1 d,y+d ) = \{ \theta + d : \theta\in C(\bs x,y)\}$ 
for all $d\in \mathbb R$. 
\end{definition} 
Simply put, a confidence procedure is location equivariant if 
 $\theta\in C(\bs x,y)$ if and only if $\theta+d \in C(\bs x+ \bl 1 d ,y+d)$. 
For example, the 
procedures developed in 
\citet{andrews2024inference} described in the previous subsection 
are location equivariant. We are able to show that for the special case of two groups with equal group variances, any location equivariant procedure with selective coverage control has infinite expected width for some $ (\eta,\theta)\in \mathbb R^2$.

\begin{theorem}
\label{thm:inf2}
In the case of two groups $(p=1)$ and
 $\tau = \sigma$ let $C$ be a location equivariant confidence procedure such that $\Pr( \theta \in C(X,Y) | X<Y,\eta,\theta) \ge 1-\alpha$ for all $(\eta,\theta)\in \mathbb R^2$. Then the selective expected width is infinite for all values of $(\eta,\theta)\in \mathbb R^2$ on the diagonal line, i.e., for any $(\eta,\theta) = (c,c)$, $c \in \bbR$. 
\end{theorem}

A proof is presented in Appendix \ref{app:B}. 
Based on the proof of the theorem we have no reason to doubt that this 
result also holds for $p>1$ and 
in the heteroscedastic case where the the variances are not identical, 
but proving this appears to be quite tedious. We note here that 
location equivariant procedures can have non-constant selective 
coverage that changes with $(\bs\eta,\theta)$. Therefore, Theorem
\ref{thm:inf2} is indeed a distinct result relative to Corollary \ref{cor:inf}.

\relax

\section{Oracle and adaptive selective intervals}

\subsection{An oracle selective confidence interval} 

The results of the previous section suggest that location equivariant 
procedures that maintain a selective  coverage rate
 above some threshold have 
the undesirable property of infinite expected width. We also 
suspect that this holds more generally for any non-equivariant 
procedure that maintains a selective coverage rate. 
Therefore, it seems that the alternative to procedures with infinite 
expected width are procedures whose selective coverage 
$\Pr( \theta\in C(\bs X,Y) | \bs X\prec Y ,\bs\eta,\theta)$ 
could be arbitrarily 
small as a function of  $(\bs\eta,\theta)$.
However, this does not preclude 
the existence of finite  expected width procedures that maintain approximate selective 
coverage 
over a  relevant subset of $(\bs\eta,\theta)$-values.
Procedures with good performance over a wide range of parameter values
can often be constructed using Bayesian 
methods. 
In this section we develop a selection-adjusted 
Bayes procedure as an approximation to an ``oracle'' procedure that 
has exact error rate control. This Bayes procedure provides the foundation for the empirical Bayes procedures studied in the next section.

First consider trying to construct a procedure with 
selective coverage control in the case that $\bs \eta$ were known. 
In this case, the model for $Y$ 
conditional on $\bs\X\prec Y$ is a one-parameter exponential family model, 
and construction of  a confidence interval 
with exact $1-\alpha$  coverage for all $\theta\in \mathbb R$ and this 
specific  $\bs\eta$ is straightforward. Specifically, 
elaborating on  (\ref{eqn:dselective})
the joint density of $(\bs X,Y)$ conditional on $\bs X\prec Y$ 
is
\begin{equation}
\psel(\bs x, y \mid \bs\eta,\theta) =  \frac{\frac1\sigma\phi(\frac{y - \theta}{\sigma}) \prod_j \frac1{\tau_j}\phi(\frac{x_j - \eta_j}{\tau_j})}
{ c(\bs\eta ,\theta) }
\times 1_{[\bs x\prec y]}
\label{eqn:jcdensity}
\end{equation}
where the denominator is the probability of $\bs X\prec Y$ 
under no selection.  
We refer to the probability distribution having this density as 
$P_{\bs\eta,\theta}$, and the model 
for $(\bs X,Y)$ given $\bs X\prec Y$ as $\mathcal P = \{ P_{\bs\eta,\theta}: ( \bs\eta,\theta) \in \mathbb R^{p+1}\}$. 
We can rewrite the density (\ref{eqn:jcdensity}) as 
\begin{equation}
\psel(\bs x, y \mid \bs\eta,\theta)  =  
\left (   \frac{ \tfrac{1}{\sigma}  \phi(\tfrac{y - \theta}{\sigma})  \prod_{j} \Phi(\tfrac{y-\eta_j}{\tau_j})   }{ c(\bs\eta,\theta) }
 \right )   \times  \left (  1_{[\bs x\prec y]}   \times 
 \prod_{j=1}^p \frac{ \tfrac{1}{\tau_j} \phi(\tfrac{x_j - \eta_j}{\tau_j}) } {   \Phi(\tfrac{y-\eta_j}{\tau_j})   } \right )
\label{eqn:fjcdensity}
\end{equation}
so that the terms in parentheses  from left to right 
are the marginal density for $y$ and the conditional density 
for $\bs x$ given $y$, both conditional on the selection event 
$\bs X \prec Y$. In what follows, we denote these densities as 
$\psel(y|\theta,\bs\eta )$ and $\pcon(\bs x|y, \bs \eta  )$.

For fixed $\bs\eta$ 
the marginal model for $Y$ has densities 
$\{ \psel( y |\theta,\bs\eta ) : \theta\in \mathbb R\}$,
which constitute a one-parameter exponential family with complete sufficient 
statistic $Y$. 
If one ascribes to the likelihood principle 
then, from the perspective of an ``underdog'' with knowledge of $\bs\eta$,
this is the model from which inference for $\theta$ is to be derived, 
as $\pcon(\bs x| y,\bs\eta)$ in (\ref{eqn:fjcdensity}) does not depend 
on any unknown parameters. 
% appeal to the likelihood principle 
A $1-\alpha$ confidence interval for $\theta$ based on observation of $Y$ from this marginal model can be constructed 
from the inversion of a collection 
of level-$\alpha$ hypothesis tests. Specifically, for each $\theta_0\in \mathbb R$
let $A(\theta_0)$ be a subset of $\mathbb R$ such that 
$\Pr( Y \in A(\theta_0) | \bs X\prec Y, \theta_0) \geq 1-\alpha$
where the probability is under the density %(\ref{eqn:mcdensity}) at the value $\theta=\theta_0$. 
$\psel( y |\theta_0,\bs\eta )$.
Then $A(\theta_0)$ is the 
acceptance region of a level-$\alpha$ test of $H:\theta=\theta_0$. 
The confidence set based on $\{ A(\theta_0) : \theta_0 \in \mathbb R\}$ is the 
set-valued function 
$C(y ) = \{ \theta_0 : y\in  A(\theta_0) \}.$
Evidently, 
\begin{align*} 
\Pr( \theta_0 \in C(Y) |\bs X\prec Y, \theta_0 ) = \Pr( Y \in A(\theta_0)| \bs X\prec Y,\theta_0 ) \geq 1-\alpha, 
\end{align*} 
and so such a $C$ has $1-\alpha$ selective coverage for this fixed value of 
$\bs\eta$. 

The precision of the confidence interval $C$ is a function of the power of 
the tests $\{A(\theta_0):\theta_0\in \mathbb R\}$ used in its construction. 
While there is no uniformly most powerful test of $H:\theta=\theta_0$,
%for any $\theta_0$, 
there does exist a uniformly most powerful unbiased (UMPU) test
because 
for fixed $\bs\eta$ the densities 
$\{ \psel( y |\theta,\bs\eta ) : \theta\in \mathbb R\}$ 
constitute a one-parameter exponential family 
%since %(\ref{eqn:mcdensity}) 
\citep[Section 4.2]{lehmann_romano_2005}. 
We have observed numerically in many scenarios 
that the coverage rates and expected widths of the confidence 
interval derived from UMPU tests are nearly identical to those of the following 
simpler-to-construct 
equal-tailed quantile test:
Let $l(\theta_0,\bs\eta)$ and $u(\theta_0,\bs\eta)$ be 
the $\alpha/2$ and $1-\alpha/2$ quantiles of the distribution with density 
%(\ref{eqn:mcdensity}).    
$\psel(y|\theta,\bs\eta)$. 
Then for each $\theta_0\in \mathbb R$, 
$  A(\theta_0 ) = \left ( l(\theta_0,\bs\eta), u(\theta_0,\bs\eta)  \right )  $
is the acceptance region of a size-$\alpha$ test of $H:\theta=\theta_0$. 
Inverting such regions for each $\theta_0\in \mathbb R$ gives the confidence interval
\begin{equation} 
C(y) = \{ \theta : l(\theta,\bs\eta)< y < u(\theta,\bs\eta) \},  
\label{eqn:oregion}
\end{equation}
which has 
exact $1-\alpha$ selective coverage for all $\theta\in \mathbb R$. 
This confidence interval may be written as
$C(y) = ( \theta_l ,\theta_u )$, 
where the lower and upper endpoints $\theta_l<\theta_u$ 
are solutions to the  equations
$y=u(\theta_l,\bs \eta)$ and $y=l(\theta_u,\bs\eta)$. 
We refer to this interval
as an oracle confidence interval for $\theta$, as its construction is 
only possible using the values of the unknown $\bs\eta$. 
This procedure has $1-\alpha$ selective 
coverage for all $\theta\in \mathbb R$ by design. We believe its expected width 
to be finite for all $(\bs\eta,\theta)$, and are able to show this theoretically in the 
%following special case of one unselected group: 
following special case involving two groups:
\begin{theorem}
\label{thm:finite}
In the case of two groups $(p = 1)$ and $\tau = \sigma$  the confidence interval (\ref{eqn:oregion})
has finite selective expected  width  for all $(\eta,\theta) \in \bbR^2$.
\end{theorem}

See Appendix \ref{app:B} for a proof. 

\subsection{Adaptive quantile estimates} 
Since $\bs\eta$ is unavailable, so are the quantile functions $l$ and $u$, and so for each hypothesized value of $\theta_0$ we estimate these 
quantities using the observed data $\bs X$ and $Y$. 
The resulting acceptance regions are of 
the form $A(\theta_0) = \{ (\x,y) : \hat l(\theta_0,\x,y) < y < \hat u(\theta_0,\x,y) \}$. 
For the resulting confidence regions to have approximate $1-\alpha$ coverage, we need that  $\hat l$ and $\hat u$  
satisfy
$\Pr( (\X,Y) \in A(\theta_0) | \bs\eta,\theta_0)\approx 1-\alpha$ for all $\bs\eta$ and $\theta_0$, 
or more or less equivalently, 
\begin{align} 
\Pr( Y< \hat l(\theta_0,\X,Y) | \bs \eta, \theta_0 ) \approx \alpha/2   \label{eqn:lc} \\ 
\Pr( Y> \hat u(\theta_0,\X,Y) | \bs \eta,\theta_0 ) \approx \alpha/2 \label{eqn:uc} 
\end{align}  
for all $(\bs\eta,\theta_0)\in \mathbb R^{p+1}$. 
On the other hand, for the confidence regions to be precise, we want these approximate tests to be powerful, that is, 
we want $\Pr( (\X,Y) \not\in A(\theta_0) | \bs \eta,\theta)$ to be large if $\theta\neq \theta_0$. 
Essentially, we want $\hat l$ and $\hat u$ to satisfy 
\begin{align} 
\Pr( Y< \hat l(\theta_0,\X,Y) | \bs\eta,\theta ) >  \alpha/2  \text{ for $\theta<\theta_0$ }  \label{eqn:lp} \\
\Pr( Y> \hat u(\theta_0,\X,Y) | \bs\eta,\theta )  > \alpha/2 \text{ for $\theta> \theta_0$. }  \label{eqn:up}
\end{align}  

We consider three strategies for obtaining 
estimates of $\hat l$ and $\hat u$, each constructed from 
plug-in estimates $\hat {\bs\eta}$ of $\bs\eta$. The estimates of $l$ and $u$  will then be
of the form 
\begin{align} 
\hat l(\theta_0,\X,Y)& =l(\theta_0,\hat{\bs\eta}(\theta_0,\X,Y) ) \label{eqn:lhat} \\
\hat u(\theta_0,\X,Y) & = u(\theta_0,\hat{\bs\eta}(\theta_0,\X,Y)),\label{eqn:uhat}
\end{align}
so that $\hat{\bs\eta}$ may depend on the data $(\X,Y)$ as well 
as the particular value of $\theta_0$ being tested. 

To achieve the approximate coverage in  (\ref{eqn:lc}) and (\ref{eqn:uc}), 
we need $\hat l(\theta_0,\X,Y) \approx l(\theta_0,\bs\eta)$ 
(and similarly $\hat u \approx u$) 
under $(\X,Y)\sim P_{\bs\eta,\theta_0}$ for a range of $\bs\eta$-values for each fixed $\theta_0$ value. 
From (\ref{eqn:lhat}) this means that we 
need $\hat{\bs\eta}(\theta_0,\X,Y)$ to be a good estimate of $\bs\eta$ in 
the submodel $\mathcal P_{\theta_0} = \{ P_{\bs\eta,\theta_0} : \bs\eta\in \mathbb R^m \}$. 
In other words, the probability of coverage of the value $\theta_0$ 
depends on the accuracy of  
$\hat{\bs\eta}(\theta_0,\X,Y)$
 under  $(\X,Y) \sim P_{\bs\eta,\theta_0}$ 
 and
 not on its accuracy 
under $(\X,Y) \sim P_{\bs\eta,\theta}$ for $\theta\neq \theta_0$. 
This fact suggests that,  to maintain a selective coverage rate, 
we might use  
an estimate of $\bs\eta$ based on the marginal distribution 
of $\X$ under the submodel $\mathcal P_{\theta_0}$. 
One such estimator is the  maximum likelihood estimator (MLE) of 
$\bs\eta$ under this submodel. We refer to this estimator as the profile MLE $\hat{\bs\eta}_P$, which is defined as 
\begin{align*}
 \hat{\bs\eta}_P(\theta_0 ,\x ) = 
 \arg \max_{\bs\eta} \psel(\x | \bs\eta,\theta_0 )   = 
   \arg\max_{\bs\eta}   \frac{  \Phi \left ( \tfrac{ x - \theta_0}{\sigma } \right ) \left \{ \prod_j \phi(\frac{x_j-\eta_j}{\tau_j})  \right \} 
   } {c(\bs\eta,\theta_0)},
\end{align*}
where $x = \max \{ x_1,\ldots, x_p\}$.

However, while use of $\hat{\bs\eta}_P$ to construct approximate quantile functions 
may result in approximate error rate control (and thus approximately 
correct coverage), its performance in terms of power (and thus 
interval width) may be poor. The reason is that the 
accuracy %precision 
of $\hat{\bs\eta}_P$ when $\theta=\theta_0$ comes at the expense of 
inaccuracy %imprecision 
when $\theta\neq \theta_0$, which could result in 
poor power. For example, suppose $(\X,Y) \sim P_{\bs\eta,\theta}$ 
for some $\theta$ that is much lower than $\theta_0$. Then we 
hope that our confidence interval is unlikely to contain $\theta_0$, that is 
 (referring to Equation 
\ref{eqn:lp}) 
$Y < \hat l(\theta_0,\X,Y)$  with high probability. 
Unfortunately, using $\hat{\bs\eta}_P$ in $\hat l$, 
so that $\hat l(\theta_0 ,\X,Y) = l(\theta_0 , \hat{\bs\eta}_P(\theta_0,\X))$, 
is likely to lead to $l(\theta_0,\bs\eta)$ being underestimated, which 
will lead to a lower probability of rejecting $\theta_0$ and 
hence a wider confidence interval. 
\iffalse Arguing heuristically, if 
$(\X,Y) \sim P_{\bs\eta,\theta}$ with $\theta<\theta_0$ then 
the values of $(\X,Y)$ will be lower than would be 
expected under $P_{\bs\eta,\theta_0}$. Since $\theta_0$ is fixed 
in estimation of $\hat{\bs\eta}_P$, the low values result in a downwardly biased 
estimate of $\bs\eta$, with  $\hat{\bs\eta}_P$ likely being less than $\bs\eta$. Finally, 
since $l(\theta_0,\bs\eta)$ is decreasing in $\bs\eta$ for fixed $\theta_0$, 
 $\hat l(\theta_0,\hat{\bs\eta}_P)$ is likely 
less than $l(\theta_0,\bs\eta)$, which reduces the probability that 
$Y$  rejects $\theta_0$ via (\ref{eqn:lp}). 
\fi

Such concerns suggest that to achieve good power and hence a narrow confidence interval we need an 
estimate $\hat{\bs\eta}(\theta_0,\X,Y)$ that is accurate 
at $\theta_0$ (to ensure coverage at $\theta_0$ if it is true) 
\emph{and} at other $\theta$ values (to ensure rejection of $\theta_0$ if it is 
false). Simply put, we seek an estimate $\hat{\bs\eta}$ that is 
accurate for a wide range  of $(\bs\eta,\theta)$-values. One possibility is to 
estimate $\bs\eta$ from the conditional model for $\bs X$ given $\{\bs X \prec Y,~Y=y\}$, 
which has densities $\{ \pcon(\bs x | y, \bs\eta) , \bs\eta\in \mathbb R^p\}$ 
given in (\ref{eqn:fjcdensity}) that depend 
 only on $\bs\eta$ and $y$ and not $\theta$. 
Letting $\hat{\bs\eta}_C$ be this conditional MLE, we consider estimating the quantile functions as
$\hat l(\theta_0,\X,Y) = l(\theta_0,\hat{\bs\eta}_C)$ and 
$\hat u(\theta_0,\X,Y) = u(\theta_0,\hat{\bs\eta}_C)$. 

While we expect $\hat{\bs\eta}_C$ to be less biased than $\hat{\bs\eta}_P$ away from 
the true $\theta$-value (and thus potentially lead to  greater power), 
the conditional MLE of $\bs\eta$ based on only a single 
vector $\bs X$ could be quite variable unless $\bs\tau$ and $\sigma$ are 
small. This suggests the  use of an estimator with lower variance, 
such as a Bayes estimator: If accurate prior 
information about $\bs\eta$ is available, then combining this with  
the conditional likelihood based on $ \pcon(\bs x | y, \bs\eta)$ 
should produce a Bayes estimator $\hat{\bs\eta}_B$ 
with similar bias as $\hat{\bs\eta}_C$ but lower variance. 
Thus, in addition to $\hat{\bs \eta}_P$ and $\hat{\bs \eta}_C$, 
 we also consider a posterior mode estimator  defined as 
\[ \hat{\bs\eta}_B =  \arg \max_{\bs\eta}  \pcon(\bs x|y,\bs\eta)\times \pi(\bs\eta) \]
where 
$\pcon(\bs x|\bs\eta,y)$ is the conditional density of $\bs X$ given $\bs X\prec Y,Y=y$ and 
$\pi(\bs\eta)$ is a probability density describing the prior information. 

\subsection{Numerical illustration}
The properties of the adaptive tests discussed in the preceding 
subsection are illustrated numerically for a simple scenario 
in Figure \ref{fig:power}. 
The figure considers the simple two-group case where 
$p=1$, $\eta=\theta_0=0$, and $\tau=\sigma=1$. 
The left-side panel  shows the power function of 
the level-$\alpha$ ($\alpha=.05$)
oracle test given by (\ref{eqn:oregion}) as well as the 
power functions of the 
various  adaptive tests with acceptance regions of the form 
\[ A(\theta_0) =\{ (x,y) : l(\theta_0,\hat \eta(\theta_0,x,y)) < y < 
u(\theta_0,\hat \eta(\theta_0,x,y))  \}, \]
where 
$\hat \eta$ is one of the three estimators described above and 
$l(\theta,\eta)$ and $u(\theta,\eta)$ are the 
.025 and .975 quantiles of the $Y$-margin of $P_{\eta,\theta}$. 

As it was designed to do, the test based on $\hat\eta_P$ (in green) provides
 level-$\alpha$ error rate control for this scenario, but has considerably less power than the oracle procedure (in blue). As discussed above, 
this is partly explained by the 
the bias of $\hat\eta_P$ when $\theta\neq \theta_0$, which can be seen on the right-side panel 
of the figure. 
In contrast, the test based on $\hat\eta_C$ has high power but also a high 
size at $\theta_0$, indicating that the corresponding interval 
will not maintain $1-\alpha$ coverage. 
This is partly explained by the high variance and positive skew of the distribution 
of $\hat\eta_C$ for values of $\theta$ less than $\eta$. 
Finally, from the right-side panel we see that a Bayes estimator $\hat\eta_B$ (using the prior $\eta\sim N(0,1)$)
has low bias and low variance compared to the other two estimators, and thus provides a good approximation to power function of the oracle procedure, as 
shown in the left-side panel. 
\begin{figure}[!t]
\centerline{\includegraphics[scale=.75]{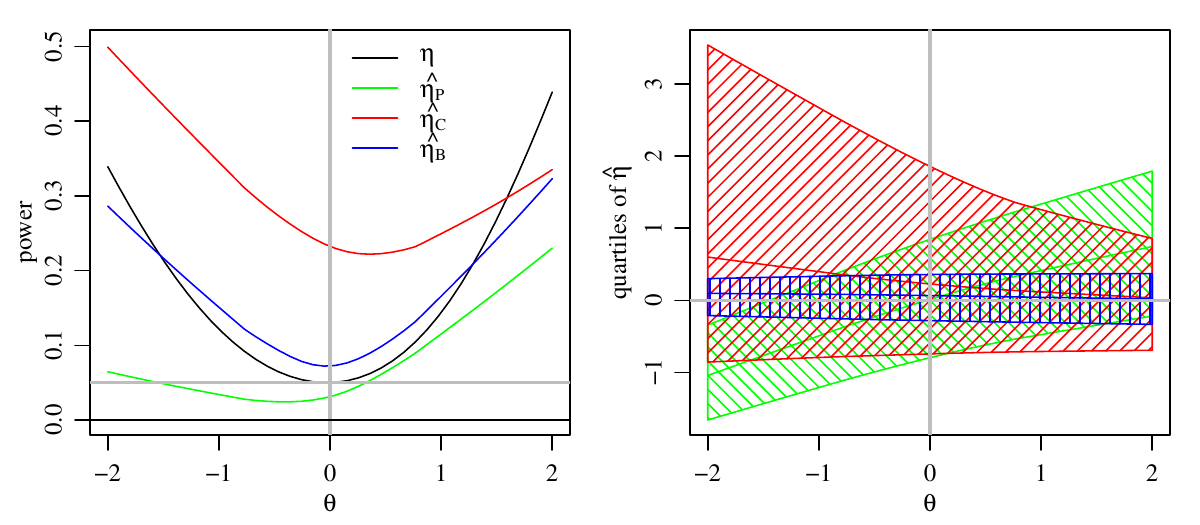}} 
\caption{Numerical comparison of different adaptive tests and estimates: 
  The left-side panel gives power functions of the nominal 
level-.05 tests of $\theta=0$. The right side panel gives the 
  25th, 50th and 75th percentiles of the distributions 
of $\hat \eta_P$ (green), $\hat \eta_F$ (red) and $\hat \eta_B$ (blue) 
under $(X,Y)\sim P_{0,0}$. }
\label{fig:power} 
\end{figure}

\section{Adaptation via empirical Bayes}

The numerical results in the previous section suggest that, in general, it is challenging to construct a practicable confidence procedure which mimics the oracle recipe and thus achieves similar selective coverage with short intervals. The only exception appears to be the case when reasonably accurate information about $\bs\eta$ is available {\it a priori}, so that a Bayes estimate of $\bs\eta$ is favorably shrunk toward the true value, allowing for good estimation of the quantiles of $\psel(y|\theta_0,\bs\eta)$ for all candidate values of $\theta_0$. But, one may ask, how is having reasonably accurate prior information on $\bs\eta$ any different from knowing $\bs\eta$ exactly? Indeed, the two situations are quite similar when $p$ is small, but a stark difference arises when $p$ is moderately large. In the latter case, estimation of $\bs\eta$ can be shrunk adaptively by gleaning key information about $\bs \eta$ from the data $\bs X$ in an empirical Bayesian manner.

Empirical Bayesian (EB) shrinkage is typically well suited for estimating multiple quantities simultaneously when the estimation error is measured by a composite loss. This appears to be the case with estimation of $\bs\eta$ when the only  goal is to accurately estimate the density $\psel(y|\theta_0,\bs\eta)$ and its quantiles. Consider a plug-in estimator $\widehat\psel(y|\theta,\bs\eta)$ with $\bs\eta$ replaced by $\hat{\bs\eta}$.
From (\ref{eqn:fjcdensity}), 
 we may write $\widehat\psel(y|\theta) \propto \psel(y|\theta,\hat{\bs\eta})r(y)$ where
\[
r(y) = \prod_{j=1}^p \frac{F(y|\eta_j,\tau_j)}{F(y|\hat\eta_j,\tau_j)},~y \in \bbR.
\]
While $r(y) \approx 1$ for all large $y$, for $\widehat \psel$ to be a good estimate of $\psel$ one must have $r(y) \approx 1$, i.e., $R(y) := \log r(y) \approx 0$ for all values of $y$. 
To see when this might be the case, 
assume $\tau_1 = \cdots = \tau_p = \tau$ and suppose the empirical distribution of $\{\eta_1,\ldots,\eta_p\}$ is well approximated by a $N(m, v)$  distribution for some $m$ and $v > 0$. For any $y \ll m$, we should have all $\eta_j > y$, and hence 
\begin{equation}
\label{approx}
R(y)  \approx \frac1{\tau}\sum_{j =1}^p (\hat\eta_j - \eta_j)\frac{f(y|\eta_j,\tau)}{F(y|\eta_j,\tau)} \approx \frac1\tau \sum_{j = 1}^p (\hat\eta_j - \eta_j)(\eta_j - y),
\end{equation}
where the first approximation is due to Taylor's theorem and the second follows from Mill's ratio bounds for normal distributions (see Appendix \ref{app:A}). 
%Since it is desirable that $R(y)^2$ be small, % on average, it is clear that we are mostly concerned about minimizing a composite, albeit complicated looking loss function in $\bs\eta$. 
From this we see that for $R(y)$ to be small so that $\widehat \psel \approx \psel$, 
we need $\hat \eta_j \approx \eta_j$ on average across $j=1,\ldots,p$, that is, 
we are mostly concerned about minimizing a composite, albeit complicated looking loss function in $\bs\eta$.

To further elucidate on the scope of shrinkage in the present context, consider the selection-unadjusted estimate $\hat\eta_j = X_j$.  On the one hand, we get
\[
\Exp{R(y)^2} %= \sum_{j > 1} \Exp{(\theta_j - y_1)^2} 
\approx p\{v + (m - y)^2\} = p(m - y)^2\{1 + o(1)\}
\]
for $y \ll m$. On the other hand, with the Bayes estimate $\hat\eta_j = \rho X_j + (1 - \rho)m$ with shrinkage factor $\rho = v/(\tau^2 + v)$, we get
\begin{align*}
\Exp{R(y)^2} & \approx \rho^2 p\{v + (m - y)^2\} + (1 - \rho)^2\{p(p+2)v^2 + pv(m - y)^2 \}\\
& = \rho p(m - y)^2\{1 + o(1)\}.
\end{align*}
Therefore, the Bayes estimate could offer substantial improvement when $\rho$ is small. In practice, we do not know $m$ and $v$, but these %``population'' 
quantities that describe the heterogeneity of $\eta_1,\ldots,\eta_p$ could be estimated from $\bs X$, for example in the empirical Bayesian tradition of maximizing the marginal likelihood function in $m$ and $v$, having integrated out $\bs\eta$. 
Specifically, consider an i.i.d.\ Gaussian model 
for $\bs \eta$ which we denote 
$\bs \eta \sim G(\bs \eta |m,v)$.
This 
induces 
a marginal model on $\bs X$ that depends on $(m,v)$. 
The density of $\bs X$ given $(m,v)$ and the selection event $\{ \bs X \prec Y,  \  Y=y \}$ 
can be constructed, from which a marginal likelihood estimate $(\hat m,\hat v)$ is obtained. 
An empirical Bayes estimate of $\bs\eta$ is then obtained by optimizing 
$\pcon(\bs x | y, \bs\eta) \times G ( \bs \eta | \hat m, \hat v)$. 
%In our implementation, 
%we carry out this estimation by subjecting the joint distribution of $(\X,\bs\eta)$ to the selection condition $\{\X \prec y\}$ where $y$ is the observed value of $Y$; 
See Appendix \ref{app:C} for implementation details. This selection adjustment for estimation of $(m,v)$ is analogous to the {\it random-parameter} adjustment discussed by \cite{yekutieli2012adjusted}.

\begin{figure}[!t]
\centering
%%% Gaussian narrow (scale = 0.5)
% Created by tikzDevice version 0.12.6 on 2025-09-10 23:29:26
% !TEX encoding = UTF-8 Unicode
\begin{tikzpicture}[x=1pt,y=1pt]
\definecolor{fillColor}{RGB}{255,255,255}
\path[use as bounding box,fill=fillColor,fill opacity=0.00] (0,0) rectangle (101.18,108.41);
\begin{scope}
\path[clip] (  0.00,  0.00) rectangle (101.18,108.41);
\definecolor{drawColor}{RGB}{0,0,0}

\path[draw=drawColor,line width= 0.4pt,line join=round,line cap=round] ( 21.61, 18.60) -- ( 92.25, 18.60);

\path[draw=drawColor,line width= 0.4pt,line join=round,line cap=round] ( 21.61, 18.60) -- ( 21.61, 17.40);

\path[draw=drawColor,line width= 0.4pt,line join=round,line cap=round] ( 45.16, 18.60) -- ( 45.16, 17.40);

\path[draw=drawColor,line width= 0.4pt,line join=round,line cap=round] ( 68.71, 18.60) -- ( 68.71, 17.40);

\path[draw=drawColor,line width= 0.4pt,line join=round,line cap=round] ( 92.25, 18.60) -- ( 92.25, 17.40);

\node[text=drawColor,anchor=base,inner sep=0pt, outer sep=0pt, scale=  0.80] at ( 21.61,  9.60) {-10};

\node[text=drawColor,anchor=base,inner sep=0pt, outer sep=0pt, scale=  0.80] at ( 45.16,  9.60) {-5};

\node[text=drawColor,anchor=base,inner sep=0pt, outer sep=0pt, scale=  0.80] at ( 68.71,  9.60) {0};

\node[text=drawColor,anchor=base,inner sep=0pt, outer sep=0pt, scale=  0.80] at ( 92.25,  9.60) {5};

\path[draw=drawColor,line width= 0.4pt,line join=round,line cap=round] ( 18.60, 21.24) -- ( 18.60, 87.17);

\path[draw=drawColor,line width= 0.4pt,line join=round,line cap=round] ( 18.60, 21.24) -- ( 17.40, 21.24);

\path[draw=drawColor,line width= 0.4pt,line join=round,line cap=round] ( 18.60, 34.42) -- ( 17.40, 34.42);

\path[draw=drawColor,line width= 0.4pt,line join=round,line cap=round] ( 18.60, 47.61) -- ( 17.40, 47.61);

\path[draw=drawColor,line width= 0.4pt,line join=round,line cap=round] ( 18.60, 60.80) -- ( 17.40, 60.80);

\path[draw=drawColor,line width= 0.4pt,line join=round,line cap=round] ( 18.60, 73.98) -- ( 17.40, 73.98);

\path[draw=drawColor,line width= 0.4pt,line join=round,line cap=round] ( 18.60, 87.17) -- ( 17.40, 87.17);

\node[text=drawColor,rotate= 90.00,anchor=base,inner sep=0pt, outer sep=0pt, scale=  0.80] at ( 16.80, 21.24) {0.0};

\node[text=drawColor,rotate= 90.00,anchor=base,inner sep=0pt, outer sep=0pt, scale=  0.80] at ( 16.80, 47.61) {0.4};

\node[text=drawColor,rotate= 90.00,anchor=base,inner sep=0pt, outer sep=0pt, scale=  0.80] at ( 16.80, 73.98) {0.8};
\end{scope}
\begin{scope}
\path[clip] ( 18.60, 18.60) rectangle ( 99.98, 89.81);
\definecolor{drawColor}{RGB}{211,211,211}

\path[draw=drawColor,line width= 0.4pt,dash pattern=on 1pt off 3pt ,line join=round,line cap=round] ( 18.60, 21.24) -- ( 99.98, 21.24);

\path[draw=drawColor,line width= 0.4pt,dash pattern=on 1pt off 3pt ,line join=round,line cap=round] ( 18.60, 34.42) -- ( 99.98, 34.42);

\path[draw=drawColor,line width= 0.4pt,dash pattern=on 1pt off 3pt ,line join=round,line cap=round] ( 18.60, 47.61) -- ( 99.98, 47.61);

\path[draw=drawColor,line width= 0.4pt,dash pattern=on 1pt off 3pt ,line join=round,line cap=round] ( 18.60, 60.80) -- ( 99.98, 60.80);

\path[draw=drawColor,line width= 0.4pt,dash pattern=on 1pt off 3pt ,line join=round,line cap=round] ( 18.60, 73.98) -- ( 99.98, 73.98);

\path[draw=drawColor,line width= 0.4pt,dash pattern=on 1pt off 3pt ,line join=round,line cap=round] ( 18.60, 87.17) -- ( 99.98, 87.17);

\path[draw=drawColor,line width= 0.4pt,dash pattern=on 1pt off 3pt ,line join=round,line cap=round] ( 21.61, 18.60) -- ( 21.61, 89.81);

\path[draw=drawColor,line width= 0.4pt,dash pattern=on 1pt off 3pt ,line join=round,line cap=round] ( 45.16, 18.60) -- ( 45.16, 89.81);

\path[draw=drawColor,line width= 0.4pt,dash pattern=on 1pt off 3pt ,line join=round,line cap=round] ( 68.71, 18.60) -- ( 68.71, 89.81);

\path[draw=drawColor,line width= 0.4pt,dash pattern=on 1pt off 3pt ,line join=round,line cap=round] ( 92.25, 18.60) -- ( 92.25, 89.81);
\end{scope}
\begin{scope}
\path[clip] (  0.00,  0.00) rectangle (101.18,108.41);
\definecolor{drawColor}{RGB}{0,0,0}

\node[text=drawColor,anchor=base,inner sep=0pt, outer sep=0pt, scale=  1.00] at ( 59.29,  0.60) {\footnotesize$t$};

\node[text=drawColor,rotate= 90.00,anchor=base,inner sep=0pt, outer sep=0pt, scale=  1.00] at (  6.60, 54.20) {\footnotesize power};
\end{scope}
\begin{scope}
\path[clip] (  0.00,  0.00) rectangle (101.18,108.41);
\definecolor{drawColor}{RGB}{0,0,0}

\node[text=drawColor,anchor=base,inner sep=0pt, outer sep=0pt, scale=  1.00] at ( 59.29, 92.21) {\footnotesize $P= 1.83\times 10^{-5} $};
\end{scope}
\begin{scope}
\path[clip] ( 18.60, 18.60) rectangle ( 99.98, 89.81);
\definecolor{drawColor}{RGB}{97,208,79}

\path[draw=drawColor,line width= 0.4pt,line join=round,line cap=round] ( 21.61, 67.64) --
	( 23.97, 64.70) --
	( 26.32, 61.69) --
	( 28.68, 58.09) --
	( 31.03, 54.49) --
	( 33.39, 51.97) --
	( 35.74, 47.52) --
	( 38.10, 43.71) --
	( 40.45, 40.53) --
	( 42.81, 37.75) --
	( 45.16, 34.44) --
	( 47.52, 32.25) --
	( 49.87, 29.68) --
	( 52.22, 27.56) --
	( 54.58, 25.78) --
	( 56.93, 24.72) --
	( 59.29, 24.59) --
	( 61.64, 24.53) --
	( 64.00, 25.33) --
	( 66.35, 26.49) --
	( 68.71, 28.89) --
	( 71.06, 31.26) --
	( 73.42, 35.52) --
	( 75.77, 40.65) --
	( 78.13, 47.68) --
	( 80.48, 55.47) --
	( 82.84, 64.62) --
	( 85.19, 74.09) --
	( 87.55, 81.50) --
	( 89.90, 85.32) --
	( 92.25, 86.69) --
	( 94.61, 87.13) --
	( 96.96, 87.17);
\definecolor{drawColor}{RGB}{34,151,230}

\path[draw=drawColor,line width= 0.4pt,line join=round,line cap=round] ( 21.61, 70.36) --
	( 23.97, 68.33) --
	( 26.32, 65.92) --
	( 28.68, 62.91) --
	( 31.03, 59.78) --
	( 33.39, 57.46) --
	( 35.74, 53.36) --
	( 38.10, 49.34) --
	( 40.45, 46.28) --
	( 42.81, 43.65) --
	( 45.16, 40.41) --
	( 47.52, 36.71) --
	( 49.87, 33.82) --
	( 52.22, 31.15) --
	( 54.58, 28.77) --
	( 56.93, 27.02) --
	( 59.29, 26.25) --
	( 61.64, 25.89) --
	( 64.00, 25.84) --
	( 66.35, 26.11) --
	( 68.71, 27.76) --
	( 71.06, 29.69) --
	( 73.42, 32.88) --
	( 75.77, 36.57) --
	( 78.13, 42.14) --
	( 80.48, 49.61) --
	( 82.84, 58.10) --
	( 85.19, 69.62) --
	( 87.55, 78.91) --
	( 89.90, 84.73) --
	( 92.25, 86.58) --
	( 94.61, 87.12) --
	( 96.96, 87.17);
\definecolor{drawColor}{RGB}{40,226,229}

\path[draw=drawColor,line width= 0.4pt,line join=round,line cap=round] ( 21.61, 68.78) --
	( 23.97, 66.62) --
	( 26.32, 64.22) --
	( 28.68, 61.14) --
	( 31.03, 58.26) --
	( 33.39, 55.91) --
	( 35.74, 51.78) --
	( 38.10, 48.20) --
	( 40.45, 45.18) --
	( 42.81, 42.72) --
	( 45.16, 39.68) --
	( 47.52, 36.24) --
	( 49.87, 33.59) --
	( 52.22, 31.13) --
	( 54.58, 29.10) --
	( 56.93, 27.53) --
	( 59.29, 26.93) --
	( 61.64, 26.62) --
	( 64.00, 26.83) --
	( 66.35, 27.46) --
	( 68.71, 29.04) --
	( 71.06, 31.28) --
	( 73.42, 34.71) --
	( 75.77, 38.25) --
	( 78.13, 43.89) --
	( 80.48, 51.27) --
	( 82.84, 59.67) --
	( 85.19, 70.46) --
	( 87.55, 79.26) --
	( 89.90, 84.86) --
	( 92.25, 86.60) --
	( 94.61, 87.12) --
	( 96.96, 87.17);
\definecolor{drawColor}{RGB}{0,0,0}

\path[draw=drawColor,line width= 0.4pt,dash pattern=on 4pt off 4pt ,line join=round,line cap=round] ( 18.60, 24.53) -- ( 99.98, 24.53);

\path[draw=drawColor,line width= 0.4pt,dash pattern=on 4pt off 4pt ,line join=round,line cap=round] ( 59.29, 18.60) -- ( 59.29, 89.81);

\node[text=drawColor,rotate= 90.00,anchor=base,inner sep=0pt, outer sep=0pt, scale=  1.00] at ( 66.50, 70.69) {\tiny$\theta = -2 $};
\end{scope}
\end{tikzpicture}
% Created by tikzDevice version 0.12.6 on 2025-09-10 23:29:26
% !TEX encoding = UTF-8 Unicode
\begin{tikzpicture}[x=1pt,y=1pt]
\definecolor{fillColor}{RGB}{255,255,255}
\path[use as bounding box,fill=fillColor,fill opacity=0.00] (0,0) rectangle (101.18,108.41);
\begin{scope}
\path[clip] (  0.00,  0.00) rectangle (101.18,108.41);
\definecolor{drawColor}{RGB}{0,0,0}

\path[draw=drawColor,line width= 0.4pt,line join=round,line cap=round] ( 38.88, 18.60) -- ( 85.98, 18.60);

\path[draw=drawColor,line width= 0.4pt,line join=round,line cap=round] ( 38.88, 18.60) -- ( 38.88, 17.40);

\path[draw=drawColor,line width= 0.4pt,line join=round,line cap=round] ( 62.43, 18.60) -- ( 62.43, 17.40);

\path[draw=drawColor,line width= 0.4pt,line join=round,line cap=round] ( 85.98, 18.60) -- ( 85.98, 17.40);

\node[text=drawColor,anchor=base,inner sep=0pt, outer sep=0pt, scale=  0.80] at ( 38.88,  9.60) {-5};

\node[text=drawColor,anchor=base,inner sep=0pt, outer sep=0pt, scale=  0.80] at ( 62.43,  9.60) {0};

\node[text=drawColor,anchor=base,inner sep=0pt, outer sep=0pt, scale=  0.80] at ( 85.98,  9.60) {5};

\path[draw=drawColor,line width= 0.4pt,line join=round,line cap=round] ( 18.60, 21.24) -- ( 18.60, 87.17);

\path[draw=drawColor,line width= 0.4pt,line join=round,line cap=round] ( 18.60, 21.24) -- ( 17.40, 21.24);

\path[draw=drawColor,line width= 0.4pt,line join=round,line cap=round] ( 18.60, 34.42) -- ( 17.40, 34.42);

\path[draw=drawColor,line width= 0.4pt,line join=round,line cap=round] ( 18.60, 47.61) -- ( 17.40, 47.61);

\path[draw=drawColor,line width= 0.4pt,line join=round,line cap=round] ( 18.60, 60.80) -- ( 17.40, 60.80);

\path[draw=drawColor,line width= 0.4pt,line join=round,line cap=round] ( 18.60, 73.98) -- ( 17.40, 73.98);

\path[draw=drawColor,line width= 0.4pt,line join=round,line cap=round] ( 18.60, 87.17) -- ( 17.40, 87.17);

\node[text=drawColor,rotate= 90.00,anchor=base,inner sep=0pt, outer sep=0pt, scale=  0.80] at ( 16.80, 21.24) {0.0};

\node[text=drawColor,rotate= 90.00,anchor=base,inner sep=0pt, outer sep=0pt, scale=  0.80] at ( 16.80, 47.61) {0.4};

\node[text=drawColor,rotate= 90.00,anchor=base,inner sep=0pt, outer sep=0pt, scale=  0.80] at ( 16.80, 73.98) {0.8};
\end{scope}
\begin{scope}
\path[clip] ( 18.60, 18.60) rectangle ( 99.98, 89.81);
\definecolor{drawColor}{RGB}{211,211,211}

\path[draw=drawColor,line width= 0.4pt,dash pattern=on 1pt off 3pt ,line join=round,line cap=round] ( 18.60, 21.24) -- ( 99.98, 21.24);

\path[draw=drawColor,line width= 0.4pt,dash pattern=on 1pt off 3pt ,line join=round,line cap=round] ( 18.60, 34.42) -- ( 99.98, 34.42);

\path[draw=drawColor,line width= 0.4pt,dash pattern=on 1pt off 3pt ,line join=round,line cap=round] ( 18.60, 47.61) -- ( 99.98, 47.61);

\path[draw=drawColor,line width= 0.4pt,dash pattern=on 1pt off 3pt ,line join=round,line cap=round] ( 18.60, 60.80) -- ( 99.98, 60.80);

\path[draw=drawColor,line width= 0.4pt,dash pattern=on 1pt off 3pt ,line join=round,line cap=round] ( 18.60, 73.98) -- ( 99.98, 73.98);

\path[draw=drawColor,line width= 0.4pt,dash pattern=on 1pt off 3pt ,line join=round,line cap=round] ( 18.60, 87.17) -- ( 99.98, 87.17);

\path[draw=drawColor,line width= 0.4pt,dash pattern=on 1pt off 3pt ,line join=round,line cap=round] ( 38.88, 18.60) -- ( 38.88, 89.81);

\path[draw=drawColor,line width= 0.4pt,dash pattern=on 1pt off 3pt ,line join=round,line cap=round] ( 62.43, 18.60) -- ( 62.43, 89.81);

\path[draw=drawColor,line width= 0.4pt,dash pattern=on 1pt off 3pt ,line join=round,line cap=round] ( 85.98, 18.60) -- ( 85.98, 89.81);
\end{scope}
\begin{scope}
\path[clip] (  0.00,  0.00) rectangle (101.18,108.41);
\definecolor{drawColor}{RGB}{0,0,0}

\node[text=drawColor,anchor=base,inner sep=0pt, outer sep=0pt, scale=  1.00] at ( 59.29,  0.60) {\footnotesize$t$};

\node[text=drawColor,rotate= 90.00,anchor=base,inner sep=0pt, outer sep=0pt, scale=  1.00] at (  6.60, 54.20) {\footnotesize power};
\end{scope}
\begin{scope}
\path[clip] (  0.00,  0.00) rectangle (101.18,108.41);
\definecolor{drawColor}{RGB}{0,0,0}

\node[text=drawColor,anchor=base,inner sep=0pt, outer sep=0pt, scale=  1.00] at ( 59.29, 92.21) {\footnotesize $P= 0.00176 $};
\end{scope}
\begin{scope}
\path[clip] ( 18.60, 18.60) rectangle ( 99.98, 89.81);
\definecolor{drawColor}{RGB}{97,208,79}

\path[draw=drawColor,line width= 0.4pt,line join=round,line cap=round] ( 21.61, 71.48) --
	( 23.97, 68.83) --
	( 26.32, 65.76) --
	( 28.68, 62.47) --
	( 31.03, 59.38) --
	( 33.39, 55.78) --
	( 35.74, 51.95) --
	( 38.10, 47.81) --
	( 40.45, 43.54) --
	( 42.81, 40.32) --
	( 45.16, 36.61) --
	( 47.52, 33.16) --
	( 49.87, 30.68) --
	( 52.22, 28.04) --
	( 54.58, 26.54) --
	( 56.93, 25.38) --
	( 59.29, 24.58) --
	( 61.64, 24.59) --
	( 64.00, 25.81) --
	( 66.35, 27.48) --
	( 68.71, 30.08) --
	( 71.06, 34.25) --
	( 73.42, 40.17) --
	( 75.77, 47.99) --
	( 78.13, 57.80) --
	( 80.48, 68.72) --
	( 82.84, 78.19) --
	( 85.19, 84.31) --
	( 87.55, 86.49) --
	( 89.90, 87.10) --
	( 92.25, 87.17) --
	( 94.61, 87.17) --
	( 96.96, 87.17);
\definecolor{drawColor}{RGB}{34,151,230}

\path[draw=drawColor,line width= 0.4pt,line join=round,line cap=round] ( 21.61, 73.94) --
	( 23.97, 72.08) --
	( 26.32, 69.47) --
	( 28.68, 66.31) --
	( 31.03, 63.97) --
	( 33.39, 60.87) --
	( 35.74, 56.80) --
	( 38.10, 53.02) --
	( 40.45, 49.36) --
	( 42.81, 45.81) --
	( 45.16, 41.79) --
	( 47.52, 37.94) --
	( 49.87, 34.72) --
	( 52.22, 31.48) --
	( 54.58, 29.13) --
	( 56.93, 26.81) --
	( 59.29, 25.51) --
	( 61.64, 25.16) --
	( 64.00, 25.40) --
	( 66.35, 26.55) --
	( 68.71, 28.58) --
	( 71.06, 31.61) --
	( 73.42, 36.15) --
	( 75.77, 42.76) --
	( 78.13, 52.90) --
	( 80.48, 64.27) --
	( 82.84, 76.08) --
	( 85.19, 83.79) --
	( 87.55, 86.43) --
	( 89.90, 87.10) --
	( 92.25, 87.17) --
	( 94.61, 87.17) --
	( 96.96, 87.17);
\definecolor{drawColor}{RGB}{40,226,229}

\path[draw=drawColor,line width= 0.4pt,line join=round,line cap=round] ( 21.61, 72.61) --
	( 23.97, 70.62) --
	( 26.32, 68.19) --
	( 28.68, 65.03) --
	( 31.03, 62.60) --
	( 33.39, 59.62) --
	( 35.74, 55.71) --
	( 38.10, 51.89) --
	( 40.45, 48.47) --
	( 42.81, 44.93) --
	( 45.16, 41.19) --
	( 47.52, 37.40) --
	( 49.87, 34.53) --
	( 52.22, 31.40) --
	( 54.58, 29.20) --
	( 56.93, 27.12) --
	( 59.29, 26.16) --
	( 61.64, 25.89) --
	( 64.00, 26.17) --
	( 66.35, 27.61) --
	( 68.71, 29.60) --
	( 71.06, 32.93) --
	( 73.42, 37.61) --
	( 75.77, 44.10) --
	( 78.13, 53.99) --
	( 80.48, 65.09) --
	( 82.84, 76.42) --
	( 85.19, 83.88) --
	( 87.55, 86.44) --
	( 89.90, 87.10) --
	( 92.25, 87.17) --
	( 94.61, 87.17) --
	( 96.96, 87.17);
\definecolor{drawColor}{RGB}{0,0,0}

\path[draw=drawColor,line width= 0.4pt,dash pattern=on 4pt off 4pt ,line join=round,line cap=round] ( 18.60, 24.53) -- ( 99.98, 24.53);

\path[draw=drawColor,line width= 0.4pt,dash pattern=on 4pt off 4pt ,line join=round,line cap=round] ( 59.13, 18.60) -- ( 59.13, 89.81);

\node[text=drawColor,rotate= 90.00,anchor=base,inner sep=0pt, outer sep=0pt, scale=  1.00] at ( 66.34, 70.69) {\tiny$\theta = -0.7 $};
\end{scope}
\end{tikzpicture}
% Created by tikzDevice version 0.12.6 on 2025-09-10 23:29:26
% !TEX encoding = UTF-8 Unicode
\begin{tikzpicture}[x=1pt,y=1pt]
\definecolor{fillColor}{RGB}{255,255,255}
\path[use as bounding box,fill=fillColor,fill opacity=0.00] (0,0) rectangle (101.18,108.41);
\begin{scope}
\path[clip] (  0.00,  0.00) rectangle (101.18,108.41);
\definecolor{drawColor}{RGB}{0,0,0}

\path[draw=drawColor,line width= 0.4pt,line join=round,line cap=round] ( 32.60, 18.60) -- ( 79.70, 18.60);

\path[draw=drawColor,line width= 0.4pt,line join=round,line cap=round] ( 32.60, 18.60) -- ( 32.60, 17.40);

\path[draw=drawColor,line width= 0.4pt,line join=round,line cap=round] ( 56.15, 18.60) -- ( 56.15, 17.40);

\path[draw=drawColor,line width= 0.4pt,line join=round,line cap=round] ( 79.70, 18.60) -- ( 79.70, 17.40);

\node[text=drawColor,anchor=base,inner sep=0pt, outer sep=0pt, scale=  0.80] at ( 32.60,  9.60) {-5};

\node[text=drawColor,anchor=base,inner sep=0pt, outer sep=0pt, scale=  0.80] at ( 56.15,  9.60) {0};

\node[text=drawColor,anchor=base,inner sep=0pt, outer sep=0pt, scale=  0.80] at ( 79.70,  9.60) {5};

\path[draw=drawColor,line width= 0.4pt,line join=round,line cap=round] ( 18.60, 21.24) -- ( 18.60, 87.17);

\path[draw=drawColor,line width= 0.4pt,line join=round,line cap=round] ( 18.60, 21.24) -- ( 17.40, 21.24);

\path[draw=drawColor,line width= 0.4pt,line join=round,line cap=round] ( 18.60, 34.42) -- ( 17.40, 34.42);

\path[draw=drawColor,line width= 0.4pt,line join=round,line cap=round] ( 18.60, 47.61) -- ( 17.40, 47.61);

\path[draw=drawColor,line width= 0.4pt,line join=round,line cap=round] ( 18.60, 60.80) -- ( 17.40, 60.80);

\path[draw=drawColor,line width= 0.4pt,line join=round,line cap=round] ( 18.60, 73.98) -- ( 17.40, 73.98);

\path[draw=drawColor,line width= 0.4pt,line join=round,line cap=round] ( 18.60, 87.17) -- ( 17.40, 87.17);

\node[text=drawColor,rotate= 90.00,anchor=base,inner sep=0pt, outer sep=0pt, scale=  0.80] at ( 16.80, 21.24) {0.0};

\node[text=drawColor,rotate= 90.00,anchor=base,inner sep=0pt, outer sep=0pt, scale=  0.80] at ( 16.80, 47.61) {0.4};

\node[text=drawColor,rotate= 90.00,anchor=base,inner sep=0pt, outer sep=0pt, scale=  0.80] at ( 16.80, 73.98) {0.8};
\end{scope}
\begin{scope}
\path[clip] ( 18.60, 18.60) rectangle ( 99.98, 89.81);
\definecolor{drawColor}{RGB}{211,211,211}

\path[draw=drawColor,line width= 0.4pt,dash pattern=on 1pt off 3pt ,line join=round,line cap=round] ( 18.60, 21.24) -- ( 99.98, 21.24);

\path[draw=drawColor,line width= 0.4pt,dash pattern=on 1pt off 3pt ,line join=round,line cap=round] ( 18.60, 34.42) -- ( 99.98, 34.42);

\path[draw=drawColor,line width= 0.4pt,dash pattern=on 1pt off 3pt ,line join=round,line cap=round] ( 18.60, 47.61) -- ( 99.98, 47.61);

\path[draw=drawColor,line width= 0.4pt,dash pattern=on 1pt off 3pt ,line join=round,line cap=round] ( 18.60, 60.80) -- ( 99.98, 60.80);

\path[draw=drawColor,line width= 0.4pt,dash pattern=on 1pt off 3pt ,line join=round,line cap=round] ( 18.60, 73.98) -- ( 99.98, 73.98);

\path[draw=drawColor,line width= 0.4pt,dash pattern=on 1pt off 3pt ,line join=round,line cap=round] ( 18.60, 87.17) -- ( 99.98, 87.17);

\path[draw=drawColor,line width= 0.4pt,dash pattern=on 1pt off 3pt ,line join=round,line cap=round] ( 32.60, 18.60) -- ( 32.60, 89.81);

\path[draw=drawColor,line width= 0.4pt,dash pattern=on 1pt off 3pt ,line join=round,line cap=round] ( 56.15, 18.60) -- ( 56.15, 89.81);

\path[draw=drawColor,line width= 0.4pt,dash pattern=on 1pt off 3pt ,line join=round,line cap=round] ( 79.70, 18.60) -- ( 79.70, 89.81);
\end{scope}
\begin{scope}
\path[clip] (  0.00,  0.00) rectangle (101.18,108.41);
\definecolor{drawColor}{RGB}{0,0,0}

\node[text=drawColor,anchor=base,inner sep=0pt, outer sep=0pt, scale=  1.00] at ( 59.29,  0.60) {\footnotesize$t$};

\node[text=drawColor,rotate= 90.00,anchor=base,inner sep=0pt, outer sep=0pt, scale=  1.00] at (  6.60, 54.20) {\footnotesize power};
\end{scope}
\begin{scope}
\path[clip] (  0.00,  0.00) rectangle (101.18,108.41);
\definecolor{drawColor}{RGB}{0,0,0}

\node[text=drawColor,anchor=base,inner sep=0pt, outer sep=0pt, scale=  1.00] at ( 59.29, 92.21) {\footnotesize $P= 0.0515 $};
\end{scope}
\begin{scope}
\path[clip] ( 18.60, 18.60) rectangle ( 99.98, 89.81);
\definecolor{drawColor}{RGB}{97,208,79}

\path[draw=drawColor,line width= 0.4pt,line join=round,line cap=round] ( 21.61, 76.14) --
	( 23.97, 73.47) --
	( 26.32, 71.28) --
	( 28.68, 67.51) --
	( 31.03, 64.76) --
	( 33.39, 61.70) --
	( 35.74, 57.53) --
	( 38.10, 53.17) --
	( 40.45, 48.88) --
	( 42.81, 44.68) --
	( 45.16, 40.38) --
	( 47.52, 36.16) --
	( 49.87, 32.72) --
	( 52.22, 29.43) --
	( 54.58, 27.00) --
	( 56.93, 25.38) --
	( 59.29, 24.63) --
	( 61.64, 24.73) --
	( 64.00, 25.87) --
	( 66.35, 28.38) --
	( 68.71, 32.54) --
	( 71.06, 39.30) --
	( 73.42, 49.26) --
	( 75.77, 61.45) --
	( 78.13, 73.32) --
	( 80.48, 82.03) --
	( 82.84, 85.86) --
	( 85.19, 86.88) --
	( 87.55, 87.13) --
	( 89.90, 87.17) --
	( 92.25, 87.17) --
	( 94.61, 87.17) --
	( 96.96, 87.17);
\definecolor{drawColor}{RGB}{34,151,230}

\path[draw=drawColor,line width= 0.4pt,line join=round,line cap=round] ( 21.61, 77.81) --
	( 23.97, 75.74) --
	( 26.32, 74.05) --
	( 28.68, 71.28) --
	( 31.03, 68.66) --
	( 33.39, 66.24) --
	( 35.74, 62.09) --
	( 38.10, 58.11) --
	( 40.45, 53.85) --
	( 42.81, 49.55) --
	( 45.16, 44.93) --
	( 47.52, 40.23) --
	( 49.87, 36.32) --
	( 52.22, 32.13) --
	( 54.58, 28.98) --
	( 56.93, 26.65) --
	( 59.29, 24.97) --
	( 61.64, 24.56) --
	( 64.00, 25.31) --
	( 66.35, 26.91) --
	( 68.71, 30.16) --
	( 71.06, 35.72) --
	( 73.42, 45.39) --
	( 75.77, 58.18) --
	( 78.13, 71.38) --
	( 80.48, 81.60) --
	( 82.84, 85.81) --
	( 85.19, 86.88) --
	( 87.55, 87.13) --
	( 89.90, 87.17) --
	( 92.25, 87.17) --
	( 94.61, 87.17) --
	( 96.96, 87.17);
\definecolor{drawColor}{RGB}{40,226,229}

\path[draw=drawColor,line width= 0.4pt,line join=round,line cap=round] ( 21.61, 77.03) --
	( 23.97, 74.68) --
	( 26.32, 73.12) --
	( 28.68, 70.14) --
	( 31.03, 67.55) --
	( 33.39, 65.27) --
	( 35.74, 61.15) --
	( 38.10, 57.27) --
	( 40.45, 53.07) --
	( 42.81, 49.02) --
	( 45.16, 44.40) --
	( 47.52, 39.81) --
	( 49.87, 35.98) --
	( 52.22, 32.05) --
	( 54.58, 28.99) --
	( 56.93, 26.92) --
	( 59.29, 25.32) --
	( 61.64, 25.11) --
	( 64.00, 25.92) --
	( 66.35, 27.62) --
	( 68.71, 30.93) --
	( 71.06, 36.82) --
	( 73.42, 46.16) --
	( 75.77, 58.72) --
	( 78.13, 71.65) --
	( 80.48, 81.63) --
	( 82.84, 85.82) --
	( 85.19, 86.88) --
	( 87.55, 87.13) --
	( 89.90, 87.17) --
	( 92.25, 87.17) --
	( 94.61, 87.17) --
	( 96.96, 87.17);
\definecolor{drawColor}{RGB}{0,0,0}

\path[draw=drawColor,line width= 0.4pt,dash pattern=on 4pt off 4pt ,line join=round,line cap=round] ( 18.60, 24.53) -- ( 99.98, 24.53);

\path[draw=drawColor,line width= 0.4pt,dash pattern=on 4pt off 4pt ,line join=round,line cap=round] ( 59.45, 18.60) -- ( 59.45, 89.81);

\node[text=drawColor,rotate= 90.00,anchor=base,inner sep=0pt, outer sep=0pt, scale=  1.00] at ( 66.65, 70.69) {\tiny$\theta = 0.7 $};
\end{scope}
\end{tikzpicture}
% Created by tikzDevice version 0.12.6 on 2025-09-10 23:29:26
% !TEX encoding = UTF-8 Unicode
\begin{tikzpicture}[x=1pt,y=1pt]
\definecolor{fillColor}{RGB}{255,255,255}
\path[use as bounding box,fill=fillColor,fill opacity=0.00] (0,0) rectangle (101.18,108.41);
\begin{scope}
\path[clip] (  0.00,  0.00) rectangle (101.18,108.41);
\definecolor{drawColor}{RGB}{0,0,0}

\path[draw=drawColor,line width= 0.4pt,line join=round,line cap=round] ( 26.32, 18.60) -- ( 96.96, 18.60);

\path[draw=drawColor,line width= 0.4pt,line join=round,line cap=round] ( 26.32, 18.60) -- ( 26.32, 17.40);

\path[draw=drawColor,line width= 0.4pt,line join=round,line cap=round] ( 49.87, 18.60) -- ( 49.87, 17.40);

\path[draw=drawColor,line width= 0.4pt,line join=round,line cap=round] ( 73.42, 18.60) -- ( 73.42, 17.40);

\path[draw=drawColor,line width= 0.4pt,line join=round,line cap=round] ( 96.96, 18.60) -- ( 96.96, 17.40);

\node[text=drawColor,anchor=base,inner sep=0pt, outer sep=0pt, scale=  0.80] at ( 26.32,  9.60) {-5};

\node[text=drawColor,anchor=base,inner sep=0pt, outer sep=0pt, scale=  0.80] at ( 49.87,  9.60) {0};

\node[text=drawColor,anchor=base,inner sep=0pt, outer sep=0pt, scale=  0.80] at ( 73.42,  9.60) {5};

\node[text=drawColor,anchor=base,inner sep=0pt, outer sep=0pt, scale=  0.80] at ( 96.96,  9.60) {10};

\path[draw=drawColor,line width= 0.4pt,line join=round,line cap=round] ( 18.60, 21.24) -- ( 18.60, 87.17);

\path[draw=drawColor,line width= 0.4pt,line join=round,line cap=round] ( 18.60, 21.24) -- ( 17.40, 21.24);

\path[draw=drawColor,line width= 0.4pt,line join=round,line cap=round] ( 18.60, 34.42) -- ( 17.40, 34.42);

\path[draw=drawColor,line width= 0.4pt,line join=round,line cap=round] ( 18.60, 47.61) -- ( 17.40, 47.61);

\path[draw=drawColor,line width= 0.4pt,line join=round,line cap=round] ( 18.60, 60.80) -- ( 17.40, 60.80);

\path[draw=drawColor,line width= 0.4pt,line join=round,line cap=round] ( 18.60, 73.98) -- ( 17.40, 73.98);

\path[draw=drawColor,line width= 0.4pt,line join=round,line cap=round] ( 18.60, 87.17) -- ( 17.40, 87.17);

\node[text=drawColor,rotate= 90.00,anchor=base,inner sep=0pt, outer sep=0pt, scale=  0.80] at ( 16.80, 21.24) {0.0};

\node[text=drawColor,rotate= 90.00,anchor=base,inner sep=0pt, outer sep=0pt, scale=  0.80] at ( 16.80, 47.61) {0.4};

\node[text=drawColor,rotate= 90.00,anchor=base,inner sep=0pt, outer sep=0pt, scale=  0.80] at ( 16.80, 73.98) {0.8};
\end{scope}
\begin{scope}
\path[clip] ( 18.60, 18.60) rectangle ( 99.98, 89.81);
\definecolor{drawColor}{RGB}{211,211,211}

\path[draw=drawColor,line width= 0.4pt,dash pattern=on 1pt off 3pt ,line join=round,line cap=round] ( 18.60, 21.24) -- ( 99.98, 21.24);

\path[draw=drawColor,line width= 0.4pt,dash pattern=on 1pt off 3pt ,line join=round,line cap=round] ( 18.60, 34.42) -- ( 99.98, 34.42);

\path[draw=drawColor,line width= 0.4pt,dash pattern=on 1pt off 3pt ,line join=round,line cap=round] ( 18.60, 47.61) -- ( 99.98, 47.61);

\path[draw=drawColor,line width= 0.4pt,dash pattern=on 1pt off 3pt ,line join=round,line cap=round] ( 18.60, 60.80) -- ( 99.98, 60.80);

\path[draw=drawColor,line width= 0.4pt,dash pattern=on 1pt off 3pt ,line join=round,line cap=round] ( 18.60, 73.98) -- ( 99.98, 73.98);

\path[draw=drawColor,line width= 0.4pt,dash pattern=on 1pt off 3pt ,line join=round,line cap=round] ( 18.60, 87.17) -- ( 99.98, 87.17);

\path[draw=drawColor,line width= 0.4pt,dash pattern=on 1pt off 3pt ,line join=round,line cap=round] ( 26.32, 18.60) -- ( 26.32, 89.81);

\path[draw=drawColor,line width= 0.4pt,dash pattern=on 1pt off 3pt ,line join=round,line cap=round] ( 49.87, 18.60) -- ( 49.87, 89.81);

\path[draw=drawColor,line width= 0.4pt,dash pattern=on 1pt off 3pt ,line join=round,line cap=round] ( 73.42, 18.60) -- ( 73.42, 89.81);

\path[draw=drawColor,line width= 0.4pt,dash pattern=on 1pt off 3pt ,line join=round,line cap=round] ( 96.96, 18.60) -- ( 96.96, 89.81);
\end{scope}
\begin{scope}
\path[clip] (  0.00,  0.00) rectangle (101.18,108.41);
\definecolor{drawColor}{RGB}{0,0,0}

\node[text=drawColor,anchor=base,inner sep=0pt, outer sep=0pt, scale=  1.00] at ( 59.29,  0.60) {\footnotesize$t$};

\node[text=drawColor,rotate= 90.00,anchor=base,inner sep=0pt, outer sep=0pt, scale=  1.00] at (  6.60, 54.20) {\footnotesize power};
\end{scope}
\begin{scope}
\path[clip] (  0.00,  0.00) rectangle (101.18,108.41);
\definecolor{drawColor}{RGB}{0,0,0}

\node[text=drawColor,anchor=base,inner sep=0pt, outer sep=0pt, scale=  1.00] at ( 59.29, 92.21) {\footnotesize $P= 0.326 $};
\end{scope}
\begin{scope}
\path[clip] ( 18.60, 18.60) rectangle ( 99.98, 89.81);
\definecolor{drawColor}{RGB}{97,208,79}

\path[draw=drawColor,line width= 0.4pt,line join=round,line cap=round] ( 21.61, 80.83) --
	( 23.97, 78.76) --
	( 26.32, 77.12) --
	( 28.68, 74.19) --
	( 31.03, 71.94) --
	( 33.39, 67.81) --
	( 35.74, 64.47) --
	( 38.10, 60.36) --
	( 40.45, 56.84) --
	( 42.81, 51.28) --
	( 45.16, 46.66) --
	( 47.52, 41.49) --
	( 49.87, 36.13) --
	( 52.22, 31.90) --
	( 54.58, 28.06) --
	( 56.93, 25.80) --
	( 59.29, 24.28) --
	( 61.64, 24.89) --
	( 64.00, 26.95) --
	( 66.35, 31.19) --
	( 68.71, 39.32) --
	( 71.06, 50.99) --
	( 73.42, 64.74) --
	( 75.77, 76.28) --
	( 78.13, 83.15) --
	( 80.48, 85.98) --
	( 82.84, 86.96) --
	( 85.19, 87.13) --
	( 87.55, 87.16) --
	( 89.90, 87.17) --
	( 92.25, 87.17) --
	( 94.61, 87.17) --
	( 96.96, 87.17);
\definecolor{drawColor}{RGB}{34,151,230}

\path[draw=drawColor,line width= 0.4pt,line join=round,line cap=round] ( 21.61, 81.89) --
	( 23.97, 80.35) --
	( 26.32, 78.91) --
	( 28.68, 76.79) --
	( 31.03, 74.79) --
	( 33.39, 71.17) --
	( 35.74, 67.98) --
	( 38.10, 64.26) --
	( 40.45, 60.97) --
	( 42.81, 55.74) --
	( 45.16, 50.29) --
	( 47.52, 45.18) --
	( 49.87, 39.24) --
	( 52.22, 34.05) --
	( 54.58, 29.41) --
	( 56.93, 26.46) --
	( 59.29, 24.22) --
	( 61.64, 24.55) --
	( 64.00, 25.66) --
	( 66.35, 29.32) --
	( 68.71, 36.63) --
	( 71.06, 48.70) --
	( 73.42, 63.42) --
	( 75.77, 75.93) --
	( 78.13, 83.09) --
	( 80.48, 85.98) --
	( 82.84, 86.96) --
	( 85.19, 87.13) --
	( 87.55, 87.16) --
	( 89.90, 87.17) --
	( 92.25, 87.17) --
	( 94.61, 87.17) --
	( 96.96, 87.17);
\definecolor{drawColor}{RGB}{40,226,229}

\path[draw=drawColor,line width= 0.4pt,line join=round,line cap=round] ( 21.61, 81.46) --
	( 23.97, 79.84) --
	( 26.32, 78.30) --
	( 28.68, 76.19) --
	( 31.03, 74.17) --
	( 33.39, 70.57) --
	( 35.74, 67.36) --
	( 38.10, 63.65) --
	( 40.45, 60.36) --
	( 42.81, 55.24) --
	( 45.16, 49.81) --
	( 47.52, 44.70) --
	( 49.87, 38.89) --
	( 52.22, 33.85) --
	( 54.58, 29.35) --
	( 56.93, 26.56) --
	( 59.29, 24.51) --
	( 61.64, 24.84) --
	( 64.00, 26.12) --
	( 66.35, 29.79) --
	( 68.71, 37.17) --
	( 71.06, 49.01) --
	( 73.42, 63.64) --
	( 75.77, 75.97) --
	( 78.13, 83.09) --
	( 80.48, 85.98) --
	( 82.84, 86.96) --
	( 85.19, 87.13) --
	( 87.55, 87.16) --
	( 89.90, 87.17) --
	( 92.25, 87.17) --
	( 94.61, 87.17) --
	( 96.96, 87.17);
\definecolor{drawColor}{RGB}{0,0,0}

\path[draw=drawColor,line width= 0.4pt,dash pattern=on 4pt off 4pt ,line join=round,line cap=round] ( 18.60, 24.53) -- ( 99.98, 24.53);

\path[draw=drawColor,line width= 0.4pt,dash pattern=on 4pt off 4pt ,line join=round,line cap=round] ( 59.29, 18.60) -- ( 59.29, 89.81);

\node[text=drawColor,rotate= 90.00,anchor=base,inner sep=0pt, outer sep=0pt, scale=  1.00] at ( 66.50, 70.69) {\tiny$\theta = 2 $};
\end{scope}
\end{tikzpicture}
% Created by tikzDevice version 0.12.6 on 2025-09-10 23:29:26
% !TEX encoding = UTF-8 Unicode
\begin{tikzpicture}[x=1pt,y=1pt]
\definecolor{fillColor}{RGB}{255,255,255}
\path[use as bounding box,fill=fillColor,fill opacity=0.00] (0,0) rectangle (101.18,108.41);
\begin{scope}
\path[clip] (  0.00,  0.00) rectangle (101.18,108.41);
\definecolor{drawColor}{RGB}{0,0,0}

\path[draw=drawColor,line width= 0.4pt,line join=round,line cap=round] ( 31.03, 18.60) -- ( 78.13, 18.60);

\path[draw=drawColor,line width= 0.4pt,line join=round,line cap=round] ( 31.03, 18.60) -- ( 31.03, 17.40);

\path[draw=drawColor,line width= 0.4pt,line join=round,line cap=round] ( 54.58, 18.60) -- ( 54.58, 17.40);

\path[draw=drawColor,line width= 0.4pt,line join=round,line cap=round] ( 78.13, 18.60) -- ( 78.13, 17.40);

\node[text=drawColor,anchor=base,inner sep=0pt, outer sep=0pt, scale=  0.80] at ( 31.03,  9.60) {-10};

\node[text=drawColor,anchor=base,inner sep=0pt, outer sep=0pt, scale=  0.80] at ( 54.58,  9.60) {-5};

\node[text=drawColor,anchor=base,inner sep=0pt, outer sep=0pt, scale=  0.80] at ( 78.13,  9.60) {0};

\path[draw=drawColor,line width= 0.4pt,line join=round,line cap=round] ( 18.60, 21.24) -- ( 18.60, 87.17);

\path[draw=drawColor,line width= 0.4pt,line join=round,line cap=round] ( 18.60, 21.24) -- ( 17.40, 21.24);

\path[draw=drawColor,line width= 0.4pt,line join=round,line cap=round] ( 18.60, 34.42) -- ( 17.40, 34.42);

\path[draw=drawColor,line width= 0.4pt,line join=round,line cap=round] ( 18.60, 47.61) -- ( 17.40, 47.61);

\path[draw=drawColor,line width= 0.4pt,line join=round,line cap=round] ( 18.60, 60.80) -- ( 17.40, 60.80);

\path[draw=drawColor,line width= 0.4pt,line join=round,line cap=round] ( 18.60, 73.98) -- ( 17.40, 73.98);

\path[draw=drawColor,line width= 0.4pt,line join=round,line cap=round] ( 18.60, 87.17) -- ( 17.40, 87.17);

\node[text=drawColor,rotate= 90.00,anchor=base,inner sep=0pt, outer sep=0pt, scale=  0.80] at ( 16.80, 21.24) {0.0};

\node[text=drawColor,rotate= 90.00,anchor=base,inner sep=0pt, outer sep=0pt, scale=  0.80] at ( 16.80, 47.61) {0.4};

\node[text=drawColor,rotate= 90.00,anchor=base,inner sep=0pt, outer sep=0pt, scale=  0.80] at ( 16.80, 73.98) {0.8};
\end{scope}
\begin{scope}
\path[clip] ( 18.60, 18.60) rectangle ( 99.98, 89.81);
\definecolor{drawColor}{RGB}{211,211,211}

\path[draw=drawColor,line width= 0.4pt,dash pattern=on 1pt off 3pt ,line join=round,line cap=round] ( 18.60, 21.24) -- ( 99.98, 21.24);

\path[draw=drawColor,line width= 0.4pt,dash pattern=on 1pt off 3pt ,line join=round,line cap=round] ( 18.60, 34.42) -- ( 99.98, 34.42);

\path[draw=drawColor,line width= 0.4pt,dash pattern=on 1pt off 3pt ,line join=round,line cap=round] ( 18.60, 47.61) -- ( 99.98, 47.61);

\path[draw=drawColor,line width= 0.4pt,dash pattern=on 1pt off 3pt ,line join=round,line cap=round] ( 18.60, 60.80) -- ( 99.98, 60.80);

\path[draw=drawColor,line width= 0.4pt,dash pattern=on 1pt off 3pt ,line join=round,line cap=round] ( 18.60, 73.98) -- ( 99.98, 73.98);

\path[draw=drawColor,line width= 0.4pt,dash pattern=on 1pt off 3pt ,line join=round,line cap=round] ( 18.60, 87.17) -- ( 99.98, 87.17);

\path[draw=drawColor,line width= 0.4pt,dash pattern=on 1pt off 3pt ,line join=round,line cap=round] ( 31.03, 18.60) -- ( 31.03, 89.81);

\path[draw=drawColor,line width= 0.4pt,dash pattern=on 1pt off 3pt ,line join=round,line cap=round] ( 54.58, 18.60) -- ( 54.58, 89.81);

\path[draw=drawColor,line width= 0.4pt,dash pattern=on 1pt off 3pt ,line join=round,line cap=round] ( 78.13, 18.60) -- ( 78.13, 89.81);
\end{scope}
\begin{scope}
\path[clip] (  0.00,  0.00) rectangle (101.18,108.41);
\definecolor{drawColor}{RGB}{0,0,0}

\node[text=drawColor,anchor=base,inner sep=0pt, outer sep=0pt, scale=  1.00] at ( 59.29,  0.60) {\footnotesize$t$};

\node[text=drawColor,rotate= 90.00,anchor=base,inner sep=0pt, outer sep=0pt, scale=  1.00] at (  6.60, 54.20) {\footnotesize power};
\end{scope}
\begin{scope}
\path[clip] (  0.00,  0.00) rectangle (101.18,108.41);
\definecolor{drawColor}{RGB}{0,0,0}

\node[text=drawColor,anchor=base,inner sep=0pt, outer sep=0pt, scale=  1.00] at ( 59.29, 92.21) {\footnotesize $P= 1.72\times 10^{-12} $};
\end{scope}
\begin{scope}
\path[clip] ( 18.60, 18.60) rectangle ( 99.98, 89.81);
\definecolor{drawColor}{RGB}{97,208,79}

\path[draw=drawColor,line width= 0.4pt,line join=round,line cap=round] ( 21.61, 68.44) --
	( 23.97, 64.84) --
	( 26.32, 61.75) --
	( 28.68, 58.34) --
	( 31.03, 54.87) --
	( 33.39, 50.91) --
	( 35.74, 47.23) --
	( 38.10, 43.13) --
	( 40.45, 39.97) --
	( 42.81, 36.45) --
	( 45.16, 33.67) --
	( 47.52, 30.93) --
	( 49.87, 28.32) --
	( 52.22, 26.90) --
	( 54.58, 25.60) --
	( 56.93, 25.11) --
	( 59.29, 24.70) --
	( 61.64, 24.47) --
	( 64.00, 25.76) --
	( 66.35, 26.16) --
	( 68.71, 28.17) --
	( 71.06, 30.62) --
	( 73.42, 34.02) --
	( 75.77, 37.75) --
	( 78.13, 41.95) --
	( 80.48, 48.62) --
	( 82.84, 54.69) --
	( 85.19, 61.39) --
	( 87.55, 67.51) --
	( 89.90, 74.23) --
	( 92.25, 79.03) --
	( 94.61, 82.92) --
	( 96.96, 85.33);
\definecolor{drawColor}{RGB}{34,151,230}

\path[draw=drawColor,line width= 0.4pt,line join=round,line cap=round] ( 21.61, 82.43) --
	( 23.97, 81.53) --
	( 26.32, 80.11) --
	( 28.68, 78.70) --
	( 31.03, 76.92) --
	( 33.39, 74.38) --
	( 35.74, 72.38) --
	( 38.10, 68.46) --
	( 40.45, 65.62) --
	( 42.81, 62.68) --
	( 45.16, 58.76) --
	( 47.52, 55.06) --
	( 49.87, 50.63) --
	( 52.22, 45.87) --
	( 54.58, 41.97) --
	( 56.93, 37.59) --
	( 59.29, 34.30) --
	( 61.64, 30.63) --
	( 64.00, 27.76) --
	( 66.35, 25.79) --
	( 68.71, 24.49) --
	( 71.06, 23.53) --
	( 73.42, 23.78) --
	( 75.77, 24.13) --
	( 78.13, 24.71) --
	( 80.48, 26.53) --
	( 82.84, 28.96) --
	( 85.19, 32.73) --
	( 87.55, 37.00) --
	( 89.90, 44.77) --
	( 92.25, 54.00) --
	( 94.61, 64.51) --
	( 96.96, 75.18);
\definecolor{drawColor}{RGB}{40,226,229}

\path[draw=drawColor,line width= 0.4pt,line join=round,line cap=round] ( 21.61, 77.43) --
	( 23.97, 76.07) --
	( 26.32, 75.12) --
	( 28.68, 72.68) --
	( 31.03, 71.20) --
	( 33.39, 68.52) --
	( 35.74, 66.78) --
	( 38.10, 63.20) --
	( 40.45, 60.48) --
	( 42.81, 58.09) --
	( 45.16, 54.98) --
	( 47.52, 51.94) --
	( 49.87, 48.22) --
	( 52.22, 44.56) --
	( 54.58, 41.42) --
	( 56.93, 37.81) --
	( 59.29, 35.76) --
	( 61.64, 32.66) --
	( 64.00, 30.57) --
	( 66.35, 29.00) --
	( 68.71, 28.54) --
	( 71.06, 27.71) --
	( 73.42, 28.56) --
	( 75.77, 29.52) --
	( 78.13, 30.89) --
	( 80.48, 33.41) --
	( 82.84, 35.94) --
	( 85.19, 40.35) --
	( 87.55, 44.37) --
	( 89.90, 51.26) --
	( 92.25, 58.95) --
	( 94.61, 67.51) --
	( 96.96, 76.05);
\definecolor{drawColor}{RGB}{0,0,0}

\path[draw=drawColor,line width= 0.4pt,dash pattern=on 4pt off 4pt ,line join=round,line cap=round] ( 18.60, 24.53) -- ( 99.98, 24.53);

\path[draw=drawColor,line width= 0.4pt,dash pattern=on 4pt off 4pt ,line join=round,line cap=round] ( 59.29, 18.60) -- ( 59.29, 89.81);

\node[text=drawColor,rotate= 90.00,anchor=base,inner sep=0pt, outer sep=0pt, scale=  1.00] at ( 66.50, 70.69) {\tiny$\theta = -4 $};
\end{scope}
\end{tikzpicture}
% Created by tikzDevice version 0.12.6 on 2025-09-10 23:29:26
% !TEX encoding = UTF-8 Unicode
\begin{tikzpicture}[x=1pt,y=1pt]
\definecolor{fillColor}{RGB}{255,255,255}
\path[use as bounding box,fill=fillColor,fill opacity=0.00] (0,0) rectangle (101.18,108.41);
\begin{scope}
\path[clip] (  0.00,  0.00) rectangle (101.18,108.41);
\definecolor{drawColor}{RGB}{0,0,0}

\path[draw=drawColor,line width= 0.4pt,line join=round,line cap=round] ( 42.02, 18.60) -- ( 89.12, 18.60);

\path[draw=drawColor,line width= 0.4pt,line join=round,line cap=round] ( 42.02, 18.60) -- ( 42.02, 17.40);

\path[draw=drawColor,line width= 0.4pt,line join=round,line cap=round] ( 65.57, 18.60) -- ( 65.57, 17.40);

\path[draw=drawColor,line width= 0.4pt,line join=round,line cap=round] ( 89.12, 18.60) -- ( 89.12, 17.40);

\node[text=drawColor,anchor=base,inner sep=0pt, outer sep=0pt, scale=  0.80] at ( 42.02,  9.60) {-5};

\node[text=drawColor,anchor=base,inner sep=0pt, outer sep=0pt, scale=  0.80] at ( 65.57,  9.60) {0};

\node[text=drawColor,anchor=base,inner sep=0pt, outer sep=0pt, scale=  0.80] at ( 89.12,  9.60) {5};

\path[draw=drawColor,line width= 0.4pt,line join=round,line cap=round] ( 18.60, 21.24) -- ( 18.60, 87.17);

\path[draw=drawColor,line width= 0.4pt,line join=round,line cap=round] ( 18.60, 21.24) -- ( 17.40, 21.24);

\path[draw=drawColor,line width= 0.4pt,line join=round,line cap=round] ( 18.60, 34.42) -- ( 17.40, 34.42);

\path[draw=drawColor,line width= 0.4pt,line join=round,line cap=round] ( 18.60, 47.61) -- ( 17.40, 47.61);

\path[draw=drawColor,line width= 0.4pt,line join=round,line cap=round] ( 18.60, 60.80) -- ( 17.40, 60.80);

\path[draw=drawColor,line width= 0.4pt,line join=round,line cap=round] ( 18.60, 73.98) -- ( 17.40, 73.98);

\path[draw=drawColor,line width= 0.4pt,line join=round,line cap=round] ( 18.60, 87.17) -- ( 17.40, 87.17);

\node[text=drawColor,rotate= 90.00,anchor=base,inner sep=0pt, outer sep=0pt, scale=  0.80] at ( 16.80, 21.24) {0.0};

\node[text=drawColor,rotate= 90.00,anchor=base,inner sep=0pt, outer sep=0pt, scale=  0.80] at ( 16.80, 47.61) {0.4};

\node[text=drawColor,rotate= 90.00,anchor=base,inner sep=0pt, outer sep=0pt, scale=  0.80] at ( 16.80, 73.98) {0.8};
\end{scope}
\begin{scope}
\path[clip] ( 18.60, 18.60) rectangle ( 99.98, 89.81);
\definecolor{drawColor}{RGB}{211,211,211}

\path[draw=drawColor,line width= 0.4pt,dash pattern=on 1pt off 3pt ,line join=round,line cap=round] ( 18.60, 21.24) -- ( 99.98, 21.24);

\path[draw=drawColor,line width= 0.4pt,dash pattern=on 1pt off 3pt ,line join=round,line cap=round] ( 18.60, 34.42) -- ( 99.98, 34.42);

\path[draw=drawColor,line width= 0.4pt,dash pattern=on 1pt off 3pt ,line join=round,line cap=round] ( 18.60, 47.61) -- ( 99.98, 47.61);

\path[draw=drawColor,line width= 0.4pt,dash pattern=on 1pt off 3pt ,line join=round,line cap=round] ( 18.60, 60.80) -- ( 99.98, 60.80);

\path[draw=drawColor,line width= 0.4pt,dash pattern=on 1pt off 3pt ,line join=round,line cap=round] ( 18.60, 73.98) -- ( 99.98, 73.98);

\path[draw=drawColor,line width= 0.4pt,dash pattern=on 1pt off 3pt ,line join=round,line cap=round] ( 18.60, 87.17) -- ( 99.98, 87.17);

\path[draw=drawColor,line width= 0.4pt,dash pattern=on 1pt off 3pt ,line join=round,line cap=round] ( 42.02, 18.60) -- ( 42.02, 89.81);

\path[draw=drawColor,line width= 0.4pt,dash pattern=on 1pt off 3pt ,line join=round,line cap=round] ( 65.57, 18.60) -- ( 65.57, 89.81);

\path[draw=drawColor,line width= 0.4pt,dash pattern=on 1pt off 3pt ,line join=round,line cap=round] ( 89.12, 18.60) -- ( 89.12, 89.81);
\end{scope}
\begin{scope}
\path[clip] (  0.00,  0.00) rectangle (101.18,108.41);
\definecolor{drawColor}{RGB}{0,0,0}

\node[text=drawColor,anchor=base,inner sep=0pt, outer sep=0pt, scale=  1.00] at ( 59.29,  0.60) {\footnotesize$t$};

\node[text=drawColor,rotate= 90.00,anchor=base,inner sep=0pt, outer sep=0pt, scale=  1.00] at (  6.60, 54.20) {\footnotesize power};
\end{scope}
\begin{scope}
\path[clip] (  0.00,  0.00) rectangle (101.18,108.41);
\definecolor{drawColor}{RGB}{0,0,0}

\node[text=drawColor,anchor=base,inner sep=0pt, outer sep=0pt, scale=  1.00] at ( 59.29, 92.21) {\footnotesize $P= 3.52\times 10^{-6} $};
\end{scope}
\begin{scope}
\path[clip] ( 18.60, 18.60) rectangle ( 99.98, 89.81);
\definecolor{drawColor}{RGB}{97,208,79}

\path[draw=drawColor,line width= 0.4pt,line join=round,line cap=round] ( 21.61, 74.31) --
	( 23.97, 71.27) --
	( 26.32, 68.03) --
	( 28.68, 64.09) --
	( 31.03, 60.91) --
	( 33.39, 56.42) --
	( 35.74, 52.83) --
	( 38.10, 48.06) --
	( 40.45, 44.46) --
	( 42.81, 40.03) --
	( 45.16, 36.55) --
	( 47.52, 33.16) --
	( 49.87, 30.45) --
	( 52.22, 28.17) --
	( 54.58, 26.21) --
	( 56.93, 24.95) --
	( 59.29, 24.64) --
	( 61.64, 24.82) --
	( 64.00, 25.48) --
	( 66.35, 27.28) --
	( 68.71, 29.85) --
	( 71.06, 33.46) --
	( 73.42, 38.32) --
	( 75.77, 44.76) --
	( 78.13, 51.60) --
	( 80.48, 59.41) --
	( 82.84, 68.21) --
	( 85.19, 75.55) --
	( 87.55, 81.79) --
	( 89.90, 85.02) --
	( 92.25, 86.71) --
	( 94.61, 87.11) --
	( 96.96, 87.16);
\definecolor{drawColor}{RGB}{34,151,230}

\path[draw=drawColor,line width= 0.4pt,line join=round,line cap=round] ( 21.61, 83.85) --
	( 23.97, 82.91) --
	( 26.32, 81.78) --
	( 28.68, 80.59) --
	( 31.03, 78.50) --
	( 33.39, 76.28) --
	( 35.74, 73.28) --
	( 38.10, 69.84) --
	( 40.45, 67.14) --
	( 42.81, 62.94) --
	( 45.16, 58.69) --
	( 47.52, 53.58) --
	( 49.87, 48.80) --
	( 52.22, 43.32) --
	( 54.58, 38.33) --
	( 56.93, 33.93) --
	( 59.29, 29.71) --
	( 61.64, 26.68) --
	( 64.00, 24.61) --
	( 66.35, 23.29) --
	( 68.71, 22.96) --
	( 71.06, 23.25) --
	( 73.42, 24.26) --
	( 75.77, 26.43) --
	( 78.13, 30.23) --
	( 80.48, 34.87) --
	( 82.84, 44.15) --
	( 85.19, 56.89) --
	( 87.55, 70.84) --
	( 89.90, 81.14) --
	( 92.25, 86.05) --
	( 94.61, 87.03) --
	( 96.96, 87.16);
\definecolor{drawColor}{RGB}{40,226,229}

\path[draw=drawColor,line width= 0.4pt,line join=round,line cap=round] ( 21.61, 80.80) --
	( 23.97, 79.24) --
	( 26.32, 77.86) --
	( 28.68, 76.66) --
	( 31.03, 74.34) --
	( 33.39, 72.18) --
	( 35.74, 69.68) --
	( 38.10, 66.29) --
	( 40.45, 63.64) --
	( 42.81, 60.31) --
	( 45.16, 56.56) --
	( 47.52, 52.14) --
	( 49.87, 48.36) --
	( 52.22, 43.26) --
	( 54.58, 38.64) --
	( 56.93, 35.27) --
	( 59.29, 31.26) --
	( 61.64, 28.68) --
	( 64.00, 26.91) --
	( 66.35, 26.20) --
	( 68.71, 26.12) --
	( 71.06, 26.83) --
	( 73.42, 28.39) --
	( 75.77, 31.14) --
	( 78.13, 35.21) --
	( 80.48, 39.80) --
	( 82.84, 47.88) --
	( 85.19, 59.07) --
	( 87.55, 71.67) --
	( 89.90, 81.05) --
	( 92.25, 85.98) --
	( 94.61, 87.02) --
	( 96.96, 87.16);
\definecolor{drawColor}{RGB}{0,0,0}

\path[draw=drawColor,line width= 0.4pt,dash pattern=on 4pt off 4pt ,line join=round,line cap=round] ( 18.60, 24.53) -- ( 99.98, 24.53);

\path[draw=drawColor,line width= 0.4pt,dash pattern=on 4pt off 4pt ,line join=round,line cap=round] ( 59.45, 18.60) -- ( 59.45, 89.81);

\node[text=drawColor,rotate= 90.00,anchor=base,inner sep=0pt, outer sep=0pt, scale=  1.00] at ( 66.65, 70.69) {\tiny$\theta = -1.3 $};
\end{scope}
\end{tikzpicture}
% Created by tikzDevice version 0.12.6 on 2025-09-10 23:29:27
% !TEX encoding = UTF-8 Unicode
\begin{tikzpicture}[x=1pt,y=1pt]
\definecolor{fillColor}{RGB}{255,255,255}
\path[use as bounding box,fill=fillColor,fill opacity=0.00] (0,0) rectangle (101.18,108.41);
\begin{scope}
\path[clip] (  0.00,  0.00) rectangle (101.18,108.41);
\definecolor{drawColor}{RGB}{0,0,0}

\path[draw=drawColor,line width= 0.4pt,line join=round,line cap=round] ( 29.46, 18.60) -- ( 76.56, 18.60);

\path[draw=drawColor,line width= 0.4pt,line join=round,line cap=round] ( 29.46, 18.60) -- ( 29.46, 17.40);

\path[draw=drawColor,line width= 0.4pt,line join=round,line cap=round] ( 53.01, 18.60) -- ( 53.01, 17.40);

\path[draw=drawColor,line width= 0.4pt,line join=round,line cap=round] ( 76.56, 18.60) -- ( 76.56, 17.40);

\node[text=drawColor,anchor=base,inner sep=0pt, outer sep=0pt, scale=  0.80] at ( 29.46,  9.60) {-5};

\node[text=drawColor,anchor=base,inner sep=0pt, outer sep=0pt, scale=  0.80] at ( 53.01,  9.60) {0};

\node[text=drawColor,anchor=base,inner sep=0pt, outer sep=0pt, scale=  0.80] at ( 76.56,  9.60) {5};

\path[draw=drawColor,line width= 0.4pt,line join=round,line cap=round] ( 18.60, 21.24) -- ( 18.60, 87.17);

\path[draw=drawColor,line width= 0.4pt,line join=round,line cap=round] ( 18.60, 21.24) -- ( 17.40, 21.24);

\path[draw=drawColor,line width= 0.4pt,line join=round,line cap=round] ( 18.60, 34.42) -- ( 17.40, 34.42);

\path[draw=drawColor,line width= 0.4pt,line join=round,line cap=round] ( 18.60, 47.61) -- ( 17.40, 47.61);

\path[draw=drawColor,line width= 0.4pt,line join=round,line cap=round] ( 18.60, 60.80) -- ( 17.40, 60.80);

\path[draw=drawColor,line width= 0.4pt,line join=round,line cap=round] ( 18.60, 73.98) -- ( 17.40, 73.98);

\path[draw=drawColor,line width= 0.4pt,line join=round,line cap=round] ( 18.60, 87.17) -- ( 17.40, 87.17);

\node[text=drawColor,rotate= 90.00,anchor=base,inner sep=0pt, outer sep=0pt, scale=  0.80] at ( 16.80, 21.24) {0.0};

\node[text=drawColor,rotate= 90.00,anchor=base,inner sep=0pt, outer sep=0pt, scale=  0.80] at ( 16.80, 47.61) {0.4};

\node[text=drawColor,rotate= 90.00,anchor=base,inner sep=0pt, outer sep=0pt, scale=  0.80] at ( 16.80, 73.98) {0.8};
\end{scope}
\begin{scope}
\path[clip] ( 18.60, 18.60) rectangle ( 99.98, 89.81);
\definecolor{drawColor}{RGB}{211,211,211}

\path[draw=drawColor,line width= 0.4pt,dash pattern=on 1pt off 3pt ,line join=round,line cap=round] ( 18.60, 21.24) -- ( 99.98, 21.24);

\path[draw=drawColor,line width= 0.4pt,dash pattern=on 1pt off 3pt ,line join=round,line cap=round] ( 18.60, 34.42) -- ( 99.98, 34.42);

\path[draw=drawColor,line width= 0.4pt,dash pattern=on 1pt off 3pt ,line join=round,line cap=round] ( 18.60, 47.61) -- ( 99.98, 47.61);

\path[draw=drawColor,line width= 0.4pt,dash pattern=on 1pt off 3pt ,line join=round,line cap=round] ( 18.60, 60.80) -- ( 99.98, 60.80);

\path[draw=drawColor,line width= 0.4pt,dash pattern=on 1pt off 3pt ,line join=round,line cap=round] ( 18.60, 73.98) -- ( 99.98, 73.98);

\path[draw=drawColor,line width= 0.4pt,dash pattern=on 1pt off 3pt ,line join=round,line cap=round] ( 18.60, 87.17) -- ( 99.98, 87.17);

\path[draw=drawColor,line width= 0.4pt,dash pattern=on 1pt off 3pt ,line join=round,line cap=round] ( 29.46, 18.60) -- ( 29.46, 89.81);

\path[draw=drawColor,line width= 0.4pt,dash pattern=on 1pt off 3pt ,line join=round,line cap=round] ( 53.01, 18.60) -- ( 53.01, 89.81);

\path[draw=drawColor,line width= 0.4pt,dash pattern=on 1pt off 3pt ,line join=round,line cap=round] ( 76.56, 18.60) -- ( 76.56, 89.81);
\end{scope}
\begin{scope}
\path[clip] (  0.00,  0.00) rectangle (101.18,108.41);
\definecolor{drawColor}{RGB}{0,0,0}

\node[text=drawColor,anchor=base,inner sep=0pt, outer sep=0pt, scale=  1.00] at ( 59.29,  0.60) {\footnotesize$t$};

\node[text=drawColor,rotate= 90.00,anchor=base,inner sep=0pt, outer sep=0pt, scale=  1.00] at (  6.60, 54.20) {\footnotesize power};
\end{scope}
\begin{scope}
\path[clip] (  0.00,  0.00) rectangle (101.18,108.41);
\definecolor{drawColor}{RGB}{0,0,0}

\node[text=drawColor,anchor=base,inner sep=0pt, outer sep=0pt, scale=  1.00] at ( 59.29, 92.21) {\footnotesize $P= 0.0148 $};
\end{scope}
\begin{scope}
\path[clip] ( 18.60, 18.60) rectangle ( 99.98, 89.81);
\definecolor{drawColor}{RGB}{97,208,79}

\path[draw=drawColor,line width= 0.4pt,line join=round,line cap=round] ( 21.61, 79.91) --
	( 23.97, 78.13) --
	( 26.32, 75.79) --
	( 28.68, 72.91) --
	( 31.03, 69.53) --
	( 33.39, 66.18) --
	( 35.74, 61.86) --
	( 38.10, 56.73) --
	( 40.45, 52.11) --
	( 42.81, 46.82) --
	( 45.16, 42.44) --
	( 47.52, 37.85) --
	( 49.87, 33.05) --
	( 52.22, 29.95) --
	( 54.58, 26.97) --
	( 56.93, 25.42) --
	( 59.29, 24.60) --
	( 61.64, 24.64) --
	( 64.00, 26.18) --
	( 66.35, 29.10) --
	( 68.71, 34.16) --
	( 71.06, 40.02) --
	( 73.42, 49.98) --
	( 75.77, 59.95) --
	( 78.13, 71.16) --
	( 80.48, 79.80) --
	( 82.84, 84.64) --
	( 85.19, 86.61) --
	( 87.55, 87.06) --
	( 89.90, 87.16) --
	( 92.25, 87.17) --
	( 94.61, 87.17) --
	( 96.96, 87.17);
\definecolor{drawColor}{RGB}{34,151,230}

\path[draw=drawColor,line width= 0.4pt,line join=round,line cap=round] ( 21.61, 85.29) --
	( 23.97, 84.51) --
	( 26.32, 83.86) --
	( 28.68, 82.86) --
	( 31.03, 81.56) --
	( 33.39, 79.23) --
	( 35.74, 77.01) --
	( 38.10, 73.26) --
	( 40.45, 69.42) --
	( 42.81, 64.75) --
	( 45.16, 59.29) --
	( 47.52, 53.60) --
	( 49.87, 46.75) --
	( 52.22, 40.38) --
	( 54.58, 34.44) --
	( 56.93, 29.49) --
	( 59.29, 25.85) --
	( 61.64, 23.42) --
	( 64.00, 22.73) --
	( 66.35, 22.94) --
	( 68.71, 24.29) --
	( 71.06, 27.39) --
	( 73.42, 34.35) --
	( 75.77, 45.62) --
	( 78.13, 61.42) --
	( 80.48, 75.93) --
	( 82.84, 83.79) --
	( 85.19, 86.50) --
	( 87.55, 87.06) --
	( 89.90, 87.16) --
	( 92.25, 87.17) --
	( 94.61, 87.17) --
	( 96.96, 87.17);
\definecolor{drawColor}{RGB}{40,226,229}

\path[draw=drawColor,line width= 0.4pt,line join=round,line cap=round] ( 21.61, 83.73) --
	( 23.97, 82.85) --
	( 26.32, 82.07) --
	( 28.68, 80.84) --
	( 31.03, 79.55) --
	( 33.39, 77.11) --
	( 35.74, 75.05) --
	( 38.10, 71.77) --
	( 40.45, 68.13) --
	( 42.81, 63.72) --
	( 45.16, 58.95) --
	( 47.52, 53.77) --
	( 49.87, 47.02) --
	( 52.22, 41.06) --
	( 54.58, 35.51) --
	( 56.93, 30.88) --
	( 59.29, 26.99) --
	( 61.64, 24.61) --
	( 64.00, 24.22) --
	( 66.35, 24.67) --
	( 68.71, 26.38) --
	( 71.06, 29.70) --
	( 73.42, 36.41) --
	( 75.77, 46.71) --
	( 78.13, 61.75) --
	( 80.48, 75.90) --
	( 82.84, 83.81) --
	( 85.19, 86.46) --
	( 87.55, 87.06) --
	( 89.90, 87.16) --
	( 92.25, 87.17) --
	( 94.61, 87.17) --
	( 96.96, 87.17);
\definecolor{drawColor}{RGB}{0,0,0}

\path[draw=drawColor,line width= 0.4pt,dash pattern=on 4pt off 4pt ,line join=round,line cap=round] ( 18.60, 24.53) -- ( 99.98, 24.53);

\path[draw=drawColor,line width= 0.4pt,dash pattern=on 4pt off 4pt ,line join=round,line cap=round] ( 59.13, 18.60) -- ( 59.13, 89.81);

\node[text=drawColor,rotate= 90.00,anchor=base,inner sep=0pt, outer sep=0pt, scale=  1.00] at ( 66.34, 70.69) {\tiny$\theta = 1.3 $};
\end{scope}
\end{tikzpicture}
% Created by tikzDevice version 0.12.6 on 2025-09-10 23:29:27
% !TEX encoding = UTF-8 Unicode
\begin{tikzpicture}[x=1pt,y=1pt]
\definecolor{fillColor}{RGB}{255,255,255}
\path[use as bounding box,fill=fillColor,fill opacity=0.00] (0,0) rectangle (101.18,108.41);
\begin{scope}
\path[clip] (  0.00,  0.00) rectangle (101.18,108.41);
\definecolor{drawColor}{RGB}{0,0,0}

\path[draw=drawColor,line width= 0.4pt,line join=round,line cap=round] ( 40.45, 18.60) -- ( 87.55, 18.60);

\path[draw=drawColor,line width= 0.4pt,line join=round,line cap=round] ( 40.45, 18.60) -- ( 40.45, 17.40);

\path[draw=drawColor,line width= 0.4pt,line join=round,line cap=round] ( 64.00, 18.60) -- ( 64.00, 17.40);

\path[draw=drawColor,line width= 0.4pt,line join=round,line cap=round] ( 87.55, 18.60) -- ( 87.55, 17.40);

\node[text=drawColor,anchor=base,inner sep=0pt, outer sep=0pt, scale=  0.80] at ( 40.45,  9.60) {0};

\node[text=drawColor,anchor=base,inner sep=0pt, outer sep=0pt, scale=  0.80] at ( 64.00,  9.60) {5};

\node[text=drawColor,anchor=base,inner sep=0pt, outer sep=0pt, scale=  0.80] at ( 87.55,  9.60) {10};

\path[draw=drawColor,line width= 0.4pt,line join=round,line cap=round] ( 18.60, 21.24) -- ( 18.60, 87.17);

\path[draw=drawColor,line width= 0.4pt,line join=round,line cap=round] ( 18.60, 21.24) -- ( 17.40, 21.24);

\path[draw=drawColor,line width= 0.4pt,line join=round,line cap=round] ( 18.60, 34.42) -- ( 17.40, 34.42);

\path[draw=drawColor,line width= 0.4pt,line join=round,line cap=round] ( 18.60, 47.61) -- ( 17.40, 47.61);

\path[draw=drawColor,line width= 0.4pt,line join=round,line cap=round] ( 18.60, 60.80) -- ( 17.40, 60.80);

\path[draw=drawColor,line width= 0.4pt,line join=round,line cap=round] ( 18.60, 73.98) -- ( 17.40, 73.98);

\path[draw=drawColor,line width= 0.4pt,line join=round,line cap=round] ( 18.60, 87.17) -- ( 17.40, 87.17);

\node[text=drawColor,rotate= 90.00,anchor=base,inner sep=0pt, outer sep=0pt, scale=  0.80] at ( 16.80, 21.24) {0.0};

\node[text=drawColor,rotate= 90.00,anchor=base,inner sep=0pt, outer sep=0pt, scale=  0.80] at ( 16.80, 47.61) {0.4};

\node[text=drawColor,rotate= 90.00,anchor=base,inner sep=0pt, outer sep=0pt, scale=  0.80] at ( 16.80, 73.98) {0.8};
\end{scope}
\begin{scope}
\path[clip] ( 18.60, 18.60) rectangle ( 99.98, 89.81);
\definecolor{drawColor}{RGB}{211,211,211}

\path[draw=drawColor,line width= 0.4pt,dash pattern=on 1pt off 3pt ,line join=round,line cap=round] ( 18.60, 21.24) -- ( 99.98, 21.24);

\path[draw=drawColor,line width= 0.4pt,dash pattern=on 1pt off 3pt ,line join=round,line cap=round] ( 18.60, 34.42) -- ( 99.98, 34.42);

\path[draw=drawColor,line width= 0.4pt,dash pattern=on 1pt off 3pt ,line join=round,line cap=round] ( 18.60, 47.61) -- ( 99.98, 47.61);

\path[draw=drawColor,line width= 0.4pt,dash pattern=on 1pt off 3pt ,line join=round,line cap=round] ( 18.60, 60.80) -- ( 99.98, 60.80);

\path[draw=drawColor,line width= 0.4pt,dash pattern=on 1pt off 3pt ,line join=round,line cap=round] ( 18.60, 73.98) -- ( 99.98, 73.98);

\path[draw=drawColor,line width= 0.4pt,dash pattern=on 1pt off 3pt ,line join=round,line cap=round] ( 18.60, 87.17) -- ( 99.98, 87.17);

\path[draw=drawColor,line width= 0.4pt,dash pattern=on 1pt off 3pt ,line join=round,line cap=round] ( 40.45, 18.60) -- ( 40.45, 89.81);

\path[draw=drawColor,line width= 0.4pt,dash pattern=on 1pt off 3pt ,line join=round,line cap=round] ( 64.00, 18.60) -- ( 64.00, 89.81);

\path[draw=drawColor,line width= 0.4pt,dash pattern=on 1pt off 3pt ,line join=round,line cap=round] ( 87.55, 18.60) -- ( 87.55, 89.81);
\end{scope}
\begin{scope}
\path[clip] (  0.00,  0.00) rectangle (101.18,108.41);
\definecolor{drawColor}{RGB}{0,0,0}

\node[text=drawColor,anchor=base,inner sep=0pt, outer sep=0pt, scale=  1.00] at ( 59.29,  0.60) {\footnotesize$t$};

\node[text=drawColor,rotate= 90.00,anchor=base,inner sep=0pt, outer sep=0pt, scale=  1.00] at (  6.60, 54.20) {\footnotesize power};
\end{scope}
\begin{scope}
\path[clip] (  0.00,  0.00) rectangle (101.18,108.41);
\definecolor{drawColor}{RGB}{0,0,0}

\node[text=drawColor,anchor=base,inner sep=0pt, outer sep=0pt, scale=  1.00] at ( 59.29, 92.21) {\footnotesize $P= 0.554 $};
\end{scope}
\begin{scope}
\path[clip] ( 18.60, 18.60) rectangle ( 99.98, 89.81);
\definecolor{drawColor}{RGB}{97,208,79}

\path[draw=drawColor,line width= 0.4pt,line join=round,line cap=round] ( 21.61, 85.35) --
	( 23.97, 84.48) --
	( 26.32, 83.77) --
	( 28.68, 82.12) --
	( 31.03, 80.30) --
	( 33.39, 77.79) --
	( 35.74, 75.14) --
	( 38.10, 70.98) --
	( 40.45, 66.40) --
	( 42.81, 60.47) --
	( 45.16, 54.27) --
	( 47.52, 47.35) --
	( 49.87, 41.00) --
	( 52.22, 34.67) --
	( 54.58, 29.74) --
	( 56.93, 25.62) --
	( 59.29, 24.59) --
	( 61.64, 25.35) --
	( 64.00, 28.51) --
	( 66.35, 35.10) --
	( 68.71, 45.61) --
	( 71.06, 58.24) --
	( 73.42, 71.20) --
	( 75.77, 80.05) --
	( 78.13, 84.85) --
	( 80.48, 86.57) --
	( 82.84, 87.05) --
	( 85.19, 87.15) --
	( 87.55, 87.16) --
	( 89.90, 87.17) --
	( 92.25, 87.17) --
	( 94.61, 87.17) --
	( 96.96, 87.17);
\definecolor{drawColor}{RGB}{34,151,230}

\path[draw=drawColor,line width= 0.4pt,line join=round,line cap=round] ( 21.61, 86.66) --
	( 23.97, 86.54) --
	( 26.32, 86.15) --
	( 28.68, 85.41) --
	( 31.03, 84.62) --
	( 33.39, 83.56) --
	( 35.74, 81.81) --
	( 38.10, 78.76) --
	( 40.45, 75.25) --
	( 42.81, 70.11) --
	( 45.16, 64.25) --
	( 47.52, 55.99) --
	( 49.87, 48.12) --
	( 52.22, 39.86) --
	( 54.58, 32.07) --
	( 56.93, 26.23) --
	( 59.29, 23.59) --
	( 61.64, 22.97) --
	( 64.00, 24.60) --
	( 66.35, 29.90) --
	( 68.71, 41.21) --
	( 71.06, 55.87) --
	( 73.42, 70.34) --
	( 75.77, 79.82) --
	( 78.13, 84.83) --
	( 80.48, 86.57) --
	( 82.84, 87.05) --
	( 85.19, 87.15) --
	( 87.55, 87.16) --
	( 89.90, 87.17) --
	( 92.25, 87.17) --
	( 94.61, 87.17) --
	( 96.96, 87.17);
\definecolor{drawColor}{RGB}{40,226,229}

\path[draw=drawColor,line width= 0.4pt,line join=round,line cap=round] ( 21.61, 86.36) --
	( 23.97, 86.05) --
	( 26.32, 85.67) --
	( 28.68, 84.95) --
	( 31.03, 84.12) --
	( 33.39, 83.17) --
	( 35.74, 81.31) --
	( 38.10, 78.78) --
	( 40.45, 75.30) --
	( 42.81, 70.84) --
	( 45.16, 64.90) --
	( 47.52, 56.81) --
	( 49.87, 48.91) --
	( 52.22, 40.50) --
	( 54.58, 32.58) --
	( 56.93, 26.70) --
	( 59.29, 24.04) --
	( 61.64, 23.50) --
	( 64.00, 25.09) --
	( 66.35, 30.16) --
	( 68.71, 41.20) --
	( 71.06, 55.85) --
	( 73.42, 70.30) --
	( 75.77, 79.82) --
	( 78.13, 84.83) --
	( 80.48, 86.57) --
	( 82.84, 87.05) --
	( 85.19, 87.15) --
	( 87.55, 87.16) --
	( 89.90, 87.17) --
	( 92.25, 87.17) --
	( 94.61, 87.17) --
	( 96.96, 87.17);
\definecolor{drawColor}{RGB}{0,0,0}

\path[draw=drawColor,line width= 0.4pt,dash pattern=on 4pt off 4pt ,line join=round,line cap=round] ( 18.60, 24.53) -- ( 99.98, 24.53);

\path[draw=drawColor,line width= 0.4pt,dash pattern=on 4pt off 4pt ,line join=round,line cap=round] ( 59.29, 18.60) -- ( 59.29, 89.81);

\node[text=drawColor,rotate= 90.00,anchor=base,inner sep=0pt, outer sep=0pt, scale=  1.00] at ( 66.50, 70.69) {\tiny$\theta = 4 $};
\end{scope}
\end{tikzpicture}
% Created by tikzDevice version 0.12.6 on 2025-09-10 23:29:27
% !TEX encoding = UTF-8 Unicode
\begin{tikzpicture}[x=1pt,y=1pt]
\definecolor{fillColor}{RGB}{255,255,255}
\path[use as bounding box,fill=fillColor,fill opacity=0.00] (0,0) rectangle (101.18,108.41);
\begin{scope}
\path[clip] (  0.00,  0.00) rectangle (101.18,108.41);
\definecolor{drawColor}{RGB}{0,0,0}

\path[draw=drawColor,line width= 0.4pt,line join=round,line cap=round] ( 21.61, 18.60) -- ( 92.25, 18.60);

\path[draw=drawColor,line width= 0.4pt,line join=round,line cap=round] ( 21.61, 18.60) -- ( 21.61, 17.40);

\path[draw=drawColor,line width= 0.4pt,line join=round,line cap=round] ( 45.16, 18.60) -- ( 45.16, 17.40);

\path[draw=drawColor,line width= 0.4pt,line join=round,line cap=round] ( 68.71, 18.60) -- ( 68.71, 17.40);

\path[draw=drawColor,line width= 0.4pt,line join=round,line cap=round] ( 92.25, 18.60) -- ( 92.25, 17.40);

\node[text=drawColor,anchor=base,inner sep=0pt, outer sep=0pt, scale=  0.80] at ( 21.61,  9.60) {-10};

\node[text=drawColor,anchor=base,inner sep=0pt, outer sep=0pt, scale=  0.80] at ( 45.16,  9.60) {-5};

\node[text=drawColor,anchor=base,inner sep=0pt, outer sep=0pt, scale=  0.80] at ( 68.71,  9.60) {0};

\node[text=drawColor,anchor=base,inner sep=0pt, outer sep=0pt, scale=  0.80] at ( 92.25,  9.60) {5};

\path[draw=drawColor,line width= 0.4pt,line join=round,line cap=round] ( 18.60, 21.24) -- ( 18.60, 87.17);

\path[draw=drawColor,line width= 0.4pt,line join=round,line cap=round] ( 18.60, 21.24) -- ( 17.40, 21.24);

\path[draw=drawColor,line width= 0.4pt,line join=round,line cap=round] ( 18.60, 34.42) -- ( 17.40, 34.42);

\path[draw=drawColor,line width= 0.4pt,line join=round,line cap=round] ( 18.60, 47.61) -- ( 17.40, 47.61);

\path[draw=drawColor,line width= 0.4pt,line join=round,line cap=round] ( 18.60, 60.80) -- ( 17.40, 60.80);

\path[draw=drawColor,line width= 0.4pt,line join=round,line cap=round] ( 18.60, 73.98) -- ( 17.40, 73.98);

\path[draw=drawColor,line width= 0.4pt,line join=round,line cap=round] ( 18.60, 87.17) -- ( 17.40, 87.17);

\node[text=drawColor,rotate= 90.00,anchor=base,inner sep=0pt, outer sep=0pt, scale=  0.80] at ( 16.80, 21.24) {0.0};

\node[text=drawColor,rotate= 90.00,anchor=base,inner sep=0pt, outer sep=0pt, scale=  0.80] at ( 16.80, 47.61) {0.4};

\node[text=drawColor,rotate= 90.00,anchor=base,inner sep=0pt, outer sep=0pt, scale=  0.80] at ( 16.80, 73.98) {0.8};
\end{scope}
\begin{scope}
\path[clip] ( 18.60, 18.60) rectangle ( 99.98, 89.81);
\definecolor{drawColor}{RGB}{211,211,211}

\path[draw=drawColor,line width= 0.4pt,dash pattern=on 1pt off 3pt ,line join=round,line cap=round] ( 18.60, 21.24) -- ( 99.98, 21.24);

\path[draw=drawColor,line width= 0.4pt,dash pattern=on 1pt off 3pt ,line join=round,line cap=round] ( 18.60, 34.42) -- ( 99.98, 34.42);

\path[draw=drawColor,line width= 0.4pt,dash pattern=on 1pt off 3pt ,line join=round,line cap=round] ( 18.60, 47.61) -- ( 99.98, 47.61);

\path[draw=drawColor,line width= 0.4pt,dash pattern=on 1pt off 3pt ,line join=round,line cap=round] ( 18.60, 60.80) -- ( 99.98, 60.80);

\path[draw=drawColor,line width= 0.4pt,dash pattern=on 1pt off 3pt ,line join=round,line cap=round] ( 18.60, 73.98) -- ( 99.98, 73.98);

\path[draw=drawColor,line width= 0.4pt,dash pattern=on 1pt off 3pt ,line join=round,line cap=round] ( 18.60, 87.17) -- ( 99.98, 87.17);

\path[draw=drawColor,line width= 0.4pt,dash pattern=on 1pt off 3pt ,line join=round,line cap=round] ( 21.61, 18.60) -- ( 21.61, 89.81);

\path[draw=drawColor,line width= 0.4pt,dash pattern=on 1pt off 3pt ,line join=round,line cap=round] ( 45.16, 18.60) -- ( 45.16, 89.81);

\path[draw=drawColor,line width= 0.4pt,dash pattern=on 1pt off 3pt ,line join=round,line cap=round] ( 68.71, 18.60) -- ( 68.71, 89.81);

\path[draw=drawColor,line width= 0.4pt,dash pattern=on 1pt off 3pt ,line join=round,line cap=round] ( 92.25, 18.60) -- ( 92.25, 89.81);
\end{scope}
\begin{scope}
\path[clip] (  0.00,  0.00) rectangle (101.18,108.41);
\definecolor{drawColor}{RGB}{0,0,0}

\node[text=drawColor,anchor=base,inner sep=0pt, outer sep=0pt, scale=  1.00] at ( 59.29,  0.60) {\footnotesize$t$};

\node[text=drawColor,rotate= 90.00,anchor=base,inner sep=0pt, outer sep=0pt, scale=  1.00] at (  6.60, 54.20) {\footnotesize power};
\end{scope}
\begin{scope}
\path[clip] (  0.00,  0.00) rectangle (101.18,108.41);
\definecolor{drawColor}{RGB}{0,0,0}

\node[text=drawColor,anchor=base,inner sep=0pt, outer sep=0pt, scale=  1.00] at ( 59.29, 92.21) {\footnotesize $P= 5.55\times 10^{-10} $};
\end{scope}
\begin{scope}
\path[clip] ( 18.60, 18.60) rectangle ( 99.98, 89.81);
\definecolor{drawColor}{RGB}{97,208,79}

\path[draw=drawColor,line width= 0.4pt,line join=round,line cap=round] ( 21.61, 73.43) --
	( 23.97, 70.28) --
	( 26.32, 66.87) --
	( 28.68, 63.70) --
	( 31.03, 59.72) --
	( 33.39, 54.97) --
	( 35.74, 50.71) --
	( 38.10, 46.21) --
	( 40.45, 42.27) --
	( 42.81, 38.41) --
	( 45.16, 35.67) --
	( 47.52, 32.37) --
	( 49.87, 29.49) --
	( 52.22, 27.62) --
	( 54.58, 25.94) --
	( 56.93, 24.72) --
	( 59.29, 24.49) --
	( 61.64, 24.77) --
	( 64.00, 25.34) --
	( 66.35, 27.18) --
	( 68.71, 28.96) --
	( 71.06, 31.76) --
	( 73.42, 35.67) --
	( 75.77, 40.98) --
	( 78.13, 46.92) --
	( 80.48, 52.59) --
	( 82.84, 59.16) --
	( 85.19, 66.22) --
	( 87.55, 72.80) --
	( 89.90, 78.19) --
	( 92.25, 82.32) --
	( 94.61, 85.23) --
	( 96.96, 86.49);
\definecolor{drawColor}{RGB}{34,151,230}

\path[draw=drawColor,line width= 0.4pt,line join=round,line cap=round] ( 21.61, 86.99) --
	( 23.97, 86.90) --
	( 26.32, 86.77) --
	( 28.68, 86.42) --
	( 31.03, 86.11) --
	( 33.39, 85.83) --
	( 35.74, 85.33) --
	( 38.10, 84.13) --
	( 40.45, 82.99) --
	( 42.81, 81.13) --
	( 45.16, 78.91) --
	( 47.52, 75.62) --
	( 49.87, 71.67) --
	( 52.22, 67.12) --
	( 54.58, 60.69) --
	( 56.93, 54.73) --
	( 59.29, 47.58) --
	( 61.64, 40.90) --
	( 64.00, 34.85) --
	( 66.35, 29.30) --
	( 68.71, 25.58) --
	( 71.06, 23.43) --
	( 73.42, 22.17) --
	( 75.77, 21.62) --
	( 78.13, 21.51) --
	( 80.48, 21.75) --
	( 82.84, 22.61) --
	( 85.19, 24.91) --
	( 87.55, 28.99) --
	( 89.90, 38.58) --
	( 92.25, 53.31) --
	( 94.61, 71.13) --
	( 96.96, 82.86);
\definecolor{drawColor}{RGB}{40,226,229}

\path[draw=drawColor,line width= 0.4pt,line join=round,line cap=round] ( 21.61, 82.53) --
	( 23.97, 81.66) --
	( 26.32, 80.65) --
	( 28.68, 79.51) --
	( 31.03, 77.75) --
	( 33.39, 75.79) --
	( 35.74, 74.12) --
	( 38.10, 71.47) --
	( 40.45, 69.00) --
	( 42.81, 65.10) --
	( 45.16, 62.23) --
	( 47.52, 58.02) --
	( 49.87, 54.34) --
	( 52.22, 50.19) --
	( 54.58, 44.81) --
	( 56.93, 41.14) --
	( 59.29, 36.63) --
	( 61.64, 33.16) --
	( 64.00, 30.20) --
	( 66.35, 27.24) --
	( 68.71, 26.56) --
	( 71.06, 25.94) --
	( 73.42, 26.54) --
	( 75.77, 27.48) --
	( 78.13, 29.25) --
	( 80.48, 31.66) --
	( 82.84, 35.10) --
	( 85.19, 40.78) --
	( 87.55, 47.46) --
	( 89.90, 56.46) --
	( 92.25, 67.50) --
	( 94.61, 77.75) --
	( 96.96, 84.58);
\definecolor{drawColor}{RGB}{0,0,0}

\path[draw=drawColor,line width= 0.4pt,dash pattern=on 4pt off 4pt ,line join=round,line cap=round] ( 18.60, 24.53) -- ( 99.98, 24.53);

\path[draw=drawColor,line width= 0.4pt,dash pattern=on 4pt off 4pt ,line join=round,line cap=round] ( 59.29, 18.60) -- ( 59.29, 89.81);

\node[text=drawColor,rotate= 90.00,anchor=base,inner sep=0pt, outer sep=0pt, scale=  1.00] at ( 66.50, 70.69) {\tiny$\theta = -2 $};
\end{scope}
\end{tikzpicture}
% Created by tikzDevice version 0.12.6 on 2025-09-10 23:29:27
% !TEX encoding = UTF-8 Unicode
\begin{tikzpicture}[x=1pt,y=1pt]
\definecolor{fillColor}{RGB}{255,255,255}
\path[use as bounding box,fill=fillColor,fill opacity=0.00] (0,0) rectangle (101.18,108.41);
\begin{scope}
\path[clip] (  0.00,  0.00) rectangle (101.18,108.41);
\definecolor{drawColor}{RGB}{0,0,0}

\path[draw=drawColor,line width= 0.4pt,line join=round,line cap=round] ( 35.74, 18.60) -- ( 82.84, 18.60);

\path[draw=drawColor,line width= 0.4pt,line join=round,line cap=round] ( 35.74, 18.60) -- ( 35.74, 17.40);

\path[draw=drawColor,line width= 0.4pt,line join=round,line cap=round] ( 59.29, 18.60) -- ( 59.29, 17.40);

\path[draw=drawColor,line width= 0.4pt,line join=round,line cap=round] ( 82.84, 18.60) -- ( 82.84, 17.40);

\node[text=drawColor,anchor=base,inner sep=0pt, outer sep=0pt, scale=  0.80] at ( 35.74,  9.60) {-5};

\node[text=drawColor,anchor=base,inner sep=0pt, outer sep=0pt, scale=  0.80] at ( 59.29,  9.60) {0};

\node[text=drawColor,anchor=base,inner sep=0pt, outer sep=0pt, scale=  0.80] at ( 82.84,  9.60) {5};

\path[draw=drawColor,line width= 0.4pt,line join=round,line cap=round] ( 18.60, 21.24) -- ( 18.60, 87.17);

\path[draw=drawColor,line width= 0.4pt,line join=round,line cap=round] ( 18.60, 21.24) -- ( 17.40, 21.24);

\path[draw=drawColor,line width= 0.4pt,line join=round,line cap=round] ( 18.60, 34.42) -- ( 17.40, 34.42);

\path[draw=drawColor,line width= 0.4pt,line join=round,line cap=round] ( 18.60, 47.61) -- ( 17.40, 47.61);

\path[draw=drawColor,line width= 0.4pt,line join=round,line cap=round] ( 18.60, 60.80) -- ( 17.40, 60.80);

\path[draw=drawColor,line width= 0.4pt,line join=round,line cap=round] ( 18.60, 73.98) -- ( 17.40, 73.98);

\path[draw=drawColor,line width= 0.4pt,line join=round,line cap=round] ( 18.60, 87.17) -- ( 17.40, 87.17);

\node[text=drawColor,rotate= 90.00,anchor=base,inner sep=0pt, outer sep=0pt, scale=  0.80] at ( 16.80, 21.24) {0.0};

\node[text=drawColor,rotate= 90.00,anchor=base,inner sep=0pt, outer sep=0pt, scale=  0.80] at ( 16.80, 47.61) {0.4};

\node[text=drawColor,rotate= 90.00,anchor=base,inner sep=0pt, outer sep=0pt, scale=  0.80] at ( 16.80, 73.98) {0.8};
\end{scope}
\begin{scope}
\path[clip] ( 18.60, 18.60) rectangle ( 99.98, 89.81);
\definecolor{drawColor}{RGB}{211,211,211}

\path[draw=drawColor,line width= 0.4pt,dash pattern=on 1pt off 3pt ,line join=round,line cap=round] ( 18.60, 21.24) -- ( 99.98, 21.24);

\path[draw=drawColor,line width= 0.4pt,dash pattern=on 1pt off 3pt ,line join=round,line cap=round] ( 18.60, 34.42) -- ( 99.98, 34.42);

\path[draw=drawColor,line width= 0.4pt,dash pattern=on 1pt off 3pt ,line join=round,line cap=round] ( 18.60, 47.61) -- ( 99.98, 47.61);

\path[draw=drawColor,line width= 0.4pt,dash pattern=on 1pt off 3pt ,line join=round,line cap=round] ( 18.60, 60.80) -- ( 99.98, 60.80);

\path[draw=drawColor,line width= 0.4pt,dash pattern=on 1pt off 3pt ,line join=round,line cap=round] ( 18.60, 73.98) -- ( 99.98, 73.98);

\path[draw=drawColor,line width= 0.4pt,dash pattern=on 1pt off 3pt ,line join=round,line cap=round] ( 18.60, 87.17) -- ( 99.98, 87.17);

\path[draw=drawColor,line width= 0.4pt,dash pattern=on 1pt off 3pt ,line join=round,line cap=round] ( 35.74, 18.60) -- ( 35.74, 89.81);

\path[draw=drawColor,line width= 0.4pt,dash pattern=on 1pt off 3pt ,line join=round,line cap=round] ( 59.29, 18.60) -- ( 59.29, 89.81);

\path[draw=drawColor,line width= 0.4pt,dash pattern=on 1pt off 3pt ,line join=round,line cap=round] ( 82.84, 18.60) -- ( 82.84, 89.81);
\end{scope}
\begin{scope}
\path[clip] (  0.00,  0.00) rectangle (101.18,108.41);
\definecolor{drawColor}{RGB}{0,0,0}

\node[text=drawColor,anchor=base,inner sep=0pt, outer sep=0pt, scale=  1.00] at ( 59.29,  0.60) {\footnotesize$t$};

\node[text=drawColor,rotate= 90.00,anchor=base,inner sep=0pt, outer sep=0pt, scale=  1.00] at (  6.60, 54.20) {\footnotesize power};
\end{scope}
\begin{scope}
\path[clip] (  0.00,  0.00) rectangle (101.18,108.41);
\definecolor{drawColor}{RGB}{0,0,0}

\node[text=drawColor,anchor=base,inner sep=0pt, outer sep=0pt, scale=  1.00] at ( 59.29, 92.21) {\footnotesize $P= 9.73\times 10^{-6} $};
\end{scope}
\begin{scope}
\path[clip] ( 18.60, 18.60) rectangle ( 99.98, 89.81);
\definecolor{drawColor}{RGB}{97,208,79}

\path[draw=drawColor,line width= 0.4pt,line join=round,line cap=round] ( 21.61, 76.47) --
	( 23.97, 73.58) --
	( 26.32, 70.70) --
	( 28.68, 67.07) --
	( 31.03, 63.02) --
	( 33.39, 59.11) --
	( 35.74, 54.24) --
	( 38.10, 50.29) --
	( 40.45, 45.47) --
	( 42.81, 41.10) --
	( 45.16, 37.50) --
	( 47.52, 33.80) --
	( 49.87, 30.41) --
	( 52.22, 28.19) --
	( 54.58, 26.31) --
	( 56.93, 25.18) --
	( 59.29, 24.72) --
	( 61.64, 24.80) --
	( 64.00, 25.35) --
	( 66.35, 27.36) --
	( 68.71, 30.50) --
	( 71.06, 34.89) --
	( 73.42, 39.06) --
	( 75.77, 45.67) --
	( 78.13, 52.56) --
	( 80.48, 61.81) --
	( 82.84, 69.32) --
	( 85.19, 76.90) --
	( 87.55, 82.18) --
	( 89.90, 85.59) --
	( 92.25, 86.80) --
	( 94.61, 87.13) --
	( 96.96, 87.17);
\definecolor{drawColor}{RGB}{34,151,230}

\path[draw=drawColor,line width= 0.4pt,line join=round,line cap=round] ( 21.61, 86.98) --
	( 23.97, 86.91) --
	( 26.32, 86.77) --
	( 28.68, 86.49) --
	( 31.03, 86.12) --
	( 33.39, 85.72) --
	( 35.74, 84.58) --
	( 38.10, 83.44) --
	( 40.45, 81.70) --
	( 42.81, 78.97) --
	( 45.16, 75.41) --
	( 47.52, 71.16) --
	( 49.87, 65.58) --
	( 52.22, 58.86) --
	( 54.58, 50.85) --
	( 56.93, 43.85) --
	( 59.29, 36.66) --
	( 61.64, 31.02) --
	( 64.00, 26.14) --
	( 66.35, 23.41) --
	( 68.71, 22.08) --
	( 71.06, 21.64) --
	( 73.42, 21.57) --
	( 75.77, 21.99) --
	( 78.13, 23.14) --
	( 80.48, 27.22) --
	( 82.84, 35.88) --
	( 85.19, 52.26) --
	( 87.55, 70.97) --
	( 89.90, 82.54) --
	( 92.25, 86.44) --
	( 94.61, 87.12) --
	( 96.96, 87.17);
\definecolor{drawColor}{RGB}{40,226,229}

\path[draw=drawColor,line width= 0.4pt,line join=round,line cap=round] ( 21.61, 83.21) --
	( 23.97, 81.73) --
	( 26.32, 80.83) --
	( 28.68, 79.64) --
	( 31.03, 77.79) --
	( 33.39, 75.81) --
	( 35.74, 73.35) --
	( 38.10, 70.81) --
	( 40.45, 67.72) --
	( 42.81, 64.03) --
	( 45.16, 59.57) --
	( 47.52, 54.80) --
	( 49.87, 50.51) --
	( 52.22, 45.49) --
	( 54.58, 40.30) --
	( 56.93, 35.87) --
	( 59.29, 31.68) --
	( 61.64, 28.46) --
	( 64.00, 25.96) --
	( 66.35, 25.14) --
	( 68.71, 25.25) --
	( 71.06, 26.38) --
	( 73.42, 28.11) --
	( 75.77, 30.69) --
	( 78.13, 34.88) --
	( 80.48, 42.10) --
	( 82.84, 51.61) --
	( 85.19, 63.96) --
	( 87.55, 76.41) --
	( 89.90, 83.94) --
	( 92.25, 86.60) --
	( 94.61, 87.13) --
	( 96.96, 87.17);
\definecolor{drawColor}{RGB}{0,0,0}

\path[draw=drawColor,line width= 0.4pt,dash pattern=on 4pt off 4pt ,line join=round,line cap=round] ( 18.60, 24.53) -- ( 99.98, 24.53);

\path[draw=drawColor,line width= 0.4pt,dash pattern=on 4pt off 4pt ,line join=round,line cap=round] ( 59.29, 18.60) -- ( 59.29, 89.81);

\node[text=drawColor,rotate= 90.00,anchor=base,inner sep=0pt, outer sep=0pt, scale=  1.00] at ( 66.50, 70.69) {\tiny$\theta = 0 $};
\end{scope}
\end{tikzpicture}
% Created by tikzDevice version 0.12.6 on 2025-09-10 23:29:28
% !TEX encoding = UTF-8 Unicode
\begin{tikzpicture}[x=1pt,y=1pt]
\definecolor{fillColor}{RGB}{255,255,255}
\path[use as bounding box,fill=fillColor,fill opacity=0.00] (0,0) rectangle (101.18,108.41);
\begin{scope}
\path[clip] (  0.00,  0.00) rectangle (101.18,108.41);
\definecolor{drawColor}{RGB}{0,0,0}

\path[draw=drawColor,line width= 0.4pt,line join=round,line cap=round] ( 26.32, 18.60) -- ( 96.96, 18.60);

\path[draw=drawColor,line width= 0.4pt,line join=round,line cap=round] ( 26.32, 18.60) -- ( 26.32, 17.40);

\path[draw=drawColor,line width= 0.4pt,line join=round,line cap=round] ( 49.87, 18.60) -- ( 49.87, 17.40);

\path[draw=drawColor,line width= 0.4pt,line join=round,line cap=round] ( 73.42, 18.60) -- ( 73.42, 17.40);

\path[draw=drawColor,line width= 0.4pt,line join=round,line cap=round] ( 96.96, 18.60) -- ( 96.96, 17.40);

\node[text=drawColor,anchor=base,inner sep=0pt, outer sep=0pt, scale=  0.80] at ( 26.32,  9.60) {-5};

\node[text=drawColor,anchor=base,inner sep=0pt, outer sep=0pt, scale=  0.80] at ( 49.87,  9.60) {0};

\node[text=drawColor,anchor=base,inner sep=0pt, outer sep=0pt, scale=  0.80] at ( 73.42,  9.60) {5};

\node[text=drawColor,anchor=base,inner sep=0pt, outer sep=0pt, scale=  0.80] at ( 96.96,  9.60) {10};

\path[draw=drawColor,line width= 0.4pt,line join=round,line cap=round] ( 18.60, 21.24) -- ( 18.60, 87.17);

\path[draw=drawColor,line width= 0.4pt,line join=round,line cap=round] ( 18.60, 21.24) -- ( 17.40, 21.24);

\path[draw=drawColor,line width= 0.4pt,line join=round,line cap=round] ( 18.60, 34.42) -- ( 17.40, 34.42);

\path[draw=drawColor,line width= 0.4pt,line join=round,line cap=round] ( 18.60, 47.61) -- ( 17.40, 47.61);

\path[draw=drawColor,line width= 0.4pt,line join=round,line cap=round] ( 18.60, 60.80) -- ( 17.40, 60.80);

\path[draw=drawColor,line width= 0.4pt,line join=round,line cap=round] ( 18.60, 73.98) -- ( 17.40, 73.98);

\path[draw=drawColor,line width= 0.4pt,line join=round,line cap=round] ( 18.60, 87.17) -- ( 17.40, 87.17);

\node[text=drawColor,rotate= 90.00,anchor=base,inner sep=0pt, outer sep=0pt, scale=  0.80] at ( 16.80, 21.24) {0.0};

\node[text=drawColor,rotate= 90.00,anchor=base,inner sep=0pt, outer sep=0pt, scale=  0.80] at ( 16.80, 47.61) {0.4};

\node[text=drawColor,rotate= 90.00,anchor=base,inner sep=0pt, outer sep=0pt, scale=  0.80] at ( 16.80, 73.98) {0.8};
\end{scope}
\begin{scope}
\path[clip] ( 18.60, 18.60) rectangle ( 99.98, 89.81);
\definecolor{drawColor}{RGB}{211,211,211}

\path[draw=drawColor,line width= 0.4pt,dash pattern=on 1pt off 3pt ,line join=round,line cap=round] ( 18.60, 21.24) -- ( 99.98, 21.24);

\path[draw=drawColor,line width= 0.4pt,dash pattern=on 1pt off 3pt ,line join=round,line cap=round] ( 18.60, 34.42) -- ( 99.98, 34.42);

\path[draw=drawColor,line width= 0.4pt,dash pattern=on 1pt off 3pt ,line join=round,line cap=round] ( 18.60, 47.61) -- ( 99.98, 47.61);

\path[draw=drawColor,line width= 0.4pt,dash pattern=on 1pt off 3pt ,line join=round,line cap=round] ( 18.60, 60.80) -- ( 99.98, 60.80);

\path[draw=drawColor,line width= 0.4pt,dash pattern=on 1pt off 3pt ,line join=round,line cap=round] ( 18.60, 73.98) -- ( 99.98, 73.98);

\path[draw=drawColor,line width= 0.4pt,dash pattern=on 1pt off 3pt ,line join=round,line cap=round] ( 18.60, 87.17) -- ( 99.98, 87.17);

\path[draw=drawColor,line width= 0.4pt,dash pattern=on 1pt off 3pt ,line join=round,line cap=round] ( 26.32, 18.60) -- ( 26.32, 89.81);

\path[draw=drawColor,line width= 0.4pt,dash pattern=on 1pt off 3pt ,line join=round,line cap=round] ( 49.87, 18.60) -- ( 49.87, 89.81);

\path[draw=drawColor,line width= 0.4pt,dash pattern=on 1pt off 3pt ,line join=round,line cap=round] ( 73.42, 18.60) -- ( 73.42, 89.81);

\path[draw=drawColor,line width= 0.4pt,dash pattern=on 1pt off 3pt ,line join=round,line cap=round] ( 96.96, 18.60) -- ( 96.96, 89.81);
\end{scope}
\begin{scope}
\path[clip] (  0.00,  0.00) rectangle (101.18,108.41);
\definecolor{drawColor}{RGB}{0,0,0}

\node[text=drawColor,anchor=base,inner sep=0pt, outer sep=0pt, scale=  1.00] at ( 59.29,  0.60) {\footnotesize$t$};

\node[text=drawColor,rotate= 90.00,anchor=base,inner sep=0pt, outer sep=0pt, scale=  1.00] at (  6.60, 54.20) {\footnotesize power};
\end{scope}
\begin{scope}
\path[clip] (  0.00,  0.00) rectangle (101.18,108.41);
\definecolor{drawColor}{RGB}{0,0,0}

\node[text=drawColor,anchor=base,inner sep=0pt, outer sep=0pt, scale=  1.00] at ( 59.29, 92.21) {\footnotesize $P= 0.00667 $};
\end{scope}
\begin{scope}
\path[clip] ( 18.60, 18.60) rectangle ( 99.98, 89.81);
\definecolor{drawColor}{RGB}{97,208,79}

\path[draw=drawColor,line width= 0.4pt,line join=round,line cap=round] ( 21.61, 80.33) --
	( 23.97, 77.81) --
	( 26.32, 74.94) --
	( 28.68, 71.83) --
	( 31.03, 68.08) --
	( 33.39, 64.96) --
	( 35.74, 60.38) --
	( 38.10, 55.59) --
	( 40.45, 50.61) --
	( 42.81, 46.08) --
	( 45.16, 40.57) --
	( 47.52, 36.64) --
	( 49.87, 32.45) --
	( 52.22, 29.34) --
	( 54.58, 26.83) --
	( 56.93, 25.38) --
	( 59.29, 24.61) --
	( 61.64, 24.73) --
	( 64.00, 26.03) --
	( 66.35, 28.44) --
	( 68.71, 32.50) --
	( 71.06, 38.77) --
	( 73.42, 46.43) --
	( 75.77, 55.85) --
	( 78.13, 66.80) --
	( 80.48, 76.28) --
	( 82.84, 82.99) --
	( 85.19, 85.90) --
	( 87.55, 87.00) --
	( 89.90, 87.15) --
	( 92.25, 87.17) --
	( 94.61, 87.17) --
	( 96.96, 87.17);
\definecolor{drawColor}{RGB}{34,151,230}

\path[draw=drawColor,line width= 0.4pt,line join=round,line cap=round] ( 21.61, 87.04) --
	( 23.97, 86.90) --
	( 26.32, 86.77) --
	( 28.68, 86.60) --
	( 31.03, 86.01) --
	( 33.39, 85.35) --
	( 35.74, 84.13) --
	( 38.10, 82.21) --
	( 40.45, 79.92) --
	( 42.81, 75.68) --
	( 45.16, 70.80) --
	( 47.52, 65.04) --
	( 49.87, 57.39) --
	( 52.22, 49.02) --
	( 54.58, 40.80) --
	( 56.93, 33.49) --
	( 59.29, 28.19) --
	( 61.64, 24.63) --
	( 64.00, 22.48) --
	( 66.35, 21.76) --
	( 68.71, 21.65) --
	( 71.06, 22.38) --
	( 73.42, 25.11) --
	( 75.77, 32.83) --
	( 78.13, 47.99) --
	( 80.48, 67.34) --
	( 82.84, 80.34) --
	( 85.19, 85.55) --
	( 87.55, 86.96) --
	( 89.90, 87.15) --
	( 92.25, 87.17) --
	( 94.61, 87.17) --
	( 96.96, 87.17);
\definecolor{drawColor}{RGB}{40,226,229}

\path[draw=drawColor,line width= 0.4pt,line join=round,line cap=round] ( 21.61, 84.04) --
	( 23.97, 82.82) --
	( 26.32, 81.79) --
	( 28.68, 80.36) --
	( 31.03, 78.90) --
	( 33.39, 76.57) --
	( 35.74, 74.54) --
	( 38.10, 71.23) --
	( 40.45, 67.30) --
	( 42.81, 63.19) --
	( 45.16, 57.55) --
	( 47.52, 53.13) --
	( 49.87, 47.09) --
	( 52.22, 40.73) --
	( 54.58, 35.37) --
	( 56.93, 30.69) --
	( 59.29, 27.40) --
	( 61.64, 25.47) --
	( 64.00, 24.36) --
	( 66.35, 25.29) --
	( 68.71, 26.60) --
	( 71.06, 29.49) --
	( 73.42, 35.00) --
	( 75.77, 43.65) --
	( 78.13, 56.88) --
	( 80.48, 71.17) --
	( 82.84, 81.33) --
	( 85.19, 85.72) --
	( 87.55, 86.98) --
	( 89.90, 87.15) --
	( 92.25, 87.17) --
	( 94.61, 87.17) --
	( 96.96, 87.17);
\definecolor{drawColor}{RGB}{0,0,0}

\path[draw=drawColor,line width= 0.4pt,dash pattern=on 4pt off 4pt ,line join=round,line cap=round] ( 18.60, 24.53) -- ( 99.98, 24.53);

\path[draw=drawColor,line width= 0.4pt,dash pattern=on 4pt off 4pt ,line join=round,line cap=round] ( 59.29, 18.60) -- ( 59.29, 89.81);

\node[text=drawColor,rotate= 90.00,anchor=base,inner sep=0pt, outer sep=0pt, scale=  1.00] at ( 66.50, 70.69) {\tiny$\theta = 2 $};
\end{scope}
\end{tikzpicture}
% Created by tikzDevice version 0.12.6 on 2025-09-10 23:29:28
% !TEX encoding = UTF-8 Unicode
\begin{tikzpicture}[x=1pt,y=1pt]
\definecolor{fillColor}{RGB}{255,255,255}
\path[use as bounding box,fill=fillColor,fill opacity=0.00] (0,0) rectangle (101.18,108.41);
\begin{scope}
\path[clip] (  0.00,  0.00) rectangle (101.18,108.41);
\definecolor{drawColor}{RGB}{0,0,0}

\path[draw=drawColor,line width= 0.4pt,line join=round,line cap=round] ( 40.45, 18.60) -- ( 87.55, 18.60);

\path[draw=drawColor,line width= 0.4pt,line join=round,line cap=round] ( 40.45, 18.60) -- ( 40.45, 17.40);

\path[draw=drawColor,line width= 0.4pt,line join=round,line cap=round] ( 64.00, 18.60) -- ( 64.00, 17.40);

\path[draw=drawColor,line width= 0.4pt,line join=round,line cap=round] ( 87.55, 18.60) -- ( 87.55, 17.40);

\node[text=drawColor,anchor=base,inner sep=0pt, outer sep=0pt, scale=  0.80] at ( 40.45,  9.60) {0};

\node[text=drawColor,anchor=base,inner sep=0pt, outer sep=0pt, scale=  0.80] at ( 64.00,  9.60) {5};

\node[text=drawColor,anchor=base,inner sep=0pt, outer sep=0pt, scale=  0.80] at ( 87.55,  9.60) {10};

\path[draw=drawColor,line width= 0.4pt,line join=round,line cap=round] ( 18.60, 21.24) -- ( 18.60, 87.17);

\path[draw=drawColor,line width= 0.4pt,line join=round,line cap=round] ( 18.60, 21.24) -- ( 17.40, 21.24);

\path[draw=drawColor,line width= 0.4pt,line join=round,line cap=round] ( 18.60, 34.42) -- ( 17.40, 34.42);

\path[draw=drawColor,line width= 0.4pt,line join=round,line cap=round] ( 18.60, 47.61) -- ( 17.40, 47.61);

\path[draw=drawColor,line width= 0.4pt,line join=round,line cap=round] ( 18.60, 60.80) -- ( 17.40, 60.80);

\path[draw=drawColor,line width= 0.4pt,line join=round,line cap=round] ( 18.60, 73.98) -- ( 17.40, 73.98);

\path[draw=drawColor,line width= 0.4pt,line join=round,line cap=round] ( 18.60, 87.17) -- ( 17.40, 87.17);

\node[text=drawColor,rotate= 90.00,anchor=base,inner sep=0pt, outer sep=0pt, scale=  0.80] at ( 16.80, 21.24) {0.0};

\node[text=drawColor,rotate= 90.00,anchor=base,inner sep=0pt, outer sep=0pt, scale=  0.80] at ( 16.80, 47.61) {0.4};

\node[text=drawColor,rotate= 90.00,anchor=base,inner sep=0pt, outer sep=0pt, scale=  0.80] at ( 16.80, 73.98) {0.8};
\end{scope}
\begin{scope}
\path[clip] ( 18.60, 18.60) rectangle ( 99.98, 89.81);
\definecolor{drawColor}{RGB}{211,211,211}

\path[draw=drawColor,line width= 0.4pt,dash pattern=on 1pt off 3pt ,line join=round,line cap=round] ( 18.60, 21.24) -- ( 99.98, 21.24);

\path[draw=drawColor,line width= 0.4pt,dash pattern=on 1pt off 3pt ,line join=round,line cap=round] ( 18.60, 34.42) -- ( 99.98, 34.42);

\path[draw=drawColor,line width= 0.4pt,dash pattern=on 1pt off 3pt ,line join=round,line cap=round] ( 18.60, 47.61) -- ( 99.98, 47.61);

\path[draw=drawColor,line width= 0.4pt,dash pattern=on 1pt off 3pt ,line join=round,line cap=round] ( 18.60, 60.80) -- ( 99.98, 60.80);

\path[draw=drawColor,line width= 0.4pt,dash pattern=on 1pt off 3pt ,line join=round,line cap=round] ( 18.60, 73.98) -- ( 99.98, 73.98);

\path[draw=drawColor,line width= 0.4pt,dash pattern=on 1pt off 3pt ,line join=round,line cap=round] ( 18.60, 87.17) -- ( 99.98, 87.17);

\path[draw=drawColor,line width= 0.4pt,dash pattern=on 1pt off 3pt ,line join=round,line cap=round] ( 40.45, 18.60) -- ( 40.45, 89.81);

\path[draw=drawColor,line width= 0.4pt,dash pattern=on 1pt off 3pt ,line join=round,line cap=round] ( 64.00, 18.60) -- ( 64.00, 89.81);

\path[draw=drawColor,line width= 0.4pt,dash pattern=on 1pt off 3pt ,line join=round,line cap=round] ( 87.55, 18.60) -- ( 87.55, 89.81);
\end{scope}
\begin{scope}
\path[clip] (  0.00,  0.00) rectangle (101.18,108.41);
\definecolor{drawColor}{RGB}{0,0,0}

\node[text=drawColor,anchor=base,inner sep=0pt, outer sep=0pt, scale=  1.00] at ( 59.29,  0.60) {\footnotesize$t$};

\node[text=drawColor,rotate= 90.00,anchor=base,inner sep=0pt, outer sep=0pt, scale=  1.00] at (  6.60, 54.20) {\footnotesize power};
\end{scope}
\begin{scope}
\path[clip] (  0.00,  0.00) rectangle (101.18,108.41);
\definecolor{drawColor}{RGB}{0,0,0}

\node[text=drawColor,anchor=base,inner sep=0pt, outer sep=0pt, scale=  1.00] at ( 59.29, 92.21) {\footnotesize $P= 0.237 $};
\end{scope}
\begin{scope}
\path[clip] ( 18.60, 18.60) rectangle ( 99.98, 89.81);
\definecolor{drawColor}{RGB}{97,208,79}

\path[draw=drawColor,line width= 0.4pt,line join=round,line cap=round] ( 21.61, 83.71) --
	( 23.97, 82.55) --
	( 26.32, 81.01) --
	( 28.68, 78.39) --
	( 31.03, 75.72) --
	( 33.39, 73.09) --
	( 35.74, 68.51) --
	( 38.10, 64.14) --
	( 40.45, 59.02) --
	( 42.81, 53.50) --
	( 45.16, 48.39) --
	( 47.52, 42.44) --
	( 49.87, 37.03) --
	( 52.22, 32.19) --
	( 54.58, 28.32) --
	( 56.93, 25.94) --
	( 59.29, 24.56) --
	( 61.64, 24.86) --
	( 64.00, 27.18) --
	( 66.35, 31.56) --
	( 68.71, 38.75) --
	( 71.06, 49.52) --
	( 73.42, 62.12) --
	( 75.77, 74.40) --
	( 78.13, 82.32) --
	( 80.48, 85.80) --
	( 82.84, 86.87) --
	( 85.19, 87.07) --
	( 87.55, 87.16) --
	( 89.90, 87.17) --
	( 92.25, 87.17) --
	( 94.61, 87.17) --
	( 96.96, 87.17);
\definecolor{drawColor}{RGB}{34,151,230}

\path[draw=drawColor,line width= 0.4pt,line join=round,line cap=round] ( 21.61, 87.05) --
	( 23.97, 87.04) --
	( 26.32, 86.82) --
	( 28.68, 86.44) --
	( 31.03, 85.93) --
	( 33.39, 85.23) --
	( 35.74, 83.42) --
	( 38.10, 81.52) --
	( 40.45, 77.89) --
	( 42.81, 73.19) --
	( 45.16, 66.82) --
	( 47.52, 59.20) --
	( 49.87, 50.60) --
	( 52.22, 41.89) --
	( 54.58, 33.68) --
	( 56.93, 28.30) --
	( 59.29, 24.18) --
	( 61.64, 22.46) --
	( 64.00, 21.98) --
	( 66.35, 23.16) --
	( 68.71, 28.23) --
	( 71.06, 39.43) --
	( 73.42, 56.23) --
	( 75.77, 72.47) --
	( 78.13, 82.02) --
	( 80.48, 85.80) --
	( 82.84, 86.86) --
	( 85.19, 87.07) --
	( 87.55, 87.16) --
	( 89.90, 87.17) --
	( 92.25, 87.17) --
	( 94.61, 87.17) --
	( 96.96, 87.17);
\definecolor{drawColor}{RGB}{40,226,229}

\path[draw=drawColor,line width= 0.4pt,line join=round,line cap=round] ( 21.61, 85.26) --
	( 23.97, 84.83) --
	( 26.32, 83.95) --
	( 28.68, 82.72) --
	( 31.03, 81.31) --
	( 33.39, 79.78) --
	( 35.74, 76.59) --
	( 38.10, 74.05) --
	( 40.45, 70.03) --
	( 42.81, 64.78) --
	( 45.16, 59.25) --
	( 47.52, 52.49) --
	( 49.87, 45.62) --
	( 52.22, 38.74) --
	( 54.58, 32.02) --
	( 56.93, 28.01) --
	( 59.29, 25.03) --
	( 61.64, 24.19) --
	( 64.00, 24.86) --
	( 66.35, 27.41) --
	( 68.71, 33.81) --
	( 71.06, 44.33) --
	( 73.42, 59.02) --
	( 75.77, 73.34) --
	( 78.13, 82.16) --
	( 80.48, 85.81) --
	( 82.84, 86.86) --
	( 85.19, 87.07) --
	( 87.55, 87.16) --
	( 89.90, 87.17) --
	( 92.25, 87.17) --
	( 94.61, 87.17) --
	( 96.96, 87.17);
\definecolor{drawColor}{RGB}{0,0,0}

\path[draw=drawColor,line width= 0.4pt,dash pattern=on 4pt off 4pt ,line join=round,line cap=round] ( 18.60, 24.53) -- ( 99.98, 24.53);

\path[draw=drawColor,line width= 0.4pt,dash pattern=on 4pt off 4pt ,line join=round,line cap=round] ( 59.29, 18.60) -- ( 59.29, 89.81);

\node[text=drawColor,rotate= 90.00,anchor=base,inner sep=0pt, outer sep=0pt, scale=  1.00] at ( 66.50, 70.69) {\tiny$\theta = 4 $};
\end{scope}
\end{tikzpicture}
\caption{Power curves $t \mapsto \Pr(Y \not\in [l(t,\hat{\bs\eta}), u(t,\hat{\bs\eta})]~|~\bs X \prec Y,\bs\eta,\theta)$ of the size-$5\%$ oracle procedure for testing $H:\theta=t$ and its empirical Bayes counterparts with $p = 50$ (oracle \protect\colorline{Rcol3}, Gaussian EB \protect\colorline{Rcol4}, nonparametric EB \protect\colorline{Rcol5}). In the top and middle rows $\bs\eta = (\eta_1,\ldots,\eta_p)$ is fixed as $\eta_j = s_0 \Phi^{-1}(\frac{j - 0.5}{p})$, $1\le j \le p$, i.e., as quantiles of $N(0,s_0^2)$ distribution with $s_0 = 0.5$ and $1.4$ respectively. In the bottom row, $\bs\eta$ is fixed as the quantiles of the Gaussian mixture $0.75 \cdot N(0,0.5^2) + 0.25 \cdot N(3,0.5^2)$. In each case, the value of $\theta$ is varied along the range of these $\eta_j$ values. We take $\sigma = \tau_1 = \cdots = \tau_p = 1$. Each figure is marked on the top with the corresponding $\theta$ value and the associated probability of the selection event.}
\label{fig:eb-power}
\end{figure}
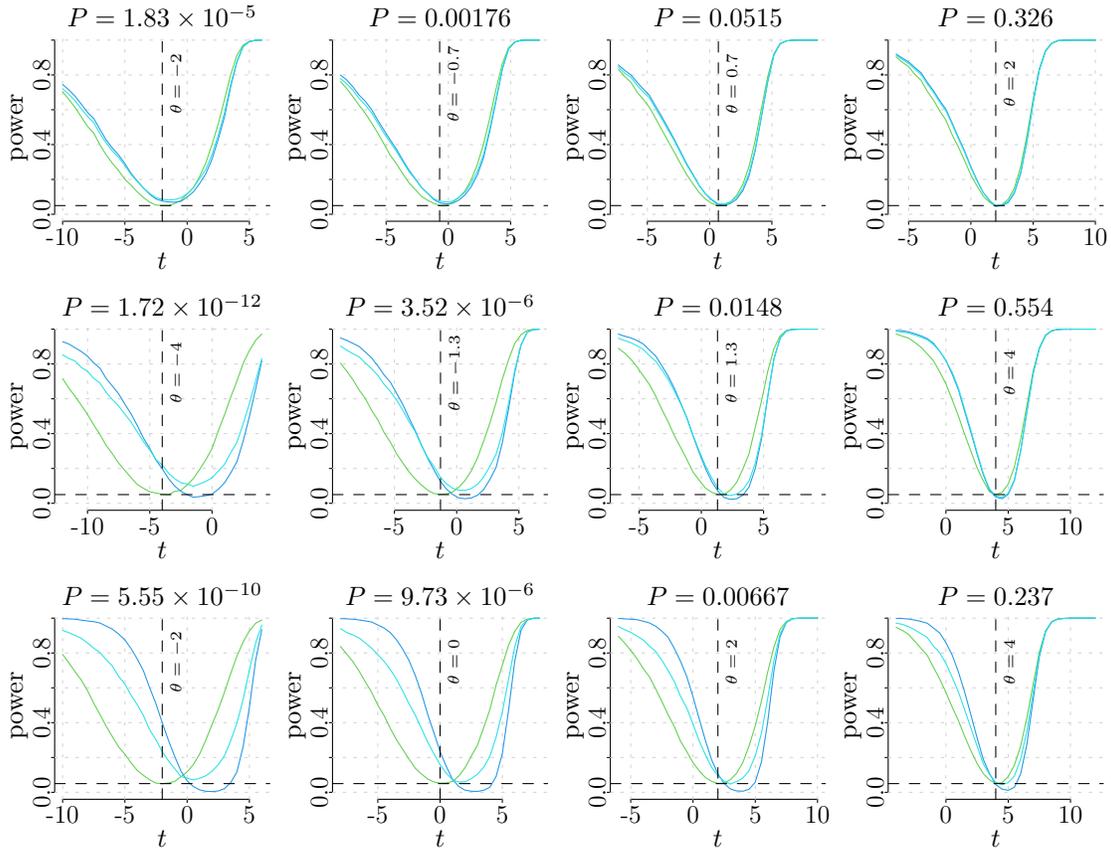

Figure \ref{fig:eb-power} shows that such a Gaussian empirical Bayesian strategy is able to approximate the power function of the oracle method using only $\bs X$ and without any serendipitously accurate prior information. The figure reports results derived from a numerical experiment with $p = 50$, $\sigma = \tau = 1$, and the true $\bs\eta$ fixed as the equispaced quantiles of the $N(0,v)$ distribution, i.e., $\eta_j = \sqrt v \Phi^{-1}(\frac{j - 0.5}{p})$, with $v$ chosen as either $0.5^2$ or $1.4^2$. In the first case, with $\rho = \frac{v}{\tau^2 + v} = 0.33$, the empirical Bayesian method (dark blue lines) closely matches the power function of the oracle (green lines) for a variety of $\theta$ values, including those that result in a small probability of the selection event. In the second case, with a larger $\rho \approx 0.58$, the match between the oracle and the empirical Bayesian method is poorer for $\theta$ values that are extremely unlikely to produce the winner, but it improves quickly as $\theta$ shifts to the right. The difference between these two situations could be anticipated from the heuristic arguments presented above: the smaller the value of $\rho$, the more beneficial empirical Bayesian shrinkage is. But even with a larger value of $\rho$, the empirical Bayesian method could still be effective for a wide range of $\theta$ values.

The same heuristics also suggest that further improvement could be achieved by considering nonparametric empirical Bayes estimates of $\bs\eta$. Consider an intermediate value $y_1$ such that $\eta_j > y_1$ only for indices $j$ in a subset $J \subset \{1,\ldots,p\}$, so that the magnitude of $R(y_1)$ is chiefly determined by the last sum in \eqref{approx} restricted to the same subset $J$. Consequently, a more appealing estimate for the associated $\bs\eta$ would be $\hat\eta_j = \rho^*X_j + (1 - \rho^*)m^*$ with $\rho^* = v^*/(\tau^2 + v^{*})$, where $N(m^*,v^{*})$ is an approximation to the empirical distribution of $\{\eta_i: i \in J\}$. In other words, adaptive localized shrinkage of $X_j$ values within smaller clusters could be more useful than one single global shrinkage. Such localized shrinkage can be obtained within an empirical Bayesian framework by nonparametrically estimating the distribution function $G$ of $\eta_1,\ldots,\eta_p$. Since we do not observe the $\eta_j$ directly, but only observe the corresponding noisy random variable $X_j$, estimation of $G$ falls in the category of mixing distribution estimation. %A popular approach is to estimate $G$ by the nonparametric maximum likelihood estimation method \citep[NPMLE][]{refs}, where the estimate $\hat G$ is discrete with at most $p$ many atoms. 
A computationally attractive approach, which allows estimation of a continuous mixing distribution $G$, is the predictive recursion algorithm of \cite{newton2002nonparametric}, subsequently refined by \cite{tokdar2009consistency} and \cite{martin2011semiparametric}; see Appendix \ref{app:C} for implementation details. We also experimented with estimating $G$ by the nonparametric maximum likelihood method, which produces a discrete estimate, but we omit those details here as the results were slightly inferior to those using  the predictive recursion estimation method. 

The bottom row of Figure \ref{fig:eb-power} presents a case where $\bs\eta$ is fixed at the equispaced quantiles of the Gaussian mixture $0.75 N(0,0.5^2) + 0.25 N(3,0.5^2)$, which has the same variance $v = 1.4^2$ as the Gaussian choice in the middle row. Here, the nonparametric estimator does better than the Gaussian empirical Bayesian method, especially in keeping the size of the test close to the nominal level $\alpha$. The power comparison at the alternative $\theta$ values is more ambiguous, with the Gaussian EB method dominating the nonparametric method for alternative values smaller than the true $\theta$ value, and the nonparametric method doing better on the other side. We will see next that this asymmetry appears to work in favor of the nonparametric method in terms of interval width.

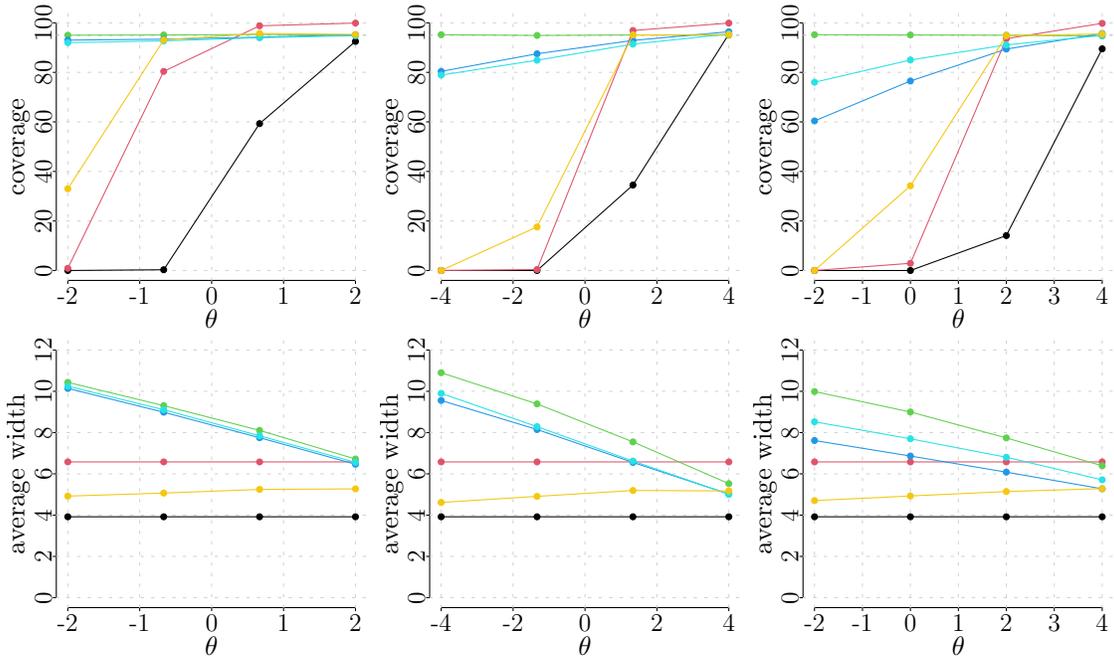
\begin{figure}[!t]
\centering
% Created by tikzDevice version 0.12.6 on 2025-09-10 22:33:25
% !TEX encoding = UTF-8 Unicode
\begin{tikzpicture}[x=1pt,y=1pt]
\definecolor{fillColor}{RGB}{255,255,255}
\path[use as bounding box,fill=fillColor,fill opacity=0.00] (0,0) rectangle (137.31,122.86);
\begin{scope}
\path[clip] (  0.00,  0.00) rectangle (137.31,122.86);
\definecolor{drawColor}{RGB}{0,0,0}

\path[draw=drawColor,line width= 0.4pt,line join=round,line cap=round] ( 22.95, 18.60) -- (131.76, 18.60);

\path[draw=drawColor,line width= 0.4pt,line join=round,line cap=round] ( 22.95, 18.60) -- ( 22.95, 17.40);

\path[draw=drawColor,line width= 0.4pt,line join=round,line cap=round] ( 50.15, 18.60) -- ( 50.15, 17.40);

\path[draw=drawColor,line width= 0.4pt,line join=round,line cap=round] ( 77.36, 18.60) -- ( 77.36, 17.40);

\path[draw=drawColor,line width= 0.4pt,line join=round,line cap=round] (104.56, 18.60) -- (104.56, 17.40);

\path[draw=drawColor,line width= 0.4pt,line join=round,line cap=round] (131.76, 18.60) -- (131.76, 17.40);

\node[text=drawColor,anchor=base,inner sep=0pt, outer sep=0pt, scale=  0.80] at ( 22.95,  9.60) {-2};

\node[text=drawColor,anchor=base,inner sep=0pt, outer sep=0pt, scale=  0.80] at ( 50.15,  9.60) {-1};

\node[text=drawColor,anchor=base,inner sep=0pt, outer sep=0pt, scale=  0.80] at ( 77.36,  9.60) {0};

\node[text=drawColor,anchor=base,inner sep=0pt, outer sep=0pt, scale=  0.80] at (104.56,  9.60) {1};

\node[text=drawColor,anchor=base,inner sep=0pt, outer sep=0pt, scale=  0.80] at (131.76,  9.60) {2};

\path[draw=drawColor,line width= 0.4pt,line join=round,line cap=round] ( 18.60, 22.35) -- ( 18.60,116.11);

\path[draw=drawColor,line width= 0.4pt,line join=round,line cap=round] ( 18.60, 22.35) -- ( 17.40, 22.35);

\path[draw=drawColor,line width= 0.4pt,line join=round,line cap=round] ( 18.60, 41.10) -- ( 17.40, 41.10);

\path[draw=drawColor,line width= 0.4pt,line join=round,line cap=round] ( 18.60, 59.85) -- ( 17.40, 59.85);

\path[draw=drawColor,line width= 0.4pt,line join=round,line cap=round] ( 18.60, 78.61) -- ( 17.40, 78.61);

\path[draw=drawColor,line width= 0.4pt,line join=round,line cap=round] ( 18.60, 97.36) -- ( 17.40, 97.36);

\path[draw=drawColor,line width= 0.4pt,line join=round,line cap=round] ( 18.60,116.11) -- ( 17.40,116.11);

\node[text=drawColor,rotate= 90.00,anchor=base,inner sep=0pt, outer sep=0pt, scale=  0.80] at ( 16.80, 22.35) {0};

\node[text=drawColor,rotate= 90.00,anchor=base,inner sep=0pt, outer sep=0pt, scale=  0.80] at ( 16.80, 41.10) {20};

\node[text=drawColor,rotate= 90.00,anchor=base,inner sep=0pt, outer sep=0pt, scale=  0.80] at ( 16.80, 59.85) {40};

\node[text=drawColor,rotate= 90.00,anchor=base,inner sep=0pt, outer sep=0pt, scale=  0.80] at ( 16.80, 78.61) {60};

\node[text=drawColor,rotate= 90.00,anchor=base,inner sep=0pt, outer sep=0pt, scale=  0.80] at ( 16.80, 97.36) {80};

\node[text=drawColor,rotate= 90.00,anchor=base,inner sep=0pt, outer sep=0pt, scale=  0.80] at ( 16.80,116.11) {100};
\end{scope}
\begin{scope}
\path[clip] ( 18.60, 18.60) rectangle (136.11,119.86);
\definecolor{drawColor}{RGB}{211,211,211}

\path[draw=drawColor,line width= 0.4pt,dash pattern=on 1pt off 3pt ,line join=round,line cap=round] ( 18.60, 22.35) -- (136.11, 22.35);

\path[draw=drawColor,line width= 0.4pt,dash pattern=on 1pt off 3pt ,line join=round,line cap=round] ( 18.60, 41.10) -- (136.11, 41.10);

\path[draw=drawColor,line width= 0.4pt,dash pattern=on 1pt off 3pt ,line join=round,line cap=round] ( 18.60, 59.85) -- (136.11, 59.85);

\path[draw=drawColor,line width= 0.4pt,dash pattern=on 1pt off 3pt ,line join=round,line cap=round] ( 18.60, 78.61) -- (136.11, 78.61);

\path[draw=drawColor,line width= 0.4pt,dash pattern=on 1pt off 3pt ,line join=round,line cap=round] ( 18.60, 97.36) -- (136.11, 97.36);

\path[draw=drawColor,line width= 0.4pt,dash pattern=on 1pt off 3pt ,line join=round,line cap=round] ( 18.60,116.11) -- (136.11,116.11);

\path[draw=drawColor,line width= 0.4pt,dash pattern=on 1pt off 3pt ,line join=round,line cap=round] ( 22.95, 18.60) -- ( 22.95,119.86);

\path[draw=drawColor,line width= 0.4pt,dash pattern=on 1pt off 3pt ,line join=round,line cap=round] ( 50.15, 18.60) -- ( 50.15,119.86);

\path[draw=drawColor,line width= 0.4pt,dash pattern=on 1pt off 3pt ,line join=round,line cap=round] ( 77.36, 18.60) -- ( 77.36,119.86);

\path[draw=drawColor,line width= 0.4pt,dash pattern=on 1pt off 3pt ,line join=round,line cap=round] (104.56, 18.60) -- (104.56,119.86);

\path[draw=drawColor,line width= 0.4pt,dash pattern=on 1pt off 3pt ,line join=round,line cap=round] (131.76, 18.60) -- (131.76,119.86);
\end{scope}
\begin{scope}
\path[clip] (  0.00,  0.00) rectangle (137.31,122.86);
\definecolor{drawColor}{RGB}{0,0,0}

\node[text=drawColor,anchor=base,inner sep=0pt, outer sep=0pt, scale=  1.00] at ( 77.36,  0.60) {\footnotesize$\theta$};

\node[text=drawColor,rotate= 90.00,anchor=base,inner sep=0pt, outer sep=0pt, scale=  1.00] at (  6.60, 69.23) {\footnotesize coverage};
\end{scope}
\begin{scope}
\path[clip] ( 18.60, 18.60) rectangle (136.11,119.86);
\definecolor{drawColor}{gray}{0.70}

\path[draw=drawColor,line width= 0.4pt,dash pattern=on 1pt off 3pt ,line join=round,line cap=round] ( 18.60,111.42) -- (136.11,111.42);
\definecolor{drawColor}{RGB}{0,0,0}

\path[draw=drawColor,line width= 0.4pt,line join=round,line cap=round] ( 22.95, 22.35) --
	( 59.22, 22.63) --
	( 95.49, 77.95) --
	(131.76,109.08);
\definecolor{fillColor}{RGB}{0,0,0}

\path[draw=drawColor,line width= 0.4pt,line join=round,line cap=round,fill=fillColor] ( 22.95, 22.35) circle (  1.12);

\path[draw=drawColor,line width= 0.4pt,line join=round,line cap=round,fill=fillColor] ( 59.22, 22.63) circle (  1.12);

\path[draw=drawColor,line width= 0.4pt,line join=round,line cap=round,fill=fillColor] ( 95.49, 77.95) circle (  1.12);

\path[draw=drawColor,line width= 0.4pt,line join=round,line cap=round,fill=fillColor] (131.76,109.08) circle (  1.12);
\definecolor{drawColor}{RGB}{223,83,107}

\path[draw=drawColor,line width= 0.4pt,line join=round,line cap=round] ( 22.95, 23.19) --
	( 59.22, 97.73) --
	( 95.49,114.98) --
	(131.76,116.01);
\definecolor{fillColor}{RGB}{223,83,107}

\path[draw=drawColor,line width= 0.4pt,line join=round,line cap=round,fill=fillColor] ( 22.95, 23.19) circle (  1.12);

\path[draw=drawColor,line width= 0.4pt,line join=round,line cap=round,fill=fillColor] ( 59.22, 97.73) circle (  1.12);

\path[draw=drawColor,line width= 0.4pt,line join=round,line cap=round,fill=fillColor] ( 95.49,114.98) circle (  1.12);

\path[draw=drawColor,line width= 0.4pt,line join=round,line cap=round,fill=fillColor] (131.76,116.01) circle (  1.12);
\definecolor{drawColor}{RGB}{97,208,79}

\path[draw=drawColor,line width= 0.4pt,line join=round,line cap=round] ( 22.95,111.42) --
	( 59.22,111.51) --
	( 95.49,111.70) --
	(131.76,111.51);
\definecolor{fillColor}{RGB}{97,208,79}

\path[draw=drawColor,line width= 0.4pt,line join=round,line cap=round,fill=fillColor] ( 22.95,111.42) circle (  1.12);

\path[draw=drawColor,line width= 0.4pt,line join=round,line cap=round,fill=fillColor] ( 59.22,111.51) circle (  1.12);

\path[draw=drawColor,line width= 0.4pt,line join=round,line cap=round,fill=fillColor] ( 95.49,111.70) circle (  1.12);

\path[draw=drawColor,line width= 0.4pt,line join=round,line cap=round,fill=fillColor] (131.76,111.51) circle (  1.12);
\definecolor{drawColor}{RGB}{34,151,230}

\path[draw=drawColor,line width= 0.4pt,line join=round,line cap=round] ( 22.95,109.55) --
	( 59.22,109.92) --
	( 95.49,110.67) --
	(131.76,111.61);
\definecolor{fillColor}{RGB}{34,151,230}

\path[draw=drawColor,line width= 0.4pt,line join=round,line cap=round,fill=fillColor] ( 22.95,109.55) circle (  1.12);

\path[draw=drawColor,line width= 0.4pt,line join=round,line cap=round,fill=fillColor] ( 59.22,109.92) circle (  1.12);

\path[draw=drawColor,line width= 0.4pt,line join=round,line cap=round,fill=fillColor] ( 95.49,110.67) circle (  1.12);

\path[draw=drawColor,line width= 0.4pt,line join=round,line cap=round,fill=fillColor] (131.76,111.61) circle (  1.12);
\definecolor{drawColor}{RGB}{40,226,229}

\path[draw=drawColor,line width= 0.4pt,line join=round,line cap=round] ( 22.95,108.61) --
	( 59.22,109.26) --
	( 95.49,110.48) --
	(131.76,111.23);
\definecolor{fillColor}{RGB}{40,226,229}

\path[draw=drawColor,line width= 0.4pt,line join=round,line cap=round,fill=fillColor] ( 22.95,108.61) circle (  1.12);

\path[draw=drawColor,line width= 0.4pt,line join=round,line cap=round,fill=fillColor] ( 59.22,109.26) circle (  1.12);

\path[draw=drawColor,line width= 0.4pt,line join=round,line cap=round,fill=fillColor] ( 95.49,110.48) circle (  1.12);

\path[draw=drawColor,line width= 0.4pt,line join=round,line cap=round,fill=fillColor] (131.76,111.23) circle (  1.12);
\definecolor{drawColor}{RGB}{245,199,16}

\path[draw=drawColor,line width= 0.4pt,line join=round,line cap=round] ( 22.95, 53.29) --
	( 59.22,109.64) --
	( 95.49,112.08) --
	(131.76,111.70);
\definecolor{fillColor}{RGB}{245,199,16}

\path[draw=drawColor,line width= 0.4pt,line join=round,line cap=round,fill=fillColor] ( 22.95, 53.29) circle (  1.12);

\path[draw=drawColor,line width= 0.4pt,line join=round,line cap=round,fill=fillColor] ( 59.22,109.64) circle (  1.12);

\path[draw=drawColor,line width= 0.4pt,line join=round,line cap=round,fill=fillColor] ( 95.49,112.08) circle (  1.12);

\path[draw=drawColor,line width= 0.4pt,line join=round,line cap=round,fill=fillColor] (131.76,111.70) circle (  1.12);
\end{scope}
\end{tikzpicture}
% Created by tikzDevice version 0.12.6 on 2025-09-10 22:33:26
% !TEX encoding = UTF-8 Unicode
\begin{tikzpicture}[x=1pt,y=1pt]
\definecolor{fillColor}{RGB}{255,255,255}
\path[use as bounding box,fill=fillColor,fill opacity=0.00] (0,0) rectangle (137.31,122.86);
\begin{scope}
\path[clip] (  0.00,  0.00) rectangle (137.31,122.86);
\definecolor{drawColor}{RGB}{0,0,0}

\path[draw=drawColor,line width= 0.4pt,line join=round,line cap=round] ( 22.95, 18.60) -- (131.76, 18.60);

\path[draw=drawColor,line width= 0.4pt,line join=round,line cap=round] ( 22.95, 18.60) -- ( 22.95, 17.40);

\path[draw=drawColor,line width= 0.4pt,line join=round,line cap=round] ( 50.15, 18.60) -- ( 50.15, 17.40);

\path[draw=drawColor,line width= 0.4pt,line join=round,line cap=round] ( 77.36, 18.60) -- ( 77.36, 17.40);

\path[draw=drawColor,line width= 0.4pt,line join=round,line cap=round] (104.56, 18.60) -- (104.56, 17.40);

\path[draw=drawColor,line width= 0.4pt,line join=round,line cap=round] (131.76, 18.60) -- (131.76, 17.40);

\node[text=drawColor,anchor=base,inner sep=0pt, outer sep=0pt, scale=  0.80] at ( 22.95,  9.60) {-4};

\node[text=drawColor,anchor=base,inner sep=0pt, outer sep=0pt, scale=  0.80] at ( 50.15,  9.60) {-2};

\node[text=drawColor,anchor=base,inner sep=0pt, outer sep=0pt, scale=  0.80] at ( 77.36,  9.60) {0};

\node[text=drawColor,anchor=base,inner sep=0pt, outer sep=0pt, scale=  0.80] at (104.56,  9.60) {2};

\node[text=drawColor,anchor=base,inner sep=0pt, outer sep=0pt, scale=  0.80] at (131.76,  9.60) {4};

\path[draw=drawColor,line width= 0.4pt,line join=round,line cap=round] ( 18.60, 22.35) -- ( 18.60,116.11);

\path[draw=drawColor,line width= 0.4pt,line join=round,line cap=round] ( 18.60, 22.35) -- ( 17.40, 22.35);

\path[draw=drawColor,line width= 0.4pt,line join=round,line cap=round] ( 18.60, 41.10) -- ( 17.40, 41.10);

\path[draw=drawColor,line width= 0.4pt,line join=round,line cap=round] ( 18.60, 59.85) -- ( 17.40, 59.85);

\path[draw=drawColor,line width= 0.4pt,line join=round,line cap=round] ( 18.60, 78.61) -- ( 17.40, 78.61);

\path[draw=drawColor,line width= 0.4pt,line join=round,line cap=round] ( 18.60, 97.36) -- ( 17.40, 97.36);

\path[draw=drawColor,line width= 0.4pt,line join=round,line cap=round] ( 18.60,116.11) -- ( 17.40,116.11);

\node[text=drawColor,rotate= 90.00,anchor=base,inner sep=0pt, outer sep=0pt, scale=  0.80] at ( 16.80, 22.35) {0};

\node[text=drawColor,rotate= 90.00,anchor=base,inner sep=0pt, outer sep=0pt, scale=  0.80] at ( 16.80, 41.10) {20};

\node[text=drawColor,rotate= 90.00,anchor=base,inner sep=0pt, outer sep=0pt, scale=  0.80] at ( 16.80, 59.85) {40};

\node[text=drawColor,rotate= 90.00,anchor=base,inner sep=0pt, outer sep=0pt, scale=  0.80] at ( 16.80, 78.61) {60};

\node[text=drawColor,rotate= 90.00,anchor=base,inner sep=0pt, outer sep=0pt, scale=  0.80] at ( 16.80, 97.36) {80};

\node[text=drawColor,rotate= 90.00,anchor=base,inner sep=0pt, outer sep=0pt, scale=  0.80] at ( 16.80,116.11) {100};
\end{scope}
\begin{scope}
\path[clip] ( 18.60, 18.60) rectangle (136.11,119.86);
\definecolor{drawColor}{RGB}{211,211,211}

\path[draw=drawColor,line width= 0.4pt,dash pattern=on 1pt off 3pt ,line join=round,line cap=round] ( 18.60, 22.35) -- (136.11, 22.35);

\path[draw=drawColor,line width= 0.4pt,dash pattern=on 1pt off 3pt ,line join=round,line cap=round] ( 18.60, 41.10) -- (136.11, 41.10);

\path[draw=drawColor,line width= 0.4pt,dash pattern=on 1pt off 3pt ,line join=round,line cap=round] ( 18.60, 59.85) -- (136.11, 59.85);

\path[draw=drawColor,line width= 0.4pt,dash pattern=on 1pt off 3pt ,line join=round,line cap=round] ( 18.60, 78.61) -- (136.11, 78.61);

\path[draw=drawColor,line width= 0.4pt,dash pattern=on 1pt off 3pt ,line join=round,line cap=round] ( 18.60, 97.36) -- (136.11, 97.36);

\path[draw=drawColor,line width= 0.4pt,dash pattern=on 1pt off 3pt ,line join=round,line cap=round] ( 18.60,116.11) -- (136.11,116.11);

\path[draw=drawColor,line width= 0.4pt,dash pattern=on 1pt off 3pt ,line join=round,line cap=round] ( 22.95, 18.60) -- ( 22.95,119.86);

\path[draw=drawColor,line width= 0.4pt,dash pattern=on 1pt off 3pt ,line join=round,line cap=round] ( 50.15, 18.60) -- ( 50.15,119.86);

\path[draw=drawColor,line width= 0.4pt,dash pattern=on 1pt off 3pt ,line join=round,line cap=round] ( 77.36, 18.60) -- ( 77.36,119.86);

\path[draw=drawColor,line width= 0.4pt,dash pattern=on 1pt off 3pt ,line join=round,line cap=round] (104.56, 18.60) -- (104.56,119.86);

\path[draw=drawColor,line width= 0.4pt,dash pattern=on 1pt off 3pt ,line join=round,line cap=round] (131.76, 18.60) -- (131.76,119.86);
\end{scope}
\begin{scope}
\path[clip] (  0.00,  0.00) rectangle (137.31,122.86);
\definecolor{drawColor}{RGB}{0,0,0}

\node[text=drawColor,anchor=base,inner sep=0pt, outer sep=0pt, scale=  1.00] at ( 77.36,  0.60) {\footnotesize$\theta$};

\node[text=drawColor,rotate= 90.00,anchor=base,inner sep=0pt, outer sep=0pt, scale=  1.00] at (  6.60, 69.23) {\footnotesize coverage};
\end{scope}
\begin{scope}
\path[clip] ( 18.60, 18.60) rectangle (136.11,119.86);
\definecolor{drawColor}{gray}{0.70}

\path[draw=drawColor,line width= 0.4pt,dash pattern=on 1pt off 3pt ,line join=round,line cap=round] ( 18.60,111.42) -- (136.11,111.42);
\definecolor{drawColor}{RGB}{0,0,0}

\path[draw=drawColor,line width= 0.4pt,line join=round,line cap=round] ( 22.95, 22.35) --
	( 59.22, 22.35) --
	( 95.49, 54.70) --
	(131.76,111.80);
\definecolor{fillColor}{RGB}{0,0,0}

\path[draw=drawColor,line width= 0.4pt,line join=round,line cap=round,fill=fillColor] ( 22.95, 22.35) circle (  1.12);

\path[draw=drawColor,line width= 0.4pt,line join=round,line cap=round,fill=fillColor] ( 59.22, 22.35) circle (  1.12);

\path[draw=drawColor,line width= 0.4pt,line join=round,line cap=round,fill=fillColor] ( 95.49, 54.70) circle (  1.12);

\path[draw=drawColor,line width= 0.4pt,line join=round,line cap=round,fill=fillColor] (131.76,111.80) circle (  1.12);
\definecolor{drawColor}{RGB}{223,83,107}

\path[draw=drawColor,line width= 0.4pt,line join=round,line cap=round] ( 22.95, 22.35) --
	( 59.22, 22.73) --
	( 95.49,113.20) --
	(131.76,116.01);
\definecolor{fillColor}{RGB}{223,83,107}

\path[draw=drawColor,line width= 0.4pt,line join=round,line cap=round,fill=fillColor] ( 22.95, 22.35) circle (  1.12);

\path[draw=drawColor,line width= 0.4pt,line join=round,line cap=round,fill=fillColor] ( 59.22, 22.73) circle (  1.12);

\path[draw=drawColor,line width= 0.4pt,line join=round,line cap=round,fill=fillColor] ( 95.49,113.20) circle (  1.12);

\path[draw=drawColor,line width= 0.4pt,line join=round,line cap=round,fill=fillColor] (131.76,116.01) circle (  1.12);
\definecolor{drawColor}{RGB}{97,208,79}

\path[draw=drawColor,line width= 0.4pt,line join=round,line cap=round] ( 22.95,111.61) --
	( 59.22,111.33) --
	( 95.49,111.51) --
	(131.76,111.42);
\definecolor{fillColor}{RGB}{97,208,79}

\path[draw=drawColor,line width= 0.4pt,line join=round,line cap=round,fill=fillColor] ( 22.95,111.61) circle (  1.12);

\path[draw=drawColor,line width= 0.4pt,line join=round,line cap=round,fill=fillColor] ( 59.22,111.33) circle (  1.12);

\path[draw=drawColor,line width= 0.4pt,line join=round,line cap=round,fill=fillColor] ( 95.49,111.51) circle (  1.12);

\path[draw=drawColor,line width= 0.4pt,line join=round,line cap=round,fill=fillColor] (131.76,111.42) circle (  1.12);
\definecolor{drawColor}{RGB}{34,151,230}

\path[draw=drawColor,line width= 0.4pt,line join=round,line cap=round] ( 22.95, 97.73) --
	( 59.22,104.39) --
	( 95.49,109.45) --
	(131.76,112.73);
\definecolor{fillColor}{RGB}{34,151,230}

\path[draw=drawColor,line width= 0.4pt,line join=round,line cap=round,fill=fillColor] ( 22.95, 97.73) circle (  1.12);

\path[draw=drawColor,line width= 0.4pt,line join=round,line cap=round,fill=fillColor] ( 59.22,104.39) circle (  1.12);

\path[draw=drawColor,line width= 0.4pt,line join=round,line cap=round,fill=fillColor] ( 95.49,109.45) circle (  1.12);

\path[draw=drawColor,line width= 0.4pt,line join=round,line cap=round,fill=fillColor] (131.76,112.73) circle (  1.12);
\definecolor{drawColor}{RGB}{40,226,229}

\path[draw=drawColor,line width= 0.4pt,line join=round,line cap=round] ( 22.95, 96.33) --
	( 59.22,101.95) --
	( 95.49,108.05) --
	(131.76,111.89);
\definecolor{fillColor}{RGB}{40,226,229}

\path[draw=drawColor,line width= 0.4pt,line join=round,line cap=round,fill=fillColor] ( 22.95, 96.33) circle (  1.12);

\path[draw=drawColor,line width= 0.4pt,line join=round,line cap=round,fill=fillColor] ( 59.22,101.95) circle (  1.12);

\path[draw=drawColor,line width= 0.4pt,line join=round,line cap=round,fill=fillColor] ( 95.49,108.05) circle (  1.12);

\path[draw=drawColor,line width= 0.4pt,line join=round,line cap=round,fill=fillColor] (131.76,111.89) circle (  1.12);
\definecolor{drawColor}{RGB}{245,199,16}

\path[draw=drawColor,line width= 0.4pt,line join=round,line cap=round] ( 22.95, 22.35) --
	( 59.22, 38.85) --
	( 95.49,111.51) --
	(131.76,111.70);
\definecolor{fillColor}{RGB}{245,199,16}

\path[draw=drawColor,line width= 0.4pt,line join=round,line cap=round,fill=fillColor] ( 22.95, 22.35) circle (  1.12);

\path[draw=drawColor,line width= 0.4pt,line join=round,line cap=round,fill=fillColor] ( 59.22, 38.85) circle (  1.12);

\path[draw=drawColor,line width= 0.4pt,line join=round,line cap=round,fill=fillColor] ( 95.49,111.51) circle (  1.12);

\path[draw=drawColor,line width= 0.4pt,line join=round,line cap=round,fill=fillColor] (131.76,111.70) circle (  1.12);
\end{scope}
\end{tikzpicture}
\input figures/Oracle-EB-Andrews-p50-nRep10000-selective-performance-nongaussian-coverage\\
% Created by tikzDevice version 0.12.6 on 2025-09-10 22:33:25
% !TEX encoding = UTF-8 Unicode
\begin{tikzpicture}[x=1pt,y=1pt]
\definecolor{fillColor}{RGB}{255,255,255}
\path[use as bounding box,fill=fillColor,fill opacity=0.00] (0,0) rectangle (137.31,122.86);
\begin{scope}
\path[clip] (  0.00,  0.00) rectangle (137.31,122.86);
\definecolor{drawColor}{RGB}{0,0,0}

\path[draw=drawColor,line width= 0.4pt,line join=round,line cap=round] ( 22.95, 18.60) -- (131.76, 18.60);

\path[draw=drawColor,line width= 0.4pt,line join=round,line cap=round] ( 22.95, 18.60) -- ( 22.95, 17.40);

\path[draw=drawColor,line width= 0.4pt,line join=round,line cap=round] ( 50.15, 18.60) -- ( 50.15, 17.40);

\path[draw=drawColor,line width= 0.4pt,line join=round,line cap=round] ( 77.36, 18.60) -- ( 77.36, 17.40);

\path[draw=drawColor,line width= 0.4pt,line join=round,line cap=round] (104.56, 18.60) -- (104.56, 17.40);

\path[draw=drawColor,line width= 0.4pt,line join=round,line cap=round] (131.76, 18.60) -- (131.76, 17.40);

\node[text=drawColor,anchor=base,inner sep=0pt, outer sep=0pt, scale=  0.80] at ( 22.95,  9.60) {-2};

\node[text=drawColor,anchor=base,inner sep=0pt, outer sep=0pt, scale=  0.80] at ( 50.15,  9.60) {-1};

\node[text=drawColor,anchor=base,inner sep=0pt, outer sep=0pt, scale=  0.80] at ( 77.36,  9.60) {0};

\node[text=drawColor,anchor=base,inner sep=0pt, outer sep=0pt, scale=  0.80] at (104.56,  9.60) {1};

\node[text=drawColor,anchor=base,inner sep=0pt, outer sep=0pt, scale=  0.80] at (131.76,  9.60) {2};

\path[draw=drawColor,line width= 0.4pt,line join=round,line cap=round] ( 18.60, 22.35) -- ( 18.60,116.11);

\path[draw=drawColor,line width= 0.4pt,line join=round,line cap=round] ( 18.60, 22.35) -- ( 17.40, 22.35);

\path[draw=drawColor,line width= 0.4pt,line join=round,line cap=round] ( 18.60, 37.98) -- ( 17.40, 37.98);

\path[draw=drawColor,line width= 0.4pt,line join=round,line cap=round] ( 18.60, 53.60) -- ( 17.40, 53.60);

\path[draw=drawColor,line width= 0.4pt,line join=round,line cap=round] ( 18.60, 69.23) -- ( 17.40, 69.23);

\path[draw=drawColor,line width= 0.4pt,line join=round,line cap=round] ( 18.60, 84.86) -- ( 17.40, 84.86);

\path[draw=drawColor,line width= 0.4pt,line join=round,line cap=round] ( 18.60,100.48) -- ( 17.40,100.48);

\path[draw=drawColor,line width= 0.4pt,line join=round,line cap=round] ( 18.60,116.11) -- ( 17.40,116.11);

\node[text=drawColor,rotate= 90.00,anchor=base,inner sep=0pt, outer sep=0pt, scale=  0.80] at ( 16.80, 22.35) {0};

\node[text=drawColor,rotate= 90.00,anchor=base,inner sep=0pt, outer sep=0pt, scale=  0.80] at ( 16.80, 37.98) {2};

\node[text=drawColor,rotate= 90.00,anchor=base,inner sep=0pt, outer sep=0pt, scale=  0.80] at ( 16.80, 53.60) {4};

\node[text=drawColor,rotate= 90.00,anchor=base,inner sep=0pt, outer sep=0pt, scale=  0.80] at ( 16.80, 69.23) {6};

\node[text=drawColor,rotate= 90.00,anchor=base,inner sep=0pt, outer sep=0pt, scale=  0.80] at ( 16.80, 84.86) {8};

\node[text=drawColor,rotate= 90.00,anchor=base,inner sep=0pt, outer sep=0pt, scale=  0.80] at ( 16.80,100.48) {10};

\node[text=drawColor,rotate= 90.00,anchor=base,inner sep=0pt, outer sep=0pt, scale=  0.80] at ( 16.80,116.11) {12};
\end{scope}
\begin{scope}
\path[clip] ( 18.60, 18.60) rectangle (136.11,119.86);
\definecolor{drawColor}{RGB}{211,211,211}

\path[draw=drawColor,line width= 0.4pt,dash pattern=on 1pt off 3pt ,line join=round,line cap=round] ( 18.60, 22.35) -- (136.11, 22.35);

\path[draw=drawColor,line width= 0.4pt,dash pattern=on 1pt off 3pt ,line join=round,line cap=round] ( 18.60, 37.98) -- (136.11, 37.98);

\path[draw=drawColor,line width= 0.4pt,dash pattern=on 1pt off 3pt ,line join=round,line cap=round] ( 18.60, 53.60) -- (136.11, 53.60);

\path[draw=drawColor,line width= 0.4pt,dash pattern=on 1pt off 3pt ,line join=round,line cap=round] ( 18.60, 69.23) -- (136.11, 69.23);

\path[draw=drawColor,line width= 0.4pt,dash pattern=on 1pt off 3pt ,line join=round,line cap=round] ( 18.60, 84.86) -- (136.11, 84.86);

\path[draw=drawColor,line width= 0.4pt,dash pattern=on 1pt off 3pt ,line join=round,line cap=round] ( 18.60,100.48) -- (136.11,100.48);

\path[draw=drawColor,line width= 0.4pt,dash pattern=on 1pt off 3pt ,line join=round,line cap=round] ( 18.60,116.11) -- (136.11,116.11);

\path[draw=drawColor,line width= 0.4pt,dash pattern=on 1pt off 3pt ,line join=round,line cap=round] ( 22.95, 18.60) -- ( 22.95,119.86);

\path[draw=drawColor,line width= 0.4pt,dash pattern=on 1pt off 3pt ,line join=round,line cap=round] ( 50.15, 18.60) -- ( 50.15,119.86);

\path[draw=drawColor,line width= 0.4pt,dash pattern=on 1pt off 3pt ,line join=round,line cap=round] ( 77.36, 18.60) -- ( 77.36,119.86);

\path[draw=drawColor,line width= 0.4pt,dash pattern=on 1pt off 3pt ,line join=round,line cap=round] (104.56, 18.60) -- (104.56,119.86);

\path[draw=drawColor,line width= 0.4pt,dash pattern=on 1pt off 3pt ,line join=round,line cap=round] (131.76, 18.60) -- (131.76,119.86);
\end{scope}
\begin{scope}
\path[clip] (  0.00,  0.00) rectangle (137.31,122.86);
\definecolor{drawColor}{RGB}{0,0,0}

\node[text=drawColor,anchor=base,inner sep=0pt, outer sep=0pt, scale=  1.00] at ( 77.36,  0.60) {\footnotesize$\theta$};

\node[text=drawColor,rotate= 90.00,anchor=base,inner sep=0pt, outer sep=0pt, scale=  1.00] at (  6.60, 69.23) {\footnotesize average width};
\end{scope}
\begin{scope}
\path[clip] ( 18.60, 18.60) rectangle (136.11,119.86);
\definecolor{drawColor}{RGB}{0,0,0}

\path[draw=drawColor,line width= 0.4pt,line join=round,line cap=round] ( 22.95, 52.98) --
	( 59.22, 52.98) --
	( 95.49, 52.98) --
	(131.76, 52.98);
\definecolor{fillColor}{RGB}{0,0,0}

\path[draw=drawColor,line width= 0.4pt,line join=round,line cap=round,fill=fillColor] ( 22.95, 52.98) circle (  1.12);

\path[draw=drawColor,line width= 0.4pt,line join=round,line cap=round,fill=fillColor] ( 59.22, 52.98) circle (  1.12);

\path[draw=drawColor,line width= 0.4pt,line join=round,line cap=round,fill=fillColor] ( 95.49, 52.98) circle (  1.12);

\path[draw=drawColor,line width= 0.4pt,line join=round,line cap=round,fill=fillColor] (131.76, 52.98) circle (  1.12);
\definecolor{drawColor}{RGB}{223,83,107}

\path[draw=drawColor,line width= 0.4pt,line join=round,line cap=round] ( 22.95, 73.75) --
	( 59.22, 73.75) --
	( 95.49, 73.75) --
	(131.76, 73.75);
\definecolor{fillColor}{RGB}{223,83,107}

\path[draw=drawColor,line width= 0.4pt,line join=round,line cap=round,fill=fillColor] ( 22.95, 73.75) circle (  1.12);

\path[draw=drawColor,line width= 0.4pt,line join=round,line cap=round,fill=fillColor] ( 59.22, 73.75) circle (  1.12);

\path[draw=drawColor,line width= 0.4pt,line join=round,line cap=round,fill=fillColor] ( 95.49, 73.75) circle (  1.12);

\path[draw=drawColor,line width= 0.4pt,line join=round,line cap=round,fill=fillColor] (131.76, 73.75) circle (  1.12);
\definecolor{drawColor}{RGB}{97,208,79}

\path[draw=drawColor,line width= 0.4pt,line join=round,line cap=round] ( 22.95,103.87) --
	( 59.22, 95.05) --
	( 95.49, 85.68) --
	(131.76, 74.85);
\definecolor{fillColor}{RGB}{97,208,79}

\path[draw=drawColor,line width= 0.4pt,line join=round,line cap=round,fill=fillColor] ( 22.95,103.87) circle (  1.12);

\path[draw=drawColor,line width= 0.4pt,line join=round,line cap=round,fill=fillColor] ( 59.22, 95.05) circle (  1.12);

\path[draw=drawColor,line width= 0.4pt,line join=round,line cap=round,fill=fillColor] ( 95.49, 85.68) circle (  1.12);

\path[draw=drawColor,line width= 0.4pt,line join=round,line cap=round,fill=fillColor] (131.76, 74.85) circle (  1.12);
\definecolor{drawColor}{RGB}{34,151,230}

\path[draw=drawColor,line width= 0.4pt,line join=round,line cap=round] ( 22.95,101.58) --
	( 59.22, 92.56) --
	( 95.49, 82.91) --
	(131.76, 72.87);
\definecolor{fillColor}{RGB}{34,151,230}

\path[draw=drawColor,line width= 0.4pt,line join=round,line cap=round,fill=fillColor] ( 22.95,101.58) circle (  1.12);

\path[draw=drawColor,line width= 0.4pt,line join=round,line cap=round,fill=fillColor] ( 59.22, 92.56) circle (  1.12);

\path[draw=drawColor,line width= 0.4pt,line join=round,line cap=round,fill=fillColor] ( 95.49, 82.91) circle (  1.12);

\path[draw=drawColor,line width= 0.4pt,line join=round,line cap=round,fill=fillColor] (131.76, 72.87) circle (  1.12);
\definecolor{drawColor}{RGB}{40,226,229}

\path[draw=drawColor,line width= 0.4pt,line join=round,line cap=round] ( 22.95,102.40) --
	( 59.22, 93.51) --
	( 95.49, 83.64) --
	(131.76, 73.59);
\definecolor{fillColor}{RGB}{40,226,229}

\path[draw=drawColor,line width= 0.4pt,line join=round,line cap=round,fill=fillColor] ( 22.95,102.40) circle (  1.12);

\path[draw=drawColor,line width= 0.4pt,line join=round,line cap=round,fill=fillColor] ( 59.22, 93.51) circle (  1.12);

\path[draw=drawColor,line width= 0.4pt,line join=round,line cap=round,fill=fillColor] ( 95.49, 83.64) circle (  1.12);

\path[draw=drawColor,line width= 0.4pt,line join=round,line cap=round,fill=fillColor] (131.76, 73.59) circle (  1.12);
\definecolor{drawColor}{RGB}{245,199,16}

\path[draw=drawColor,line width= 0.4pt,line join=round,line cap=round] ( 22.95, 60.81) --
	( 59.22, 61.96) --
	( 95.49, 63.34) --
	(131.76, 63.55);
\definecolor{fillColor}{RGB}{245,199,16}

\path[draw=drawColor,line width= 0.4pt,line join=round,line cap=round,fill=fillColor] ( 22.95, 60.81) circle (  1.12);

\path[draw=drawColor,line width= 0.4pt,line join=round,line cap=round,fill=fillColor] ( 59.22, 61.96) circle (  1.12);

\path[draw=drawColor,line width= 0.4pt,line join=round,line cap=round,fill=fillColor] ( 95.49, 63.34) circle (  1.12);

\path[draw=drawColor,line width= 0.4pt,line join=round,line cap=round,fill=fillColor] (131.76, 63.55) circle (  1.12);
\end{scope}
\end{tikzpicture}
% Created by tikzDevice version 0.12.6 on 2025-09-10 22:33:26
% !TEX encoding = UTF-8 Unicode
\begin{tikzpicture}[x=1pt,y=1pt]
\definecolor{fillColor}{RGB}{255,255,255}
\path[use as bounding box,fill=fillColor,fill opacity=0.00] (0,0) rectangle (137.31,122.86);
\begin{scope}
\path[clip] (  0.00,  0.00) rectangle (137.31,122.86);
\definecolor{drawColor}{RGB}{0,0,0}

\path[draw=drawColor,line width= 0.4pt,line join=round,line cap=round] ( 22.95, 18.60) -- (131.76, 18.60);

\path[draw=drawColor,line width= 0.4pt,line join=round,line cap=round] ( 22.95, 18.60) -- ( 22.95, 17.40);

\path[draw=drawColor,line width= 0.4pt,line join=round,line cap=round] ( 50.15, 18.60) -- ( 50.15, 17.40);

\path[draw=drawColor,line width= 0.4pt,line join=round,line cap=round] ( 77.36, 18.60) -- ( 77.36, 17.40);

\path[draw=drawColor,line width= 0.4pt,line join=round,line cap=round] (104.56, 18.60) -- (104.56, 17.40);

\path[draw=drawColor,line width= 0.4pt,line join=round,line cap=round] (131.76, 18.60) -- (131.76, 17.40);

\node[text=drawColor,anchor=base,inner sep=0pt, outer sep=0pt, scale=  0.80] at ( 22.95,  9.60) {-4};

\node[text=drawColor,anchor=base,inner sep=0pt, outer sep=0pt, scale=  0.80] at ( 50.15,  9.60) {-2};

\node[text=drawColor,anchor=base,inner sep=0pt, outer sep=0pt, scale=  0.80] at ( 77.36,  9.60) {0};

\node[text=drawColor,anchor=base,inner sep=0pt, outer sep=0pt, scale=  0.80] at (104.56,  9.60) {2};

\node[text=drawColor,anchor=base,inner sep=0pt, outer sep=0pt, scale=  0.80] at (131.76,  9.60) {4};

\path[draw=drawColor,line width= 0.4pt,line join=round,line cap=round] ( 18.60, 22.35) -- ( 18.60,116.11);

\path[draw=drawColor,line width= 0.4pt,line join=round,line cap=round] ( 18.60, 22.35) -- ( 17.40, 22.35);

\path[draw=drawColor,line width= 0.4pt,line join=round,line cap=round] ( 18.60, 37.98) -- ( 17.40, 37.98);

\path[draw=drawColor,line width= 0.4pt,line join=round,line cap=round] ( 18.60, 53.60) -- ( 17.40, 53.60);

\path[draw=drawColor,line width= 0.4pt,line join=round,line cap=round] ( 18.60, 69.23) -- ( 17.40, 69.23);

\path[draw=drawColor,line width= 0.4pt,line join=round,line cap=round] ( 18.60, 84.86) -- ( 17.40, 84.86);

\path[draw=drawColor,line width= 0.4pt,line join=round,line cap=round] ( 18.60,100.48) -- ( 17.40,100.48);

\path[draw=drawColor,line width= 0.4pt,line join=round,line cap=round] ( 18.60,116.11) -- ( 17.40,116.11);

\node[text=drawColor,rotate= 90.00,anchor=base,inner sep=0pt, outer sep=0pt, scale=  0.80] at ( 16.80, 22.35) {0};

\node[text=drawColor,rotate= 90.00,anchor=base,inner sep=0pt, outer sep=0pt, scale=  0.80] at ( 16.80, 37.98) {2};

\node[text=drawColor,rotate= 90.00,anchor=base,inner sep=0pt, outer sep=0pt, scale=  0.80] at ( 16.80, 53.60) {4};

\node[text=drawColor,rotate= 90.00,anchor=base,inner sep=0pt, outer sep=0pt, scale=  0.80] at ( 16.80, 69.23) {6};

\node[text=drawColor,rotate= 90.00,anchor=base,inner sep=0pt, outer sep=0pt, scale=  0.80] at ( 16.80, 84.86) {8};

\node[text=drawColor,rotate= 90.00,anchor=base,inner sep=0pt, outer sep=0pt, scale=  0.80] at ( 16.80,100.48) {10};

\node[text=drawColor,rotate= 90.00,anchor=base,inner sep=0pt, outer sep=0pt, scale=  0.80] at ( 16.80,116.11) {12};
\end{scope}
\begin{scope}
\path[clip] ( 18.60, 18.60) rectangle (136.11,119.86);
\definecolor{drawColor}{RGB}{211,211,211}

\path[draw=drawColor,line width= 0.4pt,dash pattern=on 1pt off 3pt ,line join=round,line cap=round] ( 18.60, 22.35) -- (136.11, 22.35);

\path[draw=drawColor,line width= 0.4pt,dash pattern=on 1pt off 3pt ,line join=round,line cap=round] ( 18.60, 37.98) -- (136.11, 37.98);

\path[draw=drawColor,line width= 0.4pt,dash pattern=on 1pt off 3pt ,line join=round,line cap=round] ( 18.60, 53.60) -- (136.11, 53.60);

\path[draw=drawColor,line width= 0.4pt,dash pattern=on 1pt off 3pt ,line join=round,line cap=round] ( 18.60, 69.23) -- (136.11, 69.23);

\path[draw=drawColor,line width= 0.4pt,dash pattern=on 1pt off 3pt ,line join=round,line cap=round] ( 18.60, 84.86) -- (136.11, 84.86);

\path[draw=drawColor,line width= 0.4pt,dash pattern=on 1pt off 3pt ,line join=round,line cap=round] ( 18.60,100.48) -- (136.11,100.48);

\path[draw=drawColor,line width= 0.4pt,dash pattern=on 1pt off 3pt ,line join=round,line cap=round] ( 18.60,116.11) -- (136.11,116.11);

\path[draw=drawColor,line width= 0.4pt,dash pattern=on 1pt off 3pt ,line join=round,line cap=round] ( 22.95, 18.60) -- ( 22.95,119.86);

\path[draw=drawColor,line width= 0.4pt,dash pattern=on 1pt off 3pt ,line join=round,line cap=round] ( 50.15, 18.60) -- ( 50.15,119.86);

\path[draw=drawColor,line width= 0.4pt,dash pattern=on 1pt off 3pt ,line join=round,line cap=round] ( 77.36, 18.60) -- ( 77.36,119.86);

\path[draw=drawColor,line width= 0.4pt,dash pattern=on 1pt off 3pt ,line join=round,line cap=round] (104.56, 18.60) -- (104.56,119.86);

\path[draw=drawColor,line width= 0.4pt,dash pattern=on 1pt off 3pt ,line join=round,line cap=round] (131.76, 18.60) -- (131.76,119.86);
\end{scope}
\begin{scope}
\path[clip] (  0.00,  0.00) rectangle (137.31,122.86);
\definecolor{drawColor}{RGB}{0,0,0}

\node[text=drawColor,anchor=base,inner sep=0pt, outer sep=0pt, scale=  1.00] at ( 77.36,  0.60) {\footnotesize$\theta$};

\node[text=drawColor,rotate= 90.00,anchor=base,inner sep=0pt, outer sep=0pt, scale=  1.00] at (  6.60, 69.23) {\footnotesize average width};
\end{scope}
\begin{scope}
\path[clip] ( 18.60, 18.60) rectangle (136.11,119.86);
\definecolor{drawColor}{RGB}{0,0,0}

\path[draw=drawColor,line width= 0.4pt,line join=round,line cap=round] ( 22.95, 52.98) --
	( 59.22, 52.98) --
	( 95.49, 52.98) --
	(131.76, 52.98);
\definecolor{fillColor}{RGB}{0,0,0}

\path[draw=drawColor,line width= 0.4pt,line join=round,line cap=round,fill=fillColor] ( 22.95, 52.98) circle (  1.12);

\path[draw=drawColor,line width= 0.4pt,line join=round,line cap=round,fill=fillColor] ( 59.22, 52.98) circle (  1.12);

\path[draw=drawColor,line width= 0.4pt,line join=round,line cap=round,fill=fillColor] ( 95.49, 52.98) circle (  1.12);

\path[draw=drawColor,line width= 0.4pt,line join=round,line cap=round,fill=fillColor] (131.76, 52.98) circle (  1.12);
\definecolor{drawColor}{RGB}{223,83,107}

\path[draw=drawColor,line width= 0.4pt,line join=round,line cap=round] ( 22.95, 73.75) --
	( 59.22, 73.75) --
	( 95.49, 73.75) --
	(131.76, 73.75);
\definecolor{fillColor}{RGB}{223,83,107}

\path[draw=drawColor,line width= 0.4pt,line join=round,line cap=round,fill=fillColor] ( 22.95, 73.75) circle (  1.12);

\path[draw=drawColor,line width= 0.4pt,line join=round,line cap=round,fill=fillColor] ( 59.22, 73.75) circle (  1.12);

\path[draw=drawColor,line width= 0.4pt,line join=round,line cap=round,fill=fillColor] ( 95.49, 73.75) circle (  1.12);

\path[draw=drawColor,line width= 0.4pt,line join=round,line cap=round,fill=fillColor] (131.76, 73.75) circle (  1.12);
\definecolor{drawColor}{RGB}{97,208,79}

\path[draw=drawColor,line width= 0.4pt,line join=round,line cap=round] ( 22.95,107.51) --
	( 59.22, 95.71) --
	( 95.49, 81.33) --
	(131.76, 65.52);
\definecolor{fillColor}{RGB}{97,208,79}

\path[draw=drawColor,line width= 0.4pt,line join=round,line cap=round,fill=fillColor] ( 22.95,107.51) circle (  1.12);

\path[draw=drawColor,line width= 0.4pt,line join=round,line cap=round,fill=fillColor] ( 59.22, 95.71) circle (  1.12);

\path[draw=drawColor,line width= 0.4pt,line join=round,line cap=round,fill=fillColor] ( 95.49, 81.33) circle (  1.12);

\path[draw=drawColor,line width= 0.4pt,line join=round,line cap=round,fill=fillColor] (131.76, 65.52) circle (  1.12);
\definecolor{drawColor}{RGB}{34,151,230}

\path[draw=drawColor,line width= 0.4pt,line join=round,line cap=round] ( 22.95, 97.01) --
	( 59.22, 86.02) --
	( 95.49, 73.56) --
	(131.76, 61.58);
\definecolor{fillColor}{RGB}{34,151,230}

\path[draw=drawColor,line width= 0.4pt,line join=round,line cap=round,fill=fillColor] ( 22.95, 97.01) circle (  1.12);

\path[draw=drawColor,line width= 0.4pt,line join=round,line cap=round,fill=fillColor] ( 59.22, 86.02) circle (  1.12);

\path[draw=drawColor,line width= 0.4pt,line join=round,line cap=round,fill=fillColor] ( 95.49, 73.56) circle (  1.12);

\path[draw=drawColor,line width= 0.4pt,line join=round,line cap=round,fill=fillColor] (131.76, 61.58) circle (  1.12);
\definecolor{drawColor}{RGB}{40,226,229}

\path[draw=drawColor,line width= 0.4pt,line join=round,line cap=round] ( 22.95, 99.70) --
	( 59.22, 87.10) --
	( 95.49, 74.07) --
	(131.76, 61.53);
\definecolor{fillColor}{RGB}{40,226,229}

\path[draw=drawColor,line width= 0.4pt,line join=round,line cap=round,fill=fillColor] ( 22.95, 99.70) circle (  1.12);

\path[draw=drawColor,line width= 0.4pt,line join=round,line cap=round,fill=fillColor] ( 59.22, 87.10) circle (  1.12);

\path[draw=drawColor,line width= 0.4pt,line join=round,line cap=round,fill=fillColor] ( 95.49, 74.07) circle (  1.12);

\path[draw=drawColor,line width= 0.4pt,line join=round,line cap=round,fill=fillColor] (131.76, 61.53) circle (  1.12);
\definecolor{drawColor}{RGB}{245,199,16}

\path[draw=drawColor,line width= 0.4pt,line join=round,line cap=round] ( 22.95, 58.41) --
	( 59.22, 60.71) --
	( 95.49, 62.93) --
	(131.76, 62.78);
\definecolor{fillColor}{RGB}{245,199,16}

\path[draw=drawColor,line width= 0.4pt,line join=round,line cap=round,fill=fillColor] ( 22.95, 58.41) circle (  1.12);

\path[draw=drawColor,line width= 0.4pt,line join=round,line cap=round,fill=fillColor] ( 59.22, 60.71) circle (  1.12);

\path[draw=drawColor,line width= 0.4pt,line join=round,line cap=round,fill=fillColor] ( 95.49, 62.93) circle (  1.12);

\path[draw=drawColor,line width= 0.4pt,line join=round,line cap=round,fill=fillColor] (131.76, 62.78) circle (  1.12);
\end{scope}
\end{tikzpicture}
\input figures/Oracle-EB-Andrews-p50-nRep10000-selective-performance-nongaussian-width\\
\caption{Coverage and average width of empirical Bayesian confidence procedures based on the oracle (oracle \protect\colorline{Rcol3}, Gaussian EB \protect\colorline{Rcol4}, nonparametric EB \protect\colorline{Rcol5}). Left and middle columns are for the Gaussian experiments with scale 0.5 and 1.4 respectively, and the right column is for the Gaussian mixture setting. Also included are the ordinary confidence interval (\protect\colorline{Rcol1}) and its Bonferroni adjustment (\protect\colorline{Rcol2}), as well as the hybrid interval (\protect\colorline{Rcol7}) of \cite{andrews2024inference}. Both the Bonferroni-adjusted and the hybrid procedures maintain marginal coverage.}
\label{fig:eb-result}
\end{figure}

Figure \ref{fig:eb-result} shows how the 95\% confidence procedures associated with the two empirical Bayesian methods perform in terms of selective coverage and average width. For comparison, the figure also shows performance of the unadjusted interval ($Y \pm 1.96 \sigma$), its Bonferroni adjusted version ($Y \pm \Phi^{-1}(1 - \frac{0.05}{2p}) \sigma = Y \pm 3.29\sigma$), and the hybrid method of \cite{andrews2024inference}. All of these benchmark methods have relatively low selective coverage for small $\theta$ values. In contrast, for the Gaussian $\bs\eta$ experiment with low spread of $\eta_j$'s, the selective coverage of the empirical Bayesian methods is close to the nominal level over a wide range of $\theta$ values, including the case of $\theta = -2$ which corresponds to a selection probability of $1.8\times 10^{-5}$. Additionally, the average width is comparable to that of the oracle method. These results are expected because of the close resemblance between the associated power functions that we saw in Figure \ref{fig:eb-power}. For the second Gaussian experiment with a larger $v$ value, the coverage guarantees drop off as $\theta$ is shifted to the left, but still remain reasonably high ($\approx 80\%$) even when selection events for such a $\theta$ are  extremely rare (e.g., $<2\times 10^{-12}$ for $\theta = -4$). Here again the empirical Bayesian methods give similar selective performances which are far superior to those of the benchmarks. For the Gaussian mixture experiment, which has the same spread of $\eta_j$'s as the second Gaussian experiment, the nonparametric empirical Bayesian method performs about the same as it does in the second Gaussian experiment, but the Gaussian empirical Bayesian method offers poorer selective coverage for very small $\theta$ values. 
We note that the hybrid method of \cite{andrews2024inference} appears to offer selective coverage comparable to the Bonferroni adjusted interval. This phenomenon is partially explained by the fact that the hybrid method restricts the final interval to be enclosed within a further relaxed Bonferroni interval $Y \pm \Phi^{-1}(1 - \frac{0.05\times 0.1}{2p})\sigma = Y \pm 3.89\sigma$, and hence inherits the low selective coverage properties of this enclosing interval. 

In general, the expected widths of the procedures are positively related to their 
selective coverage control. The unadjusted, Bonferroni, and hybrid method have, in order, the smallest average interval widths and the lowest selective coverage control, at least for the rare events. The empirical Bayes procedures maintain reasonable 
selective coverage control even for extremely rare events, 
but have wider interval widths. 
As we might expect, the relative ordering of the average widths of these procedures  appears perfectly correlated with the relative ordering of the power functions to the left of the null value in Figure \ref{fig:eb-power}. 
Among all procedures evaluated, the oracle procedure is the only one with exact selective coverage control for 
all values of $\theta$, and correspondingly, generally 
has the highest expected interval widths.
The implication of these numerical results is the same 
as that of the theoretical results in Section 2 - the price to be paid for selective 
coverage control is wider expected interval width, even for an oracle procedure.

%Given a vector $\bs X = (X_1,\ldots,X_n)$ of independent observations $X_i \sim p_i(x) := \int k_i(x|\theta)g(\theta)d\theta$, the predictive recursion algorithm starts with an initial guess $g_0(\theta)$ of $g(\theta)$ and then recursively updates it as
%\[
%g_i(\theta) = (1 - w_i) g_{i-1}(\theta) + (1 - w_i) \frac{k_i(X_i|\theta)g_{i-1}(\theta)}{\int k_i(X_i|\theta')g_{i-1}(\theta')d\theta'},~i = 1,\ldots,n,
%\]
%and produces $g_n(\theta)$ as the final estimate, where $w_i \in (0,1)$, $i \ge 1$, are predetermined weights, typically taken to be $w_i \propto (1+ i)^{-b}$ for some $b \in (1/2,1]$ to satisfy the asymptotic conditions $\sum_i w_i = \infty$ and $\sum_i w_i^2 < \infty$ under which $g_n(\theta)$ can be shown to converge weakly to the true density $g(\theta)$ when all $k_i(\cdot|\theta)$ equal a single kernel $k(\cdot|\theta)$ that satisfies some regularity conditions and $g(\theta)$ is compactly supported \citep{tokdar2009consistency}. For our application we take $\bs X = \bs Y_{\!\!\!-s}$, $n = p - 1$, and $k_i(x|\theta) = f(x|\theta,\sigma_i)$.% = f(x|\theta,\sigma)I(x < Y_s)/F(Y_s|\theta,\sigma)$.

\section{Discussion} 
It is well-known that 
estimation and inference procedures for data-selected populations 
that do not account for the selection process can be misleading, giving 
biased estimates, and tests and confidence intervals with poor error 
rate control \citep{benjamini_2010,taylor_tibshirani_2015}.  
Selection-adjusted procedures can 
be constructed that do maintain some type of error rate control, but 
the type of error control that is maintained will determine how powerful or 
precise the resulting tests and confidence intervals can be. 
In this article, we have studied the relationship between confidence interval 
precision and selective coverage rates 
  for the mean of the ``winning''  population 
in the multiple normal means model. 
We have found, not surprisingly, that they are 
inversely related - procedures with good selective coverage control are 
wider than those that only have marginal coverage control. 

However, the choice of 
procedure should be driven primarily by the type of error rate control that 
is most relevant for the inference to be made, with expected width being a 
secondary consideration. 
Recall the example from the Introduction, where the data
correspond to educational outcomes of different schools in a school system. It was 
argued that a system superintendent who oversees all 
of the schools 
may be more interested in marginal coverage control, whereas the staff of a given school may,  
upon their selection, be primarily interested in selective coverage control.  
The rationale for this divergence is that 
the different types of errors have different consequences for the two parties. 

In this article our evaluation criteria have been frequentist, and the procedures we have studied have been  frequentist 
in nature, in that they were derived from the inversion 
of level-$\alpha$ 
hypothesis tests. 
It is worth considering how the marginal perspective (that of the superintendent) 
and the selective perspective
(that of the staff of the selected school, or an ``underdog'') diverge when using purely Bayesian methods. 
Without going into extensive details, a Bayesian interested in only marginal control 
could proceed simply by constructing standard posterior credible intervals: If
$\mu_1,\ldots,\mu_{p+1}$ are i.i.d.\ $\pi$, then marginally over $\bs\mu$,  the probability 
that the $\mu$-value of the winning group being is in its credible interval is 
exactly $1-\alpha$. Conditionally on $\bs\mu$ but marginally over the winning group, this 
coverage probability will still be close to $1-\alpha$ if $\pi(\bs\mu)$ is a reasonable approximation to the empirical 
distribution of $\mu_1,\ldots, \mu_{p+1}$, but as usual, potentially far from $1-\alpha$ 
if the the prior $\pi$ is inaccurate.  
However, even if the prior is accurate, frequentist selective coverage control will not 
be maintained uniformly across selection events. 

In contrast, from the perspective of an underdog, or of a selected school, all inferences are conditional in the selection event, and so the appropriate model for inference has 
densities $\{ \psel(y| \theta,\bs\eta )\times \pcon(\bs x| y,\bs\eta) : (\theta, \bs\eta)\in \mathbb R^{p+1}\}$ 
given by (\ref{eqn:fjcdensity}). Unlike the unadjusted credible interval just described, 
 a posterior credible interval for $\theta$ will have reasonable selective coverage control 
as long as the prior is not inaccurate. In particular, if the prior for $\theta$ is 
diffuse, and the prior for $\bs\eta$  resembles the empirical distribution of $\eta_1,\ldots, \eta_p$, we expect the Bayesian credible interval procedure to closely resemble the oracle and 
empirical Bayes procedures described in the article. 

From a methodological point of view, 
we have also shown empirically that, in this type of multipopulation scenario, 
an oracle procedure that has finite expected width and exact 
selective coverage control can be reasonably approximated by empirical Bayes 
procedures that estimate the nuisance parameters in the selection-adjusted 
sampling model for the selected group. 
Such hybrid frequentist-empirical Bayes procedures may be useful in other 
selective inference procedures where selective coverage control is desired but cannot be exactly maintained without the knowledge of nuisance parameters.

\bibliographystyle{plainnat}
\bibliography{refs} 

\begin{thebibliography}{16}
\providecommand{\natexlab}[1]{#1}
\providecommand{\url}[1]{\texttt{#1}}
\expandafter\ifx\csname urlstyle\endcsname\relax
  \providecommand{\doi}[1]{doi: #1}\else
  \providecommand{\doi}{doi: \begingroup \urlstyle{rm}\Url}\fi

\bibitem[Andrews et~al.(2024)Andrews, Kitagawa, and
  McCloskey]{andrews2024inference}
Isaiah Andrews, Toru Kitagawa, and Adam McCloskey.
\newblock Inference on winners.
\newblock \emph{The Quarterly Journal of Economics}, 139\penalty0 (1):\penalty0
  305--358, 2024.

\bibitem[Benjamini(2010)]{benjamini_2010}
Yoav Benjamini.
\newblock Simultaneous and selective inference: Current successes and future
  challenges.
\newblock \emph{Biometrical Journal}, 52\penalty0 (6):\penalty0 708--721, 2010.

\bibitem[Dawid(1994)]{dawid1994selection}
AP~Dawid.
\newblock Selection paradoxes of {B}ayesian inference.
\newblock \emph{Lecture Notes-Monograph Series}, pages 211--220, 1994.

\bibitem[Efron(2011)]{efron2011tweedie}
Bradley Efron.
\newblock Tweedie’s formula and selection bias.
\newblock \emph{Journal of the American Statistical Association}, 106\penalty0
  (496):\penalty0 1602--1614, 2011.

\bibitem[Gelman(2006)]{gelman2006prior}
A~Gelman.
\newblock Prior distributions for variance parameters in hierarchical models
  (comment on an article by browne and draper).
\newblock \emph{Bayesian Analysis}, 1:\penalty0 515--533, 2006.

\bibitem[Ghosh(1961)]{ghosh_1961}
Jayanta~Kumar Ghosh.
\newblock On the relation among shortest confidence intervals of different
  types.
\newblock \emph{Calcutta Statist. Assoc. Bull.}, 10:\penalty0 147--152, 1961.
\newblock ISSN 0008-0683.
\newblock \doi{10.1177/0008068319610404}.
\newblock URL \url{https://doi.org/10.1177/0008068319610404}.

\bibitem[Lehmann and Romano(2005)]{lehmann_romano_2005}
E.~L. Lehmann and Joseph~P. Romano.
\newblock \emph{Testing statistical hypotheses}.
\newblock Springer Texts in Statistics. Springer, New York, third edition,
  2005.
\newblock ISBN 0-387-98864-5.

\bibitem[Martin and Tokdar(2011)]{martin2011semiparametric}
Ryan Martin and Surya~T Tokdar.
\newblock Semiparametric inference in mixture models with predictive recursion
  marginal likelihood.
\newblock \emph{Biometrika}, 98\penalty0 (3):\penalty0 567--582, 2011.

\bibitem[Newton(2002)]{newton2002nonparametric}
Michael~A Newton.
\newblock On a nonparametric recursive estimator of the mixing distribution.
\newblock \emph{Sankhy{\=a}: The Indian Journal of Statistics, Series A}, pages
  306--322, 2002.

\bibitem[Pratt(1961)]{pratt_1961}
John~W. Pratt.
\newblock Length of confidence intervals.
\newblock \emph{J. Amer. Statist. Assoc.}, 56:\penalty0 549--567, 1961.
\newblock ISSN 0162-1459,1537-274X.
\newblock URL
  \url{http://links.jstor.org/sici?sici=0162-1459(196109)56:295<549:LOCI>2.0.CO;2-C&origin=MSN}.

\bibitem[Snijders and Bosker(2011)]{snijders2011multilevel}
Tom~AB Snijders and Roel Bosker.
\newblock \emph{Multilevel analysis: An introduction to basic and advanced
  multilevel modeling}.
\newblock sage, 2011.

\bibitem[Taylor and Tibshirani(2015)]{taylor_tibshirani_2015}
Jonathan Taylor and Robert~J. Tibshirani.
\newblock Statistical learning and selective inference.
\newblock \emph{Proc. Natl. Acad. Sci. USA}, 112\penalty0 (25):\penalty0
  7629--7634, 2015.
\newblock ISSN 0027-8424,1091-6490.
\newblock \doi{10.1073/pnas.1507583112}.
\newblock URL \url{https://doi.org/10.1073/pnas.1507583112}.

\bibitem[Tokdar et~al.(2009)Tokdar, Martin, and Ghosh]{tokdar2009consistency}
Surya~T Tokdar, Ryan Martin, and Jayanta~K Ghosh.
\newblock Consistency of a recursive estimate of mixing distributions.
\newblock \emph{The Annals of Statistics}, pages 2502--2522, 2009.

\bibitem[Yekutieli(2012)]{yekutieli2012adjusted}
Daniel Yekutieli.
\newblock Adjusted bayesian inference for selected parameters.
\newblock \emph{Journal of the Royal Statistical Society Series B: Statistical
  Methodology}, 74\penalty0 (3):\penalty0 515--541, 2012.

\bibitem[Yu and Hoff(2018)]{yu2018adaptive}
Chaoyu Yu and Peter~D Hoff.
\newblock Adaptive multigroup confidence intervals with constant coverage.
\newblock \emph{Biometrika}, 105\penalty0 (2):\penalty0 319--335, 2018.

\bibitem[Zrnic and Fithian(2024)]{zrnic_fithian_2024}
Tijana Zrnic and William Fithian.
\newblock A flexible defense against the winner's curse, 2024.
\newblock URL \url{https://arxiv.org/abs/2411.18569}.

\end{thebibliography}

\appendix 
\section{Auxiliary results}
\label{app:A}
%The calculations below are done for the case of $p =2$ and $\sigma_1 = \sigma_2 = 1$. 
%For selection-specific and complete conditional guarantee calculations, it is sufficient to fix $S = 2$. 
For any $(\eta,\theta) \in \bbR^2$ let $Q_{\eta,\theta} = N(\eta,1)\times N(\theta,1)$ and $P_{\eta,\theta} = Q_{\eta,\theta}|_{\bbU}$ where $\bbU = \{(x,y) \in \bbR^2: x \le y\}$. Clearly, our focus is on reporting a confidence set for $\theta$ based on a paired observation $(X,Y) \sim P_{\eta,\theta}$. 
The density of $P_{\eta,\theta}$ can be written as
\[
\textstyle p(x,y|\eta,\theta) = \frac{1}{2\pi c(\eta - \theta)}\exp\{-\frac{(x - \eta)^2 + (y-\theta)^2}{2}\} \times 1((x,y) \in \bbU),
\]
with $c(\eta) := Q_{\eta,0}(\bbU)$, due to the fact that $Q_{\eta,\theta}(\bbU) = Q_{\eta - \theta,0}(\bbU)$. Clearly, $c(0) = \frac12$. Indeed, the following statements can be made about $c(\eta)$. Below $\phi(x)$ and $\Phi(x)$ denote the density and distribution function of the standard normal distribution. 

\begin{lemma}
\label{lem00}
$c(\eta) = \Phi(-\frac{\eta}{\sqrt2})$ and $\lim_{\eta \to \infty}\sqrt\pi\eta e^{\eta^2/4}c(\eta) = 1$. 
\end{lemma}

\begin{proof}
Let $Z_1,Z_2$ be independent standard normal variables. Then, $c(\eta) = Q_{\eta,0}(\bbU) = \Pr(Z_1 + \eta \le Z_2) = \Pr(\frac{Z_1 - Z_2}{\sqrt 2} \le -\frac{\eta}{\sqrt 2}) = \Phi(-\frac{\eta}{\sqrt 2})$ since $\frac{Z_1 - Z_2}{\sqrt 2}$ is also a standard normal variable. Apply the well known Mill's ratio inequalities for the standard normal distribution, namely,
\[
\frac{x}{1 + x^2} < \frac{\Phi(-x)}{\phi(x)} < \frac 1x, \quad x > 0
\]
to conclude
\begin{equation}
\frac{\eta^2/2}{1 + \eta^2/2} < \sqrt{\pi}\eta e^{\eta^2/4}c(\eta) < 1, \quad \eta > 0,
\end{equation}
and hence $\lim_{\eta \to \infty}\sqrt\pi\eta e^{\eta^2/4}c(\eta)$ exists and must equal 1. 
\end{proof}

Exact analytical formulas for probabilities under $P_{\eta,\theta}$ are hard to obtain. However, we will shortly establish the crucial result that much of the probability under $P_{\eta,0}$ concentrates on the circle section (see Figure \ref{fig:B})
\begin{equation}
\label{eq:set-b}
B_{\eta,\Delta} = \{(x,y) \in \bbU: (x - \eta)^2 + y^2 \le \tfrac{\eta^2}2 + \Delta^2\}
\end{equation}
universally across $\eta \ge 0$ for a fixed and large $\Delta > 0$. There are multiple integral formulas for expressing $P_{\eta,0}(B_{\eta,\Delta})$. One set of such formulas can be derived from the following considerations. Again, let $Z_1,Z_2$ be independent standard normal variables. Then, $U = \frac{Z_2 - Z_1}{\sqrt 2}$ and $V = \frac{Z_1 + Z_2}{\sqrt 2}$ also are independent standard normal variables, and 
\[
Q_{\eta,0}(B_{\eta,\Delta}) = \Pr(Z_1 + \eta \le Z_2, Z_1^2 + Z_2^2 \le r^2) = \Pr(U \ge \tfrac\eta{\sqrt 2}, U^2 + V^2 \le r^2).
\]
This last expression immediately suggests the integral formula
\begin{align}
P_{\eta,0}(B_{\eta,\Delta}) = \frac{Q_{\eta,0}(B_{\eta,\Delta})}{c(\eta)}
& = 2\int_0^\Delta \phi(v) \frac{\Phi(\sqrt{{\eta^2}/2 + \Delta^2 - v^2}) - \Phi(\eta/{\sqrt 2})}{1 - \Phi(\eta/\sqrt2)}dv \label{eq:prob-b-2},
\end{align}
which will prove valuable in establishing sharp universal bounds on $P_{\eta,0}(B_{\eta,\Delta})$. Before proceeding, we note a useful elementary result relating to normal distributions. 

\begin{figure}[t]
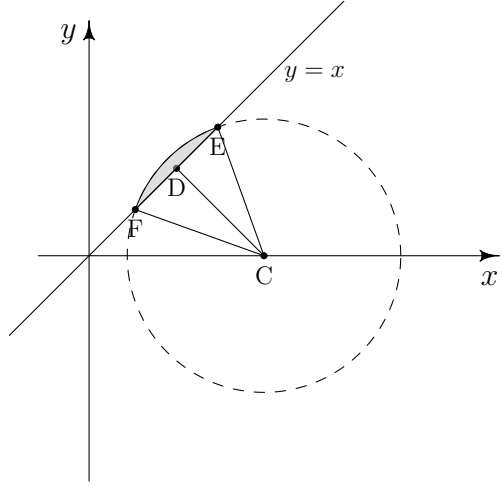

\centering
\input figures/setB
\caption{The shaded region is the set $B_{\eta,\Delta}$ (for a given $\Delta > 0$) which is the intersection of the half plane $\bbU$ with  the circle with center C = $(\eta,0)$ and of radius $r$ satisfying $r^2 = \frac{\eta^2}2+\Delta^2$. The segment CD is perpendicular to the diagonal line $y = x$ with D = $(\frac\eta2,\frac\eta2)$ and has length $\eta/\sqrt 2$. The diagonal line intersects the circle at E $=(\frac\eta2+\frac\Delta{\sqrt2},\frac\eta2+\frac\Delta{\sqrt2})$ and F $=(\frac\eta2-\frac\Delta{\sqrt2},\frac\eta2-\frac\Delta{\sqrt2})$. Both ED and FD have length $\Delta$ each.}
\label{fig:B}
\end{figure}

\begin{lemma}
\label{lem:ineq1}
$g(x) := \frac{\Phi(\sqrt{x^2 + a^2}) - \Phi(x)}{1 - \Phi(x)}$ decreases in $x > 0$ with $\lim_{x \to \infty} g(x) = 1  - e^{-\frac{a^2}2}$.
\end{lemma}

\begin{proof}
Rewrite $g(x) = 1 - \Phi(-\sqrt{x^2 + a^2})/\Phi(-x)$. Apply the Mill's ratio inequalities, namely,
\[
\frac{x}{1 + x^2} < \frac{\Phi(-x)}{\phi(x)} < \frac 1x, \quad x > 0
\]
to bound
\begin{equation}
\frac{1 + x^2 + a^2}{x\sqrt{x^2 + a^2}} e^{-\frac{a^2}{2}} < \frac{\Phi(-\sqrt{x^2 + a^2})}{\Phi(-x)} < \frac{1+x^2}{x\sqrt{x^2 + a^2}}e^{-\frac{a^2}2}
\label{eq:tail-ratio}
\end{equation}
from which it follows immediately that $\lim_{x \to \infty} g(x) = e^{-a^2/2}$. To see why $g(x)$ is decreasing note that its derivative
\[
g'(x) = -\frac{\phi(x)\{\frac{\Phi(-\sqrt{x^2 + a^2})}{\Phi(-x)} - \frac{xe^{-a^2/2}}{\sqrt{x^2 + a^2}} \}}{\Phi(-x)} < 0
\]
because the lower bound in Equation \eqref{eq:tail-ratio} is larger than $\frac{xe^{-a^2/2}}{\sqrt{x^2 + a^2}}$. 
\end{proof}

Next, we present results on universal upper and lower bounds on $P_{\eta,0}(B_{\eta,\Delta})$ across all $\eta \ge 0$ and all $\Delta > 0$. 
Note that the cumulative distribution function of a chi-squared random variable with 3 degrees of freedom equals 
\begin{equation}
F_{\chi^2_3}(x) = 2\Phi(\sqrt x) - 1 - \frac{\sqrt xe^{-x/2}}{\sqrt{\pi/2}}, \quad x > 0.
\label{eq:chisq}
\end{equation}
This holds because $F_{\chi^2_3}(x)  = \frac1{\sqrt{2\pi}}\int_0^x{\sqrt z e^{-z/2}}dz = \frac2{\sqrt2\pi}\int_0^{\sqrt{x}}y^2 e^{-y^2/2}dy$ by substituting $y = \sqrt z$. Therefore, by integration by parts
\[
F_{\chi^2_3}(x) = \tfrac2{\sqrt{2\pi}}\int_0^{\sqrt{x}}y\tfrac{d}{dy}(-e^{-y^2/2})dy = \tfrac2{\sqrt{2\pi}}y(-e^{-y^2/2})\big|_0^{\sqrt x} + \tfrac2{\sqrt{2\pi}}\int_0^{\sqrt{x}}e^{-y^2/2}dy
\]
which immediately gives the identity in \eqref{eq:chisq}.

\begin{lemma}
\label{lem:biglemma}
The following statements hold for any $\Delta > 0$:
\begin{enumerate}
\item $P_{\eta,0}(B_{\eta,\Delta})$ is continuous and monotonically decreasing in $\eta \ge 0$. 
\item $\lim_{\eta \to 0} P_{\eta,0}(B_{\eta,\Delta}) = P_{0,0}(B_{0,\Delta}) = 1- e^{-\Delta^2/2}$.
\item $\lim_{\eta \to \infty} P_{\eta,0}(B_{\eta,\Delta}) = F_{\chi^2_3}(\Delta^2)$. 
\item $1 - e^{-\Delta^2/2} \ge P_{\eta,0}(B_{\eta,\Delta}) > F_{\chi^2_3}(\Delta^2)$ for every $\eta \ge 0$. 
\end{enumerate}
\end{lemma}

\begin{proof}
The last statement is an easy consequence of the first three. The integral representation \eqref{eq:prob-b-2}, coupled with the dominated convergence theorem, establishes continuity of $P_{\eta,0}(B_{\eta,\Delta})$ in $\eta \ge 0$. That $P_{\eta,0}(B_{\eta,\Delta})$ is monotonically decreasing in $\eta$ can be readily concluded since the integrand in \eqref{eq:prob-b-2} is monotonically decreasing in $\eta$ due to Lemma \ref{lem:ineq1}. The left limit result is a consequence of continuity and the fact that $P_{0,0}(B_{0,\Delta}) = 2\Pr(U \ge 0, U^2 + V^2 \le \Delta^2) = \Pr(U^2  + V^2 \le \Delta^2) = 1 - e^{-\Delta^2/2}$. For the right limit, apply Lemma \ref{lem:ineq1} and the monotone convergence theorem to argue
\[
\lim_{\eta \to \infty} P_{\eta,0}(B_{\eta,\Delta}) = 2\int_0^\Delta \phi(v)(1 - e^{-\frac{\Delta^2-v^2}2})dv = 2\Phi(\Delta)-1-\tfrac{\Delta e^{-\Delta^2/2}}{\sqrt{\pi/2}} = F_{\chi^2_3}(\Delta^2),
\]
with the last equality following from the identity \eqref{eq:chisq}. 
\end{proof}

Our final lemma gives probability bounds for another type of circular sections that are closely related to $B_{\eta,\Delta}$ but are centered at the vertical axis. For any $r > 0$ define
\begin{equation}
\label{eq:bstar}
B^*_{\theta,r} = \{(x,y)\in \bbU: x^2 + (y-\theta)^2 \le \tfrac{(\theta \wedge 0)^2}{2} +  r^2\},~~\theta \in \bbR.
\end{equation}
Notice that $B^*_{0,r} = B_{0,r}$ and $B^*_{\theta,r} = B_{-\theta,r} + \theta \bs 1$ if $\theta  < 0$. When $\theta \ge \sqrt2 r$, $B^*_{\theta,r}$ is simply the disc of radius $r$ with center $(0,\theta)$. Let $K_{\theta,r}$ denote these discs. For $0 < \theta <  \sqrt 2 r$, we may write $B^*_{\theta,r} = K_{\theta,r} \setminus B'_{\theta,r}$, where $B'_{\theta,r}$ is the reflection of $B_{\theta,r}$ against the diagonal line $y = x$. 
\begin{lemma}
\label{lem:bstar}
$P_{0,\theta}(B^*_{\theta,r}) \ge F_{\chi^2_3}(r^2)$ for every $\theta \in \bbR$.
\end{lemma}

\begin{proof}
The above relations between $B^*$ and $B$ and Lemma 8 immediately give
\begin{itemize}
\item If $\theta = 0$, $P_{0,\theta}(B^*_{\theta,r}) = P_{0,0}(B_{0,r}) = 1 - e^{-r^2/2} > F_{\chi^2_3}(r^2)$. 
\item If $\theta < 0$, $P_{0,\theta}(B^*_{\theta,r}) = P_{0,\theta}(B_{-\theta,r} + \theta\bs 1) = P_{-\theta,0}(B_{-\theta,r}) > F_{\chi^2_3}(r^2)$.
\item If $\theta \ge \sqrt 2 r$, $P_{0,\theta}(B^*_{\theta,r}) = \frac{Q_{0,\theta}(K_{\theta,r})}{c(-\theta)} = \frac{1 - e^{-r^2/2}}{c(-\theta)} > 1 - e^{-r^2/2} > F_{\chi^2_3}(r^2)$ because $c(-\theta) < 1$. 
\item If $0 < \theta < \sqrt 2 r$, 
\begin{align*}
P_{0,\theta}(B^*_{\theta,r}) & = \frac{Q_{0,\theta}(B^*_{\theta,r})}{c(-\theta)} =  \frac{Q_{0,\theta}(K_{\theta,r}) - Q_{0,\theta}(B'_{\theta,r})}{c(-\theta)}\\
& = \frac{Q_{0,\theta}(K_{\theta,r}) - Q_{\theta,0}(B_{\theta,r})}{c(-\theta)} =  \frac{Q_{0,\theta}(K_{\theta,r}) - c(\theta)P_{\theta,0}(B_{\theta,r})}{c(-\theta)}\\
& > \{1 - e^{-r^2/2}\} \frac{1 - c(\theta)}{c(-\theta)} = 1 - e^{-r^2/2} > F_{\chi^2_3}(r^2).
\end{align*}
\end{itemize}

\end{proof}

\relax

\section{Proofs of main results}
\label{app:B}

\begin{proof}[Proof of Theorem \ref{thm:inf2}]
Let $C(X,Y)$ be an equivariant set procedure for $\theta$ with 100$(1-\alpha)$\% selection-specific confidence. Let $A_\theta = \{(x,y) \in \bbU: \theta \in C(x,y)\}$. By equivariance, $A_\theta = A_0 + (\theta,\theta)$, and hence, $1 - \alpha \le P_{\eta,\theta}(\{\theta \in C\}) = P_{\eta,\theta}(A_\theta) = P_{\eta,\theta}(A_0 + (\theta,\theta)) = P_{\eta-\theta,0}(A_0)$. Consequently, $P_{\eta,0}(A_0) \ge 1 - \alpha$ for every $\eta \in \bbR$. On the other hand, 
\[
P_{0,0}(|C|) = \int \int_\bbR 1(\theta\in C)d\theta dP_{0,0} = \int_\bbR P_{0,0}(A_{\theta})d\theta = \int_\bbR P_{\eta,\eta}(A_0)d\eta.
\]
By Lemma \ref{lem:biglemma}, there exist positive constants $\Delta$ and $M$ such that $P_{\eta,0}(B_{\eta,\Delta}) > \frac{1+\alpha}2$ for all $\eta \ge M$ where $B_{\eta,\Delta}$ is as in \eqref{eq:set-b}. Therefore $P_{\eta,0}(A_0 \cap B_{\eta,\Delta}) \ge P_{\eta,0}(A_0) + P_{\eta,0}(B_{\eta,\Delta}) - 1 \ge \frac{1-\alpha}2$ for all $\eta \ge M$. Because $p(x,y|\eta,0)$ is at most $e^{-\eta^2/4}/\{2\pi c(\eta)\}$ on $B_{\eta,\Delta}$, it follows that (with help from Lemma \ref{lem00})
\[
|A_0 \cap B_{\eta,\Delta}| \ge (1-\alpha)\pi c(\eta)e^{\eta^2/4} \ge k \eta^{-1}
\]
for all $\eta \ge M'$ for some positive constants $k$ and $M' \ge M$. Now, $B_{\eta,\Delta} \subset \{(x,y) \in \bbU: (x-\frac\eta2)^2 + (y - \frac\eta2)^2 \le \Delta^2\}$, and hence, 
\[
P_{\frac\eta2,\frac\eta2}(A_0 \cap B_{\eta,\Delta}) \ge |A_0 \cap B_{\eta,\Delta}| \min_{(x,y) \in B_{\eta,\Delta}} p(x,y|\tfrac\eta2,\tfrac\eta2) \ge |A_0 \cap B_{\eta,\Delta}| \frac{e^{-\Delta^2/2}}{\pi} \ge k' \eta^{-1}
\]
for some constant $k'$ for all $\eta \ge M'$. Consequently, $P_{0,0}(|C|) \ge \frac12  \int_\bbR P_{\frac\eta2,\frac\eta2}(A_0 \cap B_{\eta,\Delta}) d\eta \ge \int_{M'}^\infty k'\eta^{-1}d\eta = \infty$. Next, apply equivariance to see $P_{\eta,\eta}(|C|) = P_{0,0}(|C + (\eta,\eta)|) = P_{0,0}(|C|) = \infty$ for every $\eta \in \bbR$. 
\end{proof}

\begin{proof}[Proof of Theorem \ref{thm:finite}]
For $u \in (0,1)$, let $y_{\eta,\theta}(u)$ denote the $u$-th quantile of the marginal distribution of $Y$ under $(X,Y) \sim P_{\eta,\theta}$. For a given $\alpha \in (0,1)$ and a fixed $\eta \in \bbR$, the interval $I_{\eta,\theta_0} = [y_{\eta,\theta_0}(\alpha/2),y_{\eta,\theta_0}(1 - \alpha/2)]$ gives the acceptance region of a size-$\alpha$ test based on $Y$ alone for $H:\theta = \theta_0$ vs $K:\theta \ne \theta_0$. The oracle procedure $D$ in Theorem \ref{thm:finite} can be written as $D(Y,\eta) = \{\theta: Y \in I_{\eta,\theta}\}$ which simply inverts these tests to obtain a set procedure for $\theta$ -- based on the knowledge of $\eta$ -- with constant $100(1-\alpha)\%$ coverage. Without loss of generality we assume $\eta = 0$ and establish that 
\[
P_{0,\theta}(|D_0|) = \int_{\bbR} P_{0,\theta}(\bbR \times I_{0,\theta'})d\theta' < \infty
\]
for every $\theta \in \bbR$, where $D_0(Y) = D(Y,0)$. 
 
Consider the sets $B^*_{\theta,r}$ as in \eqref{eq:bstar} with $r > 0$ chosen large enough so that $F_{\chi^2_3}(r^2) \ge 1 - \alpha/2$. Lemma \ref{lem:bstar} says $P_{0,\theta}(B^*_{\theta,r}) \ge 1 - \alpha/2$ for every $\theta$. Consequently, it must be that $y_{0,\theta}(\tfrac\alpha2) \ge \inf\{y: (x,y) \in B^*_{\theta,r}\}$ and $y_{0,\theta}(1 - \tfrac\alpha2) \le \sup\{y: (x,y) \in B^*_{\theta,r}\}$. In other words, $I_{0,\theta} \subset [l(\theta),u(\theta)]$ where
\begin{align*}
l(\theta) & := \inf\{y: (x,y) \in B^*_{\theta,r}\} = 
\left\{\begin{array}{ll} 
\theta - r, &~\theta \ge r\\[2pt]
\frac\theta2 - \{\frac{r^2}{2} - \frac{\theta^2}{4}\}^{1/2}, &~0 \le \theta < r\\[5pt]
\frac\theta2 - \frac{r}{\sqrt 2}, &~ \theta < 0. 
\end{array}\right.\\
\mbox{and},~u(\theta) & := \sup\{y: (x,y) \in B^*_{\theta,r}\} = 
\left\{\begin{array}{ll} 
\theta + r, &~\theta \ge 0\\[2pt]
\theta + \{r^2 + \frac{\theta^2}{2}\}^{1/2}, &~-\sqrt 2 r \le \theta < 0\\[5pt]
\frac\theta2 + \frac{r}{\sqrt 2}, &~ \theta < -\sqrt 2 r. 
\end{array}\right.
\end{align*}
Therefore,
\begin{equation}
\label{eq:upbound}
P_{0,\theta}(\bbR \times I_{0,\theta'}) \le P_{0,\theta}(\bbR \times [l(\theta'),u(\theta')]) \le 1 - P_{0,\theta}(B^*_{\theta,\delta}) \le 1 - F_{\chi^2_3}(\delta^2)
\end{equation}
where $\delta = \delta(\theta';\theta)$ is the largest possible value such that $B^*_{\theta,r'}$ does not overlap with $\bbR \times [l(\theta'),u(\theta')]$. Since the interval $[l(\theta'),u(\theta')]$ always contains either $\theta$ or $\theta/2$ and has width no larger than $3r/\sqrt{2}$, there must exist positive numbers $K_1$ and $K_2$ such that $\delta(\theta';\theta) \ge K_1|\theta'|$ for all $|\theta'| > K_2$. Therefore,
\[
P_{0,\theta}(|D_0|) \le 2K_2+ \int_{|\theta'| > K_2} \{1 - F_{\chi^2_3}(K_1^2\theta'^2)\} d\theta' < \infty
\]
since $1 - F_{\chi^2_3}(K_1^2\theta'^2) \le K_1^3|\theta'|^3e^{-K_1^2\theta'^2/2}$ is integrable. 
\end{proof}

\section{Empirical Bayesian procedures}
\label{app:C}

%Here we detail the case where $\tau_1 = \cdots = \tau_p = \tau$. The heterogeneous variance case can be handled analogously. 
\subsection{Gaussian EB}
For the Gaussian empirical Bayesian method we assume a hierarchical model on $(\X,\bs\eta)$:
\[
X_j \sim N(\eta_j,\tau_j^2),~~\eta_j \sim N(m,v),~\text{independently across}~j = 1,\ldots,p.
\]
Estimation is carried out by adjusting for the selection event $\{\X \prec y\}$, where $y$ is the observed value of $Y$. Conditional on this adjustment,
% to the joint distribution of $(\X,\bs\eta)$. This density of this conditional joint distribution can be written as
%\begin{align*}
%p(\x,\bs\eta & |m,v,\X \prec y) \propto \prod_{j = 1}^p\cbr{ \frac1{\tau_j} \phi\pbr{\frac{x_j - \eta_j}{\tau_j}} \frac1{\sqrt v}\phi\pbr{\frac{\eta_j - m}{\sqrt{v}}} I(x_j < y)}.
%%& = \prod_{j = 1}^p \cbr{\frac1{\tau_j\sqrt{\rho_j}}\phi\pbr{\frac{\eta_j - (\rho_j x_i + (1-\rho_j)m)}{\tau_j\sqrt{\rho_j}}}}
%\end{align*}
$X_1,\ldots,X_p$ are marginally independent with $X_j \sim N(m, v + \tau_j^2)$ restricted to the interval $(-\infty, y)$. The posterior distribution of $\bs\eta$ given $\X = \x$ is unaffected by the adjustment $\{\X \prec y\}$, with $\eta_1,\ldots,\eta_p$ being independent and $\eta_j \sim N(\rho_j x_j + (1 - \rho_j)m, \rho_j\tau^2)$ where $\rho_j = \frac{v}{v + \tau_j^2}$. Integrating out $\bs\eta$, we see that $\X$ provides information about $m$ and $v$ through the marginal likelihood
\[
L(m,v) = \prod_{j = 1}^p \frac{f(x_j|m,\{v + \tau_j^2\}^{1/2})}{F(y | m,\{v + \tau_j^2\}^{1/2})}
\]
which could be maximized to obtain the so called type II maximum likelihood estimate of $(m,v)$. However, a straight optimization of $L(m,v)$ may produce an estimate $\hat v = 0$ and mild regularization of the marginal likelihood often results in improved estimation \citep{gelman2006prior}. In our case we take 
\begin{equation}
(\hat m, \hat v) = \arg\max_{m,v} L(m,v)\pi(v), \label{ml2}
\end{equation}
where $\pi(v) \propto (1 + v)^{-1}v^{-1/2}$ results from a half-Cauchy prior on $\sqrt v$. Given these estimates of $(m,v)$ the estimate $\hat{\bs\eta}$ is defined by the plug-in posterior means 
\[
\hat\eta_j = \hat \rho_j x_j + (1 - \hat \rho_j)\hat m
\] 
with $\hat \rho_j = \frac{\hat v}{\hat v + \tau_j^2}$. The optimization in \eqref{ml2} can be carried out numerically using standard Newton type methods. In our experiments we used the Broyden-Fletcher-Goldfarb-Shanno algorithm as implemented by the \textsf{optim} function in the \textsf{R} programming language.

% prior that $\eta_1,\ldots,\eta_p$ are independent draws from some $N(m,v)$ distribution. As noted earlier, an attractive plug-in estimate $\hat{\bs\eta} = (\hat\eta_1,\ldots,\hat\eta_p)$ is given by the posterior means 
%\[
%\hat\eta_j = \frac{v}{\tau_j^2 + v}  X_j + \frac{\tau_j^2}{\tau_j^2 + v} m.
%\] 
%In the empirical Bayes implementation, we further plug in estimates $(\hat m, \hat v)$ for $(m,v)$ into the above formula. 

\subsection{Nonparametric EB}
The nonparametric empirical Bayesian method assumes a hierarchical model on $(\X,\bs\eta)$ analogous to the one in the Gaussian case, with the difference that $\eta_1,\ldots,\eta_p$ are now taken to be independent draws from a density $g(\eta)$ which may not be Gaussian:
\[
X_j \sim N(\eta_j,\tau_j^2),~~\eta_j \sim g,~\text{independently across}~j = 1,\ldots,p.
\]
If $g$ were known, we could estimate each $\eta_j$ by the corresponding posterior mean
\[
\hat\eta_j = \frac{\int \eta f(x_j|\eta,\tau_j)g(\eta)}{\int f(x_j|\eta,\tau_j)g(\eta)d\eta}.
\]
With $g$ unknown, we will take $\hat\eta_j$ as above with a nonparametric estimate 
$\hat g$ plugged in for $g$. Below we discuss an estimation strategy which adjusts for the selection event $\{\X \prec y\}$ where $y$ is the observed value of $Y$. 

For this nonparametric method, we primarily focus on the homogeneous variance case where $\tau_1 = \cdots = \tau_p = \tau$. Adjusting for $\{\X \prec y\}$ we could rewrite the hierarchical model as
\[
X_j \sim \kappa(\cdot|\eta_j),~\eta_j \sim g^*(\eta)
\]
where $\kappa(x|\eta) = f(x|\eta,\tau)I(x < y)/F(y|\eta,\tau)$ is the density of $N(\eta,\tau^2)$ restricted to $(-\infty, y)$ and $g^*(\eta) \propto F(y|\eta,\tau)g(\eta)$. Notice that $\hat\eta_j$ can be rewritten as $\hat\eta_j = \int \eta \kappa(x_j|\eta)g^*(\eta)d\eta/\int \kappa(x_j|\eta)g^*(\eta)d\eta$.

We obtain a nonparametric estimate of the adjusted prior density $g^*(\eta)$ by using
the predictive recursion algorithm by M Newton \citep{newton2002nonparametric} as follows:
start with an initial estimate $g^*_0$, recursively update it by the equations 
\begin{equation}
g^*_j(\eta) = (1 - w_j) g^*_{j-1}(\eta) + w_j \frac{\kappa(x_j | \eta)g^*_{j-1}(\eta)}{\int \kappa(x_j | t)g^*_{j-1}(t)dt}, \quad j = 1,\ldots,p,
\label{recursion}
\end{equation} 
and return the estimate $\hat g^* = g^*_n$. Here
$w_1,w_2,\ldots \in (0,1)$ is a prespecified weight sequence. One typically
chooses the weights so that, {\it asymptotically},
$\sum_{j=1}^\infty w_j = \infty, \sum_{j = 1}^\infty w_j^2 < \infty$, to guarantee consistency
of $\hat g$ \citep{tokdar2009consistency}. Our numerical work uses $w_j = (1 + j)^{-2/3}$. We also repeat the recursion on 50 random permutations of the data and take the average of these 50
estimates as our final estimate of $g^*$. For every use of the recursion formula
\eqref{recursion}, the integral is carried out numerically using a Gaussian quadrature.

\end{document}